\documentclass[12pt,reqno]{amsart}

\usepackage[vmargin=1in,hmargin=1.2in]{geometry}
\usepackage{graphicx, latexsym, graphics}
\usepackage{amsfonts, amssymb, amsmath, amsthm, mathtools, mathrsfs}
\usepackage{algorithmic, algorithm}

\newcommand{\ds}{\displaystyle}
\newcommand{\NN}{\mathbb N}

\newcommand{\CC}{\mathbb C}
\newcommand{\RR}{\mathbb R}
\newcommand{\ZZ}{\mathbb Z}

\newcommand{\DD}{\mathcal D}
\newcommand{\SSS}{\mathcal S}

\newcommand{\VV}{\mathcal V}
\newcommand{\MM}{\mathcal M}

\theoremstyle{plain}
\newtheorem{theorem}{Theorem}[section]
\newtheorem{proposition}[theorem]{Proposition}
\newtheorem{lemma}[theorem]{Lemma}
\newtheorem{corollary}[theorem]{Corollary}

\theoremstyle{remark}
\newtheorem{remark}[theorem]{Remark}

\theoremstyle{definition}
\newtheorem{definition}[theorem]{Definition}
\newtheorem{example}[theorem]{Example}
\allowdisplaybreaks

\numberwithin{equation}{section}

\newcommand{\supp}{\operatorname{supp}}
\newcommand{\Conf}{\operatorname{Conf}}
\newcommand{\Sym}{\operatorname{Sym}}
\newcommand{\Sp}{\operatorname{Sp}}
\newcommand{\Mp}{\operatorname{Mp}}
\DeclareMathOperator*{\esssup}{ess\,sup}

\begin{document}

\author[S. Pilipovi\' c]{Stevan Pilipovi\' c}
\thanks{The work of S. Pilipovi\'c was partially supported by the Serbian Academy of Sciences and Arts, project, F10.}
\address{Department of Mathematics and Informatics,
University of Novi Sad, Trg Dositeja Obradovi\'{c}a 4, 21000 Novi Sad, Serbia}
\email{stevan.pilipovic@dmi.uns.ac.rs}

\author[B. Prangoski]{Bojan Prangoski}
\thanks{The work of B. Prangoski was partially supported by the bilateral project ``Microlocal analysis and applications'' funded by the Macedonian and Serbian academies of sciences and arts.}
\address{Department of Mathematics, Faculty of Mechanical
Engineering-Skopje, Ss. Cyril and Methodius University in Skopje, Karposh 2 b.b., 1000 Skopje, Macedonia}
\email{bprangoski@yahoo.com}

\title[Characterisation of the Weyl-H\"ormander classes]{Characterisation of the Weyl-H\"ormander classes by time-frequency shifts}



\frenchspacing

\begin{abstract}
We characterise the Weyl-H\"ormander symbol classes $S(M,g)$ via the growth of the action of the corresponding $\Psi$DOs on time-frequency shifts of a single test function. For this purpose, we introduce a geometric short-time Fourier transform which is well-suited for the analysis of $S(M,g)$. We define new modulation spaces and achieve the characterisation of the Weyl-H\"ormander classes by showing that they are intersections of such modulation spaces suitable for the time-frequency characterisation.
\end{abstract}
\maketitle

\section{Introduction}

In \cite{sjostrand1,sjostrand2}, Sj\"ostrand studied a class of pseudo-differential operators continuous on $L^2(\RR^n)$ which is closed under composition and taking inverses (when they exist). The corresponding space of Weyl symbols, nowadays known as the Sj\"ostrand algebra, is the modulation space \cite{F2,F3,Gr1}
$$
M^{\infty,1}(\RR^{2n})=\{a\in\SSS'(\RR^{2n})\,|\, V_{\varphi}a\in L^{\infty,1}(\RR^{4n})\}\footnote{Sj\"ostrand's original definition is different but amounts to the same space.}
$$
where $L^{\infty,1}(\RR^{4n})$ is the standard mixed-norm Lebesgue space and $V_{\varphi}$ is the short-time Fourier transform (from now, abbreviated as STFT) with window $\varphi\in\SSS(\RR^{2n})\backslash\{0\}$, i.e. $V_{\varphi}a(X,\Xi)=\langle a, e^{-2\pi i \Xi\, \cdot}\,\overline{\varphi(\cdot-X)}\rangle$, $X,\Xi\in\RR^{2n}$. Sj\"ostrand later generalised the theory in \cite{sjostrand3} to broader classes of symbols with limited regularity. These articles have sprung a line of research concerning $\Psi$DOs and FIOs with symbols in modulation and Wiener amalgam spaces and their application to PDEs; see \cite{ACT,ben-oko,bmntt,CN,CR,CGNR,CNT,fei-gn,Gr2,gro-rz,GT,kob-sug-t,T-1,0T} and the references therein. In \cite{sjostrand2}, Sj\"ostrand proved the following characterisation of $M^{\infty,1}(\RR^{2n})$ (see the proof of \cite[Proposition 3.1]{sjostrand2}). Let $a\in\SSS'(\RR^{2n})$ and let $\chi(x)=e^{-\pi |x|^2}$, $x\in\RR^n$. Then $a\in M^{\infty,1}(\RR^{2n})$ if and only
$$
\exists f\in L^1(\RR^{2n})\; \mbox{such that}\; |\langle a^w \pi(X)\chi,\overline{\pi(\Xi)\chi}\rangle|\leq f(\Xi-X),\quad X,\Xi\in\RR^{2n};
$$
here $\pi(X)$, $X\in\RR^{2n}$, stands for the time-frequency shift $\pi(X)\phi(y)=e^{2\pi i \xi y}\phi(y-x)$, $X=(x,\xi)\in\RR^{2n}$. To be precise, this result is not part of the claim of \cite[Proposition 3.1]{sjostrand2} but it is shown in its proof. The proof employs so called families of uniformly confined symbols introduced by Bony and Lerner \cite{bon-ler,bony1} and they will play a significant role in this article as well; discrete version of such types of families were previously employed by H\"ormander \cite{hormander}. Similar constructions were also considered by Feichtinger \cite{fei4} in the context of Wiener amalgam and decomposition spaces. Later, Gr\"ochenig and Rzeszotnik \cite{gro-rz} (see also \cite{Gr2}) employed time-frequency and Banach algebra techniques to generalise the results of \cite{sjostrand2}. As a byproduct of their analysis, they proved the following characterisation of the H\"ormander algebra $S^0_{0,0}(\RR^{2n})$ \cite[Theorem 6.2]{gro-rz}. Let $\chi\in\SSS(\RR^n)\backslash\{0\}$ be such that $\{\pi(\lambda)\chi\,|\, \lambda\in\Lambda\}$ is a tight Gabor frame for $L^2(\RR^n)$ on the lattice $\Lambda$ in $\RR^{2n}$. Then for $a\in\SSS'(\RR^{2n})$, the following conditions are equivalent:
\begin{itemize}
\item[$(i)$] $a\in \SSS^0_{0,0}(\RR^{2n})$;
\item[$(ii)$] for every $s\geq 0$ there is $C_s\geq 1$ such that $|\langle a^w \pi(X)\chi,\overline{\pi(\Xi)\chi}\rangle|\leq C_s(1+|\Xi-X|)^{-s}$, $X,\Xi\in\RR^{2n}$;
\item[$(iii)$] for every $s\geq 0$ there is $C_s\geq 1$ such that $|\langle a^w \pi(\lambda_1)\chi,\overline{\pi(\lambda_2)\chi}\rangle|\leq C_s(1+|\lambda_2-\lambda_1|)^{-s}$, $\lambda_1,\lambda_2\in\Lambda$.
\end{itemize}
The assumption that $\{\pi(\lambda)\chi\,|\, \lambda\in\Lambda\}$ is a tight Gabor frame for $L^2(\RR^n)$ is only needed for the equivalence with $(iii)$. The property $(iii)$ is usually called almost diagonalisation as it shows that the matrix of the pseudodifferential operator with respect to the Gabor frame is rapidly decaying outside of the diagonal; the property $(ii)$ can be viewed as a continuous version of this. Of course, the important part is that the symbol class $S^0_{0,0}(\RR^{2n})$ (resp., $M^{\infty,1}(\RR^{2n})$ in the above result of Sj\"ostrand) is characterised by this property; see \cite[Section 5.2]{CR}, \cite{CGNR,CNT,CGN} and the references therein for other such characterisations of modulation and Wiener amalgam spaces and almost diagonalisation of $\Psi$DOs and FIOs. It is natural to ask whether such characterisation can be obtained for the symbol classes of the frequently employed global calculi. In the main result of this article (Theorem \ref{main-theorem-dia}) we provide a characterisation in the spirit of $(ii)$ of the general Weyl-H\"ormander classes $S(M,g)$ \cite{hormander,horm3}. They include as special cases the symbol classes of almost all global calculi that appear in the literature, like the Shubin and SG-calculus \cite{NR,Shubin}, the Beals-Fefferman calculus \cite{bf1}, the H\"ormander $S_{\rho,\delta}$-calculus \cite{horm3,Shubin}, etc. Our main result covers as a special case the characterisation $(ii)$ of $S^0_{0,0}(\RR^{2n})$. A discrete characterisation in the spirit of $(iii)$ still remains as an open problem.\\
\indent The general idea in the proof of the main result of the article is to extend the time-frequency analysis to the Weyl-H\"ormander classes. However, the techniques from \cite{Gr2,gro-rz} are unlikely to work in other cases apart from the classes $S^m_{0,0}$. The reason for this is that the STFT is well-suited for analysing distributions when the underlying space carries the Euclidean metric. But, for the purposes of analysing the classes $S(M,g)$, the phases space is generally not Euclidean: it carries a metric $g$ that defines $S(M,g)$ which is never Euclidian except in the case of $S^m_{0,0}$. For this purpose, we generalise the STFT by replacing the window with a uniformly confined family of symbols; such families were introduced by Bony and Lerner in their seminal paper \cite{bon-ler} (cf. \cite{lernerB,unterberger}). These families keep track of the growth of the metric as the translates of the window in the classical STFT do for the Euclidian metric. Because of this, the generalised STFT which we introduce is better suited for the analysis of the general Weyl-H\"ormander classes. With its help, we define new modulation spaces. Our goal then is to show that $S(M,g)$ can be represented as an intersection of such modulation spaces and prove a characterisation of each of them in the same spirit as $(ii)$; this will give the desired characterisation of $S(M,g)$ as well.\\
\indent The paper is organised as follows. In Subsection \ref{properties} we collect several technical results concerning H\"ormander metrics. Sections \ref{C-family} and \ref{gms} are the core of the article. We start Section \ref{C-family} by recalling the notion of uniformly confined families of symbols and show several technical results about them. Then we define our generalisation of the STFT and study its mapping properties. Section \ref{gms} is devoted to the generalised modulation spaces and their properties. Section \ref{sec dia} contains the main result of the article: the characterisation of the Weyl-H\"ormander classes via the action of the corresponding Weyl operators on time-frequency shifts. However, these are no longer of the form $\pi(X)\chi$ but also carry encoded information about the metric. When $g$ is symplectic (a typical example is the metric that generates the classes $S^m_{\rho,\rho}(\RR^{2n})$, $\rho\in[0,1)$), the time-frequency shifts reduce to $\pi(X)$ but are intertwined with metaplectic operators that come from $g$ (Corollary \ref{cor-for-sym-metric-metplectope}). In the special case of the Euclidian metric, these metaplectic operators boil down to the identity operator. In Subsection \ref{ex111}, we specialise the main result to the frequently employed global calculi. At the very end, there are two appendices containing the proofs of several technical results stated in the previous sections.\\
\indent As a showcase, we present here the characterisation of $S^0_{\rho,\delta}(\RR^{2n})$, $0\leq \delta\leq \rho\leq 1$, $\delta<1$. Let $\theta_0\in\mathcal{C}^{\infty}([0,\infty))\backslash\{0\}$ be a non-negative, non-increasing function satisfying $\supp\theta_0\subseteq [0,1]$. Pick $r\in(0,1/2]$ and set
$$
\theta_{(x,\xi)}(y,\eta):=\theta_0(r^{-2}\langle \xi\rangle^{2\delta}|x-y|^2+r^{-2}\langle \xi\rangle^{-2\rho}|\xi-\eta|^2),\quad (x,\xi),(y,\eta)\in\RR^{2n}.
$$
For each $\xi\in \RR^n$, denote by $Q_{\xi}$ the diagonal $2n\times 2n$ matrix whose first $n$ entries along the diagonal are $\langle \xi\rangle^{\delta}$ while the second $n$ entries are $\langle\xi\rangle^{-\rho}$ and let $\Psi_{\xi}$ be the operator
$$
\Psi_{\xi}f(y,\eta):=f(Q_{\xi}^{-1}(y,\eta))=f(\langle \xi\rangle^{-\delta}y,\langle\xi\rangle^{\rho}\eta).
$$
Pick any $\chi\in\SSS(\RR^n)$ such that its Wigner transform is not $0$ at the origin (i.e., $W(\chi,\chi)(0)\neq 0$). Then, $a\in\SSS'(\RR^{2n})$ belongs to $S^0_{\rho,\delta}(\RR^{2n})$ if and only if for every $N>0$ there is $C_N>0$ such that
\begin{multline*}
\left|\left\langle \left(\Psi_{\frac{\xi+\eta}{2}} \left(a\theta_{\left(\frac{x+y}{2},\frac{\xi+\eta}{2}\right)}\right)\right)^w \pi\left(Q_{\frac{\xi+\eta}{2}}(x,\xi)\right)\chi, \overline{\pi\left(Q_{\frac{\xi+\eta}{2}}(y,\eta)\right)\chi}\right\rangle\right|\\
\leq C_N(1+\langle\xi+\eta\rangle^{\delta}|x-y|+\langle \xi+\eta\rangle^{-\rho}|\xi-\eta|)^{-N},\quad (x,\xi),(y,\eta)\in\RR^{2n}.
\end{multline*}
When $\rho=\delta\in[0,1)$, $\theta_{(x,\xi)}$ can be dropped from these bounds and they still characterise the elements of $S^0_{\rho,\rho}(\RR^{2n})$. Written in this form, when $\rho=0$ they reduce to the condition $(ii)$ we mentioned above. Furthermore, one can rewrite the latter as follows. The tempered distribution $a$ belongs to $S^0_{\rho,\rho}(\RR^{2n})$, $0\leq \rho<1$, if and only if for every $N>0$ there is $C_N>0$ such that
\begin{multline*}
\left|\left\langle a^w \pi(x,\xi)\Phi_{\frac{\xi+\eta}{2}}\chi, \overline{\pi(y,\eta)\Phi_{\frac{\xi+\eta}{2}}\chi}\right\rangle\right|\\
\leq C_N(1+\langle\xi+\eta\rangle^{\rho}|x-y|+\langle \xi+\eta\rangle^{-\rho}|\xi-\eta|)^{-N},\quad (x,\xi),(y,\eta)\in\RR^{2n},
\end{multline*}
where $\Phi_{\xi}$, $\xi\in\RR^n$, is the metaplectic operator $\Phi_{\xi}\chi(t)=\langle \xi\rangle^{\rho n/2}\chi(\langle \xi\rangle^{\rho}t)$.

\section{Preliminaries}

Given a smooth finite dimensional manifold\footnote{Manifolds are always assumed to be second-countable.} $\mathcal{N}$, one can unambiguously define the $\sigma$-algebra of Lebesgue measurable sets on $\mathcal{N}$ by declaring a set $E$ to be Lebesgue measurable if $\phi(E\cap U)$ is Lebesgue measurable for each chart $(\phi,U)$. The notion of a negligible set (nullset) in $\mathcal{N}$ is unambiguous since they are diffeomorphism invariant. If $F$ is a topological space, we say that $f:\mathcal{N}\rightarrow F$ is Lebesgue measurable (from now, always abbreviated as measurable) if $f$ is measurable with respect to the Borel $\sigma$-algebra on $F$ and the Lebesgue $\sigma$-algebra on $\mathcal{N}$, while we call $f$ Borel measurable if it is measurable with respect to the Borel $\sigma$-algebras on $\mathcal{N}$ and $F$.\\
\indent Let $V$ be an $n$-dimensional real vector space with $V'$ being its dual and set $W:=V\times V'$. We will always denote the points in $W$ with capital letters $X,Y,Z,\ldots$. Throughout the article, unless otherwise stated, the \underline{only} regularity assumption we impose on any Riemannian metric will be measurability\footnote{We emphasise this since, in the literature, it is customary to impose stronger regularity assumptions like continuity or smoothness.}; i.e. a Riemannian metric $g$ on $W$ is a Lebesgue measurable section of the 2-covariant tensor bundle $T^2T^*W$ that is symmetric and positive-definite at every point. We will always denote the corresponding quadratic forms by the same symbol: $g_X(T):=g_X(T,T)$, $T\in T_XW$. As standard, for each $X\in W$, we employ the canonical identification of $W$ with $T_XW$ that sends every $Y\in W$ to the directional derivative in direction $Y$ at $X$. For $T\in W\backslash\{0\}$, we denote by $\partial_T$ the vector field on $W$ given by the directional derivative in direction $T$ at every point $X\in W$. If $\partial_T$ acts on a function of two or more variables we will always emphasises in which variable it acts by writing $\partial_{T;X}, \partial_{T;Y},\ldots$. The space $W$ is symplectic when equipped with the canonical symplectic form $[(x,\xi),(y,\eta)]=\langle \xi,y\rangle -\langle \eta,x\rangle$, where $\langle \cdot,\cdot\rangle$ is the dual pairing for $V$ and $V'$. We denote by $\sigma:W\rightarrow W'$ the isomorphism induced by the symplectic form; notice that ${}^t\sigma=-\sigma$. Let $g$ be a Riemannian metric on $W$ and, for $X\in W$, denote by $Q_X:W\rightarrow W'$ the isomorphism induced by $g_X$. For $X\in W$, set $Q^{\sigma}_X:={}^t\sigma Q_X^{-1}\sigma:W\rightarrow W'$ and let $g^{\sigma}_X(T,S):=\langle Q^{\sigma}_XT,S\rangle$, $T,S\in W$. Then $g^{\sigma}$ is again a Riemannian metric on $W$ called the symplectic dual of $g$; it is also given by $g^{\sigma}_X(T)=\sup_{S\in W\backslash\{0\}} [T,S]^2/g_X(S)$. The Riemannian metric $g$ is said to be a H\"ormander metric \cite{hormander}, i.e., an admissible metric in the terminology of \cite{bony,lernerB}, if the following three conditions are satisfied:
\begin{itemize}
\item[$(i)$] slow variation: there exist $C_0\geq 1$ and $r_0>0$ such that
    $$
    g_X(X-Y)\leq r^2_0\Rightarrow C^{-1}_0g_Y(T)\leq g_X(T)\leq C_0g_Y(T),\quad  X,Y,T\in W;
    $$
\item[$(ii)$] temperance: there exist $C_0\geq 1$ and $N_0\geq 0$ such that
    $$
    \left(g_X(T)/g_Y(T)\right)^{\pm 1}\leq C_0(1+g^{\sigma}_X(X-Y))^{N_0},\quad  X,Y,T\in W;
    $$
\item[$(iii)$] the uncertainty principle: $g_X(T)\leq g^{\sigma}_X(T),\; X,T\in W$.
\end{itemize}
The constants $C_0\geq 1$, $r_0>0$ and $N_0\geq 0$ for which all of the above conditions hold true will be called \textit{structure constants} for $g$. The H\"ormander metric $g$ is said to be symplectic if $g=g^{\sigma}$.\\
\indent A positive measurable function $M$ on $W$ is said to be $g$-admissible if there are $C\geq 1$, $r>0$ and $N\geq 0$ such that
\begin{gather*}
g_X(X-Y)\leq r^2\Rightarrow C^{-1}M(Y)\leq M(X)\leq CM(Y),\quad  X,Y\in W;\\
\left(M(X)/M(Y)\right)^{\pm1}\leq C(1+g^{\sigma}_X(X-Y))^N,\quad  X,Y\in W.
\end{gather*}
The constants $r>0$ and $N\geq 0$ for which these conditions hold true are called \textit{admissibility constants} for $M$. We will also call $r$ a \textit{slow variation constant} and $N$ a \textit{temperance constant} for $M$.\\
\indent Given a $g$-admissible weight $M$, the space of symbols $S(M,g)$ is defined as the space of all $a\in\mathcal{C}^{\infty}(W)$ for which
$$
\|a\|^{(k)}_{S(M,g)}=\sup_{l\leq k}\sup_{\substack{X\in W\\ T_1,\ldots, T_l\in W\backslash\{0\}}}\frac{|a^{(l)}(X;T_1,\ldots,T_l)|} {M(X)\prod_{j=1}^lg_X(T_j)^{1/2}}<\infty,\quad k\in\NN.
$$
With this system of seminorms, $S(M,g)$ becomes a Fr\'echet space. One can always regularise the metric making it smooth without changing the notion of $g$-admissibility of a weight and the space $S(M,g)$; the same can be done for any $g$-admissible weight (see \cite{hormander}). We point out that $S(M,g)\subseteq \mathcal{O}_{\mathcal{M}}(W)$, where $\mathcal{O}_{\mathcal{M}}(W)$ is the space of multipliers for $\SSS(W)$.\\
\indent Given $a\in\SSS(W)$, the Weyl quantisation $a^w$ is the operator
$$
a^w\varphi(x)=\int_{V'}\int_V e^{2\pi i\langle x-y,\xi\rangle}a((x+y)/2,\xi)\varphi(y)dyd\xi,\quad \varphi\in \SSS(V),
$$
where $dy$ is a left-right Haar measure on $V$ with $d\xi$ being its dual measure on $V'$ so that the Fourier inversion formula holds with the standard constants. Consequently, $a^w$ as well as the completion of the product measure $dyd\xi$ on $W$ are unambiguously defined. We denote this Haar measure by $d\lambda$ and employ it to include test functions in spaces of distributions on $W$; we will also denote it by $dX,dY,dZ,\ldots$ when it appears in integrals. We point out that $d\lambda$ is exactly the measure induced by the symplectic volume form. To avoid working with distribution densities, we fix a left-right Haar measure on $V$ with its dual measure on $V'$ and we employ these to include test functions into the spaces of distributions over $V$ and $V'$. Then, for any $a\in\SSS(W)$, $a^w$ extends to a continuous operator from $\SSS'(V)$ into $\SSS(V)$. The definition of the Weyl quantisation extends to symbols in $\SSS'(W)$ and in this case $a^w:\SSS(V)\rightarrow \SSS'(V)$ is continuous. When $a\in S(M,g)$ for a $g$-admissible weight $M$, $a^w$ is continuous as an operator on $\SSS(V)$ and it uniquely extends to a continuous operator on $\SSS'(V)$ (see \cite{hormander}). Furthermore, if $a\in S(M_1,g)$ and $b\in S(M_2,g)$, then $a^wb^w=(a\#b)^w$, where $a\# b\in S(M_1M_2,g)$ and the bilinear map $\#: S(M_1,g)\times S(M_2,g)\rightarrow S(M_1M_2,g)$ is continuous (see \cite{hormander}).\\
\indent If $F_1$ and $F_2$ are two locally convex spaces (from now, always abbreviated as l.c.s.), we denote by $\mathcal{L}(F_1,F_2)$ the space of continuous linear mappings from $F_1$ into $F_2$, while $\mathcal{L}_b(F_1,F_2)$ stands for this space equipped with the topology of uniform convergence on all bounded sets. When $F_1=F_2=F$, we abbreviate these notations and simply use $\mathcal{L}(F)$ and $\mathcal{L}_b(F),$ respectively.\\
\indent Given a Fr\'echet space $F$, recall that a map $\mathbf{f}:W\rightarrow F$ is said to be strongly measurable if there exists a sequence of Lebesgue measurable simple functions on $W$ with values in $F$ which converges pointwise a.e. to $\mathbf{f}$. The map $\mathbf{f}:W\rightarrow F$ is said to be weakly measurable if for each $f'\in F'$, the function $W\rightarrow \CC$, $X\mapsto \langle f', \mathbf{f}(X)\rangle$, is measurable.

\begin{remark}\label{rem-for-meas-strong-weak}
If $F$ is a separable Fr\'echet space, the mapping $\mathbf{f}:W\rightarrow F$ is strongly measurable if and only if it is weakly measurable. This is also equivalent to the requirement that the preimage of every open set is measurable (i.e. $\mathbf{f}$ is measurable). Furthermore, there exists a sequence of Lebesgue measurable simple functions which converges pointwise everywhere to $\mathbf{f}$. These facts follow from \cite[Theorem 1]{thomas} and the remark following it since every separable Fr\'echet space is a Polish space.
\end{remark}

\indent Given a Riemannian metric $g$ on $W$, we denote by $|g_X|$, $X\in W$, the determinant of $g_X$ computed in a symplectic basis. We point out that $|g_X|$ is the same in any symplectic basis of $W$; we will always tacitly apply this fact throughout the rest of the article. To verify it, let $E_j$, $j=1,\ldots,2n$, and $\tilde{E}_j$, $j=1,\ldots,2n$, be two symplectic bases. Let $A,\tilde{A}:W\rightarrow W'$ be the isomorphisms that send these bases to their respective dual bases and let $P:W\rightarrow W$ be the symplectic map given by $P(E_j)=\tilde{E}_j$, $j=1,\ldots,2n$. Since $\det(P)=1$ and $\det(\tilde{A}^{-1}({}^tP)^{-1}\tilde{A})=1$, the claim follows from
\begin{align*}
\det((g_X(E_j,E_k))_{j,k})&=\det(A^{-1}Q_X)=\det(P^{-1}\tilde{A}^{-1}({}^tP)^{-1}\tilde{A} \tilde{A}^{-1}Q_X)\\
&=\det(\tilde{A}^{-1}Q_X) =\det((g_X(\tilde{E}_j,\tilde{E}_k))_{j,k}).
\end{align*}
From now, we denote by $|g|$ the measurable function $W\rightarrow (0,\infty)$, $X\mapsto |g_X|$. If $g$ is a H\"ormander metric, we denote by $dv_g$ the measure on $W$ induced by the volume form of $g$. Because of temperance, $dv_g$ is a complete regular Borel measure defined on the Lebesgue $\sigma$-algebra of $W$ which takes finite values on compact sets and its nullsets are exactly the Lebesgue nullsets in $W$ (see the proof of \cite[Theorem 3.11, p. 59]{grigoryan}). Notice that $dv_g=|g|^{1/2}d\lambda$.\\
\indent Let $g$ be a H\"ormander metric on $W$ and let $w:W\times W\rightarrow (0,\infty)$ be a measurable function such that both $w$ and $1/w$ are polynomially bounded, i.e. there are $C>0$ and $\tau_1,\tau_2\geq 0$ such that
\begin{align}\label{pol-b-d-wei-s}
C^{-1}(1+|X|+|Y|)^{-\tau_1}\leq w(X,Y)\leq C(1+|X|+|Y|)^{\tau_2},\quad X,Y\in W,
\end{align}
where $|\cdot|$ is a norm on $W$; notice that if $w$ satisfies \eqref{pol-b-d-wei-s} with one norm on $W$ then it satisfies it with any other norm with the same $\tau_1$ and $\tau_2$ and possibly different $C>0$. For $1\leq p\leq \infty$, we denote by $L^p_w(W\times W,dv_gd\lambda)$ the Banach space of all measurable functions $f$ on $W\times W$ such that $fw\in L^p(W\times W,dv_gd\lambda)$; when $g$ is a Euclidian metric, we will simply denote it by $L^p_w(W\times W)$. Clearly, $\SSS(W\times W)\subseteq L^p_w(W\times W,dv_gd\lambda)\subseteq \SSS'(W\times W)$ and the inclusions are continuous; when $p<\infty$, they are also dense. We point out that $L^{\infty}_w(W\times W,dv_gd\lambda)=L^{\infty}_w(W\times W)$. If $w$ is a positive measurable function on $W$ that satisfies bounds of the form \eqref{pol-b-d-wei-s} on $W$, the weighted space $L^p_w(W)$, $1\leq p\leq \infty$, is defined analogously.

\subsection{Several technical results about H\"ormander metrics}\label{properties}

Let $g$ be a H\"ormander metric with structure constants $C_0\geq 1$, $r_0>0$ and $N_0\geq 0$. Given $X\in W$ and $r>0$, denote $U_{X,r}:=\{Y\in W\,|\, g_X(X-Y)\leq r^2\}$. Throughout the article, the functions
\begin{equation}\label{funco-for-met-meas-cont-whe-gisco}
(X,Y)\mapsto g^{\sigma}_X(Y-U_{X,r})\quad \mbox{and}\quad (X,Y)\mapsto g^{\sigma}_X(U_{Y,r}-U_{X,r})
\end{equation}
will often appear in integrals in $X$ or in $Y$. We point out that these functions are always measurable and, when $g$ is continuous, \eqref{funco-for-met-meas-cont-whe-gisco} are also continuous. The proof of these facts can be found in Appendix \ref{appendix1-proof-subwithhor-metfa}: see Lemma \ref{lemma-meas-cont-metr-obtf} and the remarks following it. Here we collect a number of inequalities which we will frequently employ throughout the rest of the article. Some of these can be found in \cite{bon-che,bon-ler,hormander,lernerB}; for the sake of completeness, we give proofs for all of them in Appendix \ref{appendix1-proof-subwithhor-metfa}.

\begin{lemma}\label{tec-res-ine-for-hor-metr-ini}
Let $0<r\leq r_0$. The following inequalities hold true:
\begin{gather}
(g_X(T)/g_Y(T))^{\pm 1}\leq C_0^{N_0+2}(1+g^{\sigma}_X(Y-U_{X,r}))^{N_0},\;\; X,Y\in W,\, T\in W\backslash\{0\};\label{ineq-for-metric-p-1}\\
1+g_Y(X-Y)\leq 2(1+r^2)C_0^{N_0+2}(1+g^{\sigma}_X(Y-U_{X,r}))^{N_0+1},\;\; X,Y\in W;\label{ineq-for-metric-p-2}\\
g^{\sigma}_Y(Y-U_{X,r})\leq C_0^{N_0+2}g^{\sigma}_X(Y-U_{X,r})(1+g^{\sigma}_X(Y-U_{X,r}))^{N_0},\;\; X,Y\in W;\label{ineq-for-metric-p-3}\\
\sup_{Y\in W}\int_W (1+g^{\sigma}_X(Y-U_{X,r}))^{-(N_0+1)(n+1)-nN_0} dv_g(X)<\infty.\label{ineq-for-metric-p-3-1}
\end{gather}
Fixing any symplectic basis on $W$ to evaluate $|g_X|$ and $|g^{\sigma}_X|$, the following hold true:
\begin{gather}
|g_X|\leq 1\leq |g^{\sigma}_X|,\quad |g_X||g^{\sigma}_X|=1,\quad X\in W;\label{ineq-for-metric-p-4}\\
g_X(X-Y)\leq r_0^2\Rightarrow (|g_X|/|g_Y|)^{\pm 1}\leq C_0^{2n}\;\; \mbox{and}\;\; (|g^{\sigma}_X|/|g^{\sigma}_Y|)^{\pm 1}\leq C_0^{2n};\label{ineq-for-metric-p-5}\\
\max\{(|g_X|/|g_Y|)^{\pm 1},(|g^{\sigma}_X|/|g^{\sigma}_Y|)^{\pm 1}\}\leq C_0^{2nN_0+4n}(1+g^{\sigma}_X(Y-U_{X,r}))^{2nN_0},\;\; X,Y\in W.\label{ineq-for-metric-p-6}
\end{gather}
Let $|\cdot|$ be a norm on $W$. There exists $C\geq 1$ such that
\begin{equation}\label{ineq-for-metric-p-7}
C^{-1}(1+|X|)^{-4nN_0}\leq |g_X|\leq |g^{\sigma}_X|\leq C(1+|X|)^{4nN_0},\quad X\in W.
\end{equation}
\end{lemma}

\begin{remark}\label{equ-for-sym-metr-meas-lebesmes}
When $g$ is symplectic, \eqref{ineq-for-metric-p-4} implies $|g_X|=1$, $X\in W$, and hence $dv_g=d\lambda$.
\end{remark}

\section{A generalisation of the short-time Fourier transform}\label{C-family}

In this section, we introduce the main tool that we need for the proof of the main result: a geometric version of the symplectic short-time Fourier transform.\\
\indent From now on, $g$ is a fixed H\"ormander metric on $W$ with structure constants $C_0\geq 1$, $r_0>0$ and $N_0\geq 0$; as before, for each $X\in W$, $Q_X,Q^{\sigma}_X:W\rightarrow W'$ are the isomorphisms induced by $g_X$ and $g^{\sigma}_X$.

\subsection{Essentially uniformly confined families of symbols}

We introduce the key ingredient for the geometric generalisation of the symplectic short-time Fourier transform. It consists of families of Schwartz functions which ``keep track of the growth of the metric''; they are going to be the analogue of the windows in the classical STFT.\\
\indent We start with two technical results. Their proofs are straightforward and we omit them (cf. Remark \ref{rem-for-meas-strong-weak}).

\begin{lemma}\label{lemma-regularity-sta}
Let $A:W\rightarrow \operatorname{GL}(W)$ and $f:W\rightarrow W$ be measurable and let $\psi\in\SSS(W)$. For each $X\in W$, define $\tilde{\psi}_X(Y):=\psi(A(X)Y+f(X))$, $Y\in W$. Then $W\rightarrow \SSS(W)$, $X\mapsto \tilde{\psi}_X$, is well-defined and strongly measurable.
\end{lemma}

\begin{lemma}\label{lemma-regularity-sta1}
Let $\tilde{g}$ be a Riemannian metric on $W$ and let $\chi_0\in\SSS(\RR)$. For each $X\in W$, define $\tilde{\psi}_X(Y):=\chi_0(\tilde{g}_X(X-Y,X-Y))$, $Y\in W$. Then $W\rightarrow \SSS(W)$, $X\mapsto \tilde{\psi}_X$, is well-defined and strongly measurable.
\end{lemma}

Let $X\in W$ and $r>0$ be arbitrary but fixed. Following Bony and Lerner \cite[Definition 2.1.1]{bon-ler} (cf. \cite[Definition 2.3.1, p. 84]{lernerB}), we say that $\varphi\in\mathcal{C}^{\infty}(W)$ is $g_X$-confined in $U_{X,r}$ if\footnote{Here and throughout the rest of the article we employ the principle of vacuous (empty) product for numbers and operators, i.e. $\prod_{j=1}^0r_j= \prod_{j\in\emptyset}r_j=1$ and $\prod_{j=1}^0A_j= \prod_{j\in\emptyset}A_j=\operatorname{Id}$.}
$$
\|\varphi\|^{(k)}_{g_X,U_{X,r}}:=\sup_{l\leq k} \sup_{\substack{Y\in W\\ T_1,\ldots,T_l\in W\backslash\{0\}}} \frac{|\varphi^{(l)}(Y;T_1,\ldots,T_l)| (1+g^{\sigma}_X(Y-U_{X,r}))^{k/2}}{\prod_{j=1}^l g_X(T_j)^{1/2}}<\infty,\; k\in\NN.
$$
For fixed $X$ and $r$, the space of $g_X$-confined symbols in $U_{X,r}$ coincides with $\SSS(W)$ and the above norms generate the topology of $\SSS(W)$. The authors of \cite{bon-che,bon-ler,lernerB} considered families of functions $\{\varphi_X\in\SSS(W)\, |\, X\in W\}$ which are uniformly $g_X$-confined in $U_{X,r}$, i.e. $\sup_{X\in W} \|\varphi_X\|^{(k)}_{g_X,U_{X,r}}<\infty$, $k\in\NN$, since these are a convenient tool for studying the symbol classes $S(M,g)$ and the corresponding pseudo-differential operators (see also \cite{pil-pra1}). However, we will be interested in such families which are essentially uniformly confined.

\begin{definition}
For $r\in(0,r_0]$, we denote by $\Conf_g(W;r)$ the vector space of all equivalence classes of strongly measurable maps $\boldsymbol{\varphi}:W\rightarrow \SSS(W)$ which satisfy the following condition:
\begin{equation}\label{family-seminorms-confsym}
\|\boldsymbol{\varphi}\|^{(k)}_{g,r}:=\esssup_{X\in W}\|\boldsymbol{\varphi}(X)\|^{(k)}_{g_X,U_{X,r}}<\infty,\quad k\in\NN.
\end{equation}
\end{definition}

The right-hand side of \eqref{family-seminorms-confsym} makes sense since the function $W\rightarrow [0,\infty)$, $X\mapsto \|\boldsymbol{\varphi}(X)\|^{(k)}_{g_X,U_{X,r}}$, is measurable. To see this, set $\varphi_X:=\boldsymbol{\varphi}(X)\in \SSS(W)$, $X\in W$, and let $l,k\in\NN$ with $l\leq k$. The required measurability follows from the fact that the function
\begin{equation}\label{fun-in-norm-for-conf-sym}
W\times W\times (W\backslash\{0\})^l\rightarrow\CC,\; (X,Y,T_1,\ldots,T_l)\mapsto \frac{\varphi_X^{(l)}(Y;T_1,\ldots,T_l) (1+g^{\sigma}_X(Y-U_{X,r}))^{k/2}}{\prod_{j=1}^l g_X(T_j)^{1/2}},
\end{equation}
is measurable (see Remark \ref{rem-for-measura-cont-when-met-ssh}) and, for each fixed $X\in W$, this is a continuous function of $(Y,T_1,\ldots, T_l)\in W\times (W\backslash\{0\})^l$ (so the supremum over $W\times (W\backslash\{0\})^l$ in the definition of $\|\boldsymbol{\varphi}(X)\|^{(k)}_{g_X,U_{X,r}}$ can be replaced with essential supremum).

\begin{remark}\label{remark-for-equi-for-strong-weak-borel-meas-for-s}
In view of Remark \ref{rem-for-meas-strong-weak}, $\boldsymbol{\varphi}:W\rightarrow \SSS(W)$ is strongly measurable if and only if it is weakly measurable. We will frequently tacitly apply this fact throughout the rest of the article.
\end{remark}

It is straightforward to verify that with the system of seminorms \eqref{family-seminorms-confsym}, $\Conf_g(W;r)$ becomes a Fr\'echet space. We say that $\boldsymbol{\varphi}\in \Conf_g(W;r)$ is of class $\mathcal{C}^k$, $0\leq k\leq\infty$, if $\boldsymbol{\varphi}\in\mathcal{C}^k(W;\SSS(W))$; when $k=0$ and $k=\infty$ we also say that $\boldsymbol{\varphi}$ is continuous and smooth, respectively.

\begin{remark}
If the metric $g$ is continuous and $\boldsymbol{\varphi}\in\Conf_g(W;r)$ is continuous, then \eqref{fun-in-norm-for-conf-sym} is also continuous (see Remark \ref{rem-for-measura-cont-when-met-ssh}). Hence, the essential supremum in \eqref{family-seminorms-confsym} can be replaced by supremum.
\end{remark}

The following class of elements in $\Conf_g(W;r)$ will play an important role throughout the rest of the article.

\begin{definition}
We say that $\boldsymbol{\varphi}\in\Conf_g(W;r)$ is \textit{non-degenerate} if
\begin{equation}\label{non-degen-ele-cond-int-cottt}
\inf_{Y\in W}\int_W |\varphi_X(Y)|^2dv_g(X)>0,\quad \mbox{where}\;\; \varphi_X:=\boldsymbol{\varphi}(X),\; X\in W.
\end{equation}
\end{definition}

We collect the properties of the elements of $\Conf_g(W;r)$ that we need in the following lemma.

\begin{lemma}\label{rem-about-part-of-unity}
${}$
\begin{itemize}
\item[$(i)$] Let $0<r_1\leq r_2\leq r_0$. Then $\Conf_g(W;r_1)\subseteq \Conf_g(W;r_2)$ and the inclusion is continuous.
\item [$(ii)$] Let $\boldsymbol{\varphi},\boldsymbol{\psi}\in\Conf_g(W;r)$ and $a\in S(1,g)$. For each $X\in W$, set
$$
(\boldsymbol{\varphi}\boldsymbol{\psi})(X):=\boldsymbol{\varphi}(X)\boldsymbol{\psi}(X), \quad (a\boldsymbol{\varphi})(X):=a\boldsymbol{\varphi}(X),\quad \overline{\boldsymbol{\varphi}}(X):=\overline{\boldsymbol{\varphi}(X)}.
$$
Then $\boldsymbol{\varphi}\boldsymbol{\psi}, a\boldsymbol{\varphi},\overline{\boldsymbol{\varphi}}\in\Conf_g(W;r)$. With the multiplications defined above, $\Conf_g(W;r)$ is a Fr\'echet algebra and a topological module over the Fr\'echet algebra $S(1,g)$ where the latter has the ordinary pointwise multiplication\footnote{We emphasis this to avoid confusion since $S(1,g)$ is also a Fr\'echet algebra with multiplication given by the sharp product $\#$.}.
\item[$(iii)$] For each $\boldsymbol{\varphi}\in\Conf_g(W;r)$, set
\begin{equation}\label{func-for-com-famil-bou}
I_{\boldsymbol{\varphi}}(Y):=\int_W \varphi_X(Y) dv_g(X),\;\; Y\in W,\quad \mbox{with}\;\; \varphi_X:=\boldsymbol{\varphi}(X),\; X\in W.
\end{equation}
Then the mapping $\Conf_g(W;r)\rightarrow S(1,g)$, $\boldsymbol{\varphi}\mapsto I_{\boldsymbol{\varphi}}$, is well-defined and continuous. If $\boldsymbol{\varphi}$ is non-degenerate, then $1/I_{|\boldsymbol{\varphi}|^2}\in S(1,g)$, where $|\boldsymbol{\varphi}|^2:=\boldsymbol{\varphi}\overline{\boldsymbol{\varphi}}$.
\end{itemize}
\end{lemma}

\begin{proof} The proof of $(i)$ and $(ii)$ is straightforward; for the proof concerning $a\boldsymbol{\varphi}$ one employs \eqref{ineq-for-metric-p-1}. The fact that $\Conf_g(W;r)\rightarrow S(1,g)$, $\boldsymbol{\varphi}\mapsto I_{\boldsymbol{\varphi}}$, is well-defined and continuous follows from \eqref{ineq-for-metric-p-1} and \eqref{ineq-for-metric-p-3-1}. When $\boldsymbol{\varphi}$ is non-degenerate, \cite[Lemma 2.4]{hormander} implies that $1/I_{|\boldsymbol{\varphi}|^2}\in S(1,g)$.
\end{proof}

\begin{example}\label{exi-of-good-par-off}
We give three examples of non-degenerate elements of $\Conf_g(W;r)$ which we are going to use on several occasions throughout the rest of the article.\\
\\
\noindent $(i)$ The following construction is due to Bony and Lerner \cite[Theorem 3.1.3]{bon-ler} (see also \cite[Theorem 2.2.7, p. 70]{lernerB}); we outline the main ideas and refer to the proofs of \cite[Theorem 2.2.7, p. 70]{lernerB} and \cite[Theorem 3.1.3]{bon-ler} for the details. Pick non-negative and non-increasing $\chi_0\in\mathcal{C}^{\infty}([0,\infty))$ such that $\chi_0=1$ on $[0,1/2]$ and $\supp\chi_0\subseteq [0,1]$. Let $0<r\leq r_0$ and set $\tilde{\psi}_X:=\chi_0(r^{-2}g_X(X-\cdot))$, $X\in W$. Then $X\mapsto \tilde{\boldsymbol{\psi}}(X):=\tilde{\psi}_X$ belongs to $\Conf_g(W;r)$ (cf. Lemma \ref{lemma-regularity-sta1}), $\inf_{Y\in W} I_{\tilde{\boldsymbol{\psi}}}(Y)>0$ and $1/I_{\tilde{\boldsymbol{\psi}}}\in S(1,g)$; see the proof of \cite[Theorem 2.2.7, p. 70]{lernerB}. Define $\boldsymbol{\psi}:=\tilde{\boldsymbol{\psi}}/I_{\tilde{\boldsymbol{\psi}}}$. Then $\supp\boldsymbol{\psi}(X)\subseteq U_{X,r}$, for all $X\in W$, and $\boldsymbol{\psi}$ is non-degenerate. The latter holds true since, for each $X\in W$, $\psi_X:=\boldsymbol{\psi}(X)$ satisfies the bounds
$$
1/\|I_{\tilde{\boldsymbol{\psi}}}\|_{L^{\infty}(W)}\leq \psi_X(Y)\leq \|1/I_{\tilde{\boldsymbol{\psi}}}\|_{L^{\infty}(W)},\quad Y\in U_{X,r/2}.
$$
Furthermore, $I_{\boldsymbol{\psi}}(Y)=\int_W \psi_X(Y) dv_g(X)=1$, $Y\in W$.\\
\\
\noindent $(ii)$ We slightly modify the above construction so that the resulting element of $\Conf_g(W;r)$ is smooth. Let $0<r\leq r_0$. First we regularise the metric with the help of $\boldsymbol{\psi}$ constructed above for any fixed $r'\in(0,r_0]$. Set $\psi_X:=\boldsymbol{\psi}(X)$, $X\in W$, and define
$$
\tilde{\tilde{g}}_X(T,S):=\int_W g_Y(T,S)\psi_Y(X)dv_g(Y),\quad  X,T,S\in W.
$$
Then $\tilde{\tilde{g}}$ is a smooth Riemannian metric on $W$ and, setting
$\tilde{\tilde{g}}_X(T):=\tilde{\tilde{g}}_X(T,T),$ there exists $C'\geq 1$ such that $C'^{-1}\tilde{\tilde{g}}_X(T)\leq g_X(T)\leq C'\tilde{\tilde{g}}_X(T)$, $X,T\in W$;  see \cite[Remark 2.2.8, p. 71]{lernerB}. Let $\tilde{g}_X(T,S):=C'^{-1}\tilde{\tilde{g}}_X(T,S)$, $X,T,S\in W$. Then $\tilde{g}$ is a smooth H\"ormander metric that satisfies $\tilde{g}_X\leq g_X\leq C'^2\tilde{g}_X$. Notice that the constant $N_0$ for $g$ is the same as the corresponding constant for $\tilde{g}$ and, for the constant $\tilde{r}_0$ from the slow variation of $\tilde{g}$, we can take $\tilde{r}_0:=r_0/C'$. Let $\chi_0$ be as in $(i)$ and set $\tilde{\varphi}_X:=\chi_0(C'^2r^{-2}\tilde{g}_X(X-\cdot))$, $X\in W$. Then, similarly as above, $X\mapsto \tilde{\boldsymbol{\varphi}}(X):=\tilde{\varphi}_X$ belongs to $\Conf_g(W;r)\cap \mathcal{C}^{\infty}(W;\DD(W))$, $\inf_{Y\in W} I_{\tilde{\boldsymbol{\varphi}}}(Y)>0$ and $1/I_{\tilde{\boldsymbol{\varphi}}}\in S(1,g)$. Define $\boldsymbol{\varphi}:=\tilde{\boldsymbol{\varphi}}/I_{\tilde{\boldsymbol{\varphi}}}$. Then $\boldsymbol{\varphi}$ satisfies the following properties:
\begin{itemize}
\item[$(a)$] $\boldsymbol{\varphi}(X)$ is non-negative and $\supp\boldsymbol{\varphi}(X)\subseteq U_{X,r}$,  $X\in W$;
\item[$(b)$] $\boldsymbol{\varphi}\in\Conf_g(W;r)\cap \mathcal{C}^{\infty}(W;\DD(W))$ and $\boldsymbol{\varphi}$ is non-degenerate;
\item[$(c)$] setting $\varphi_X:=\boldsymbol{\varphi}(X)$, $X\in W$, it holds that $I_{\boldsymbol{\varphi}}(Y)=\int_W \varphi_X(Y) dv_g(X)=1$, $Y\in W$;
\item[$(d)$] for all $k,l\in\NN$ and $N>0$ it holds that
\begin{equation*}
\sup_{\substack{k'\leq k\\ l'\leq l}} \sup_{X,Y\in W} \sup_{\substack{T_1,\ldots, T_{k'}\in W\backslash\{0\}\\ S_1,\ldots, S_{l'}\in W\backslash\{0\}}} \frac{\left|\left(\prod_{j=1}^{k'}\partial_{T_j;X}\right) \left(\prod_{j=1}^{l'}\partial_{S_j;Y}\right) \varphi_X(Y)\right|(1+g^{\sigma}_X(Y-U_{X,r}))^N}{\left(\prod_{j=1}^{k'} g_X(T_j)^{1/2}\right)\left(\prod_{j=1}^{l'} g_X(S_j)^{1/2}\right)}<\infty.
\end{equation*}
\end{itemize}
The fact that $\boldsymbol{\varphi}$ is non-degenerate follows from the bounds
\begin{equation}\label{est-from-belo-and-above-part-unit}
1/\|I_{\tilde{\boldsymbol{\varphi}}}\|_{L^{\infty}(W)}\leq \varphi_X(Y)\leq \|1/I_{\tilde{\boldsymbol{\varphi}}}\|_{L^{\infty}(W)},\quad Y\in U_{X,r/(2C')},
\end{equation}
and the proof of the rest of the properties in $(a)$, $(b)$ and $(c)$ is straightforward. Since the proof of $(d)$ is rather lengthy and technical, it is moved to Appendix \ref{appendix2-proof-dexm}.\\
\\
\noindent $(iii)$ Let $g$ be the Euclidean metric on $\RR^{2n}$; notice that one can take any $r_0>0$ as a slow variation constant for $g$. Given any $\varphi\in\SSS(\RR^{2n})\backslash\{0\}$, the mapping $\boldsymbol{\varphi}:\RR^{2n}\rightarrow \SSS(\RR^{2n})$, $\boldsymbol{\varphi}(X):=\varphi(\cdot-X)$, is a smooth non-degenerate element of $\Conf_g(\RR^{2n};r)$ for any $r>0$; in this case, the quantity \eqref{non-degen-ele-cond-int-cottt} equals $\|\varphi\|^2_{L^2(\RR^{2n})}$.
\end{example}

\begin{remark}
If $g$ is smooth and satisfies the bounds \eqref{ine-for-met-der-on-all-var}, one does not need to regularise it as in Example \ref{exi-of-good-par-off} $(ii)$ in order for $\boldsymbol{\varphi}$ to satisfy the properties $(a)-(d)$; it suffices to define $\boldsymbol{\varphi}$ via $g$. The H\"ormander metrics of the commonly used calculi almost always satisfy \eqref{ine-for-met-der-on-all-var}. In fact, if $g_{(x,\xi)}=f(x,\xi)^{-2}|dx|^2+F(x,\xi)^{-2}|d\xi|^2$ is a H\"ormander metric on $\RR^{2n}$ with $f$ and $F$ smooth positive and $f\in S(f,g)$ and $F\in S(F,g)$ (this holds true for the Shubin calculus, the SG-calculus, the H\"ormander $S_{\rho,\delta}$-calculus), then it is straightforward to check that $g$ satisfies the bounds \eqref{ine-for-met-der-on-all-var}.
\end{remark}

\subsection{A geometric version of the short-time Fourier transform}\label{GSTFT}

We denote by $\mathcal{F}_{\sigma}$ the symplectic Fourier transform on $W$:
$$
\mathcal{F}_{\sigma}f(X)=\int_W e^{-2\pi i [X,Y]} f(Y)dY,\quad f\in L^1(W).
$$
Recall that $\mathcal{F}_{\sigma}\mathcal{F}_{\sigma}=\operatorname{Id}$.\\
\indent Let $\boldsymbol{\varphi}\in\Conf_g(W;r)$ and set $\varphi_X:=\boldsymbol{\varphi}(X)$, $X\in W$. We define the \textit{geometric short-time Fourier transform} (GSTFT) $\VV_{\boldsymbol{\varphi}}f$ of $f\in \SSS'(W)$ with respect to $\boldsymbol{\varphi}$ as
$$
\VV_{\boldsymbol{\varphi}}f(X,\Xi):=\mathcal{F}_{\sigma}(f\overline{\varphi_X})(\Xi)=\langle f,e^{-2\pi i [\Xi,\cdot]}\overline{\varphi_X}\rangle,
\quad X,\Xi\in W.
$$
When $f\in L^1_{(1+|\cdot|)^{-s}}(W)$, for some $s\geq 0$, where $|\cdot|$ is (any) norm on $W$, we have
$$
\VV_{\boldsymbol{\varphi}}f(X,\Xi)=\int_W e^{-2\pi i[\Xi,Y]} f(Y)\overline{\varphi_X(Y)} dY,\quad X,\Xi\in W.
$$
Given any norm $|\cdot|$ on $W$, for each $k\in\NN$, the mapping $W\rightarrow \DD_{L^{\infty}_{(1+|\cdot|)^{-k-1}}}(W)$, $\Xi\mapsto e^{-2\pi i [\Xi,\cdot]}$, is of class $\mathcal{C}^k$. Hence, $W\times W\rightarrow \SSS(W)$, $(X,\Xi)\mapsto e^{-2\pi i[\Xi,\cdot]}\overline{\varphi}_X$, is strongly measurable. If $\boldsymbol{\varphi}$ is of class $\mathcal{C}^k$, $0\leq k\leq\infty$, then this mapping is also of class $\mathcal{C}^k$. Consequently, the function $W\times W\rightarrow \CC$, $(X,\Xi)\mapsto \VV_{\boldsymbol{\varphi}}f(X,\Xi)$, is always measurable and, if $\boldsymbol{\varphi}$ is of class $\mathcal{C}^k$, $0\leq k\leq \infty$, then $\VV_{\boldsymbol{\varphi}}f\in\mathcal{C}^k(W\times W)$, $0\leq k\leq\infty$.

\begin{remark}\label{rem-for-euc-metric-stand}
Let $g$ be the standard Euclidian metric on $\RR^{2n}$ and $\boldsymbol{\varphi}\in\Conf_g(\RR^{2n};r)$ as constructed in Example \eqref{exi-of-good-par-off} $(iii)$. Then $\VV_{\boldsymbol{\varphi}}f(X,\Xi)=V_{\varphi}f(X,\sigma \Xi)$, $X,\Xi\in\RR^{2n}$, where $V_{\varphi}$ is the standard STFT on $\RR^{2n}$.
\end{remark}

Our immediate goal is to study the mapping properties of $\VV_{\boldsymbol{\varphi}}$. For this purpose, we consider the Fr\'echet space $\ds\lim_{\substack{\longleftarrow\\ s\rightarrow \infty}}L^{\infty}_{(1+|\cdot|)^s}(W\times W)$ and the $(LB)$-spaces $\ds\lim_{\substack{\longrightarrow\\ s\rightarrow \infty}} L^{\infty}_{(1+|\cdot|)^{-s}}(W\times W)$ and $\ds\lim_{\substack{\longrightarrow\\ s\rightarrow \infty}} L^1_{(1+|\cdot|)^{-s}}(W\times W)$, where $|\cdot|$ is any norm on $W\times W$ (the linking mappings in the projective and inductive limits are the canonical inclusion); neither of these spaces depends on the particular choice of the norm $|\cdot|$. From \cite[Theorem 1.4]{reiher} (cf. \cite{bar-n-o,dierolf}), it follows that both inductive limits are regular and complete; for the completeness of $\displaystyle \lim_{\substack{\longrightarrow\\ s\rightarrow \infty}} L^{\infty}_{(1+|\cdot|)^{-s}}(W\times W)$, one employs \cite[Theorem 1.4]{reiher} and a standard argument to show that the space is quasi-complete and consequently complete in view of \cite[Theorem 3, p. 402]{kothe1}. We have the following continuous inclusions:
\begin{multline*}
\SSS(W\times W)\subseteq \lim_{\substack{\longleftarrow\\ s\rightarrow \infty}} L^{\infty}_{(1+|\cdot|)^s}(W\times W) \subseteq\lim_{\substack{\longrightarrow\\ s\rightarrow \infty}} L^{\infty}_{(1+|\cdot|)^{-s}}(W\times W)\subseteq\\
\subseteq \lim_{\substack{\longrightarrow\\ s\rightarrow \infty}} L^1_{(1+|\cdot|)^{-s}}(W\times W)\subseteq \SSS'(W\times W).
\end{multline*}
Furthermore, $\SSS(W\times W)$ is dense in $\displaystyle\lim_{\substack{\longrightarrow\\ s\rightarrow \infty}} L^1_{(1+|\cdot|)^{-s}}(W\times W)$ and the following topological isomorphism holds true \cite[Theorem 1.4]{reiher}:
\begin{equation}\label{top-iso-l1linf-spa-reg-lim-topplastr}
\left(\lim_{\substack{\longrightarrow\\ s\rightarrow \infty}} L^1_{(1+|\cdot|)^{-s}}(W\times W)\right)'_b=\lim_{\substack{\longleftarrow\\ s\rightarrow \infty}} L^{\infty}_{(1+|\cdot|)^s}(W\times W),
\end{equation}
where the index $b$ stands for the strong dual topology.\\
\indent The ensuing proposition collects the continuity properties of the GSTFT.

\begin{proposition}\label{lemma-for-conti-of-stft-sympl}
Let $0<r\leq r_0$.
\begin{itemize}
\item[$(i)$] The sesquilinear mapping
\begin{equation}\label{ses-lin-mapp-stft-for-s'}
\SSS'(W)\times \Conf_g(W;r)\rightarrow \lim_{\substack{\longrightarrow \\ s\rightarrow \infty}} L^{\infty}_{(1+|\cdot|)^{-s}}(W\times W),\quad (f,\boldsymbol{\varphi})\mapsto \VV_{\boldsymbol{\varphi}}f,
\end{equation}
is well-defined and hypocontinuous. Furthermore, for any bounded subset $B$ of $\SSS'(W)$ there is $s>0$ such that $\VV_{\boldsymbol{\varphi}}f\in L^{\infty}_{(1+|\cdot|)^{-s}}(W\times W)$, for all $f\in B$, $\boldsymbol{\varphi}\in\Conf_g(W;r)$, and the set of linear mappings
\begin{equation*}
\Conf_g(W;r)\rightarrow L^{\infty}_{(1+|\cdot|)^{-s}}(W\times W),\quad\boldsymbol{\varphi}\mapsto \VV_{\overline{\boldsymbol{\varphi}}}f,\qquad f\in B,
\end{equation*}
is an equicontinuous subset of $\mathcal{L}(\Conf_g(W;r),L^{\infty}_{(1+|\cdot|)^{-s}}(W\times W))$.
\item[$(ii)$] The sesquilinear mapping
$$
\SSS(W)\times \Conf_g(W;r)\rightarrow \lim_{\substack{\longleftarrow \\ s\rightarrow \infty}} L^{\infty}_{(1+|\cdot|)^s}(W\times W),\quad (\psi,\boldsymbol{\varphi})\mapsto \VV_{\boldsymbol{\varphi}}\psi,
$$
is well-defined and continuous.
\item[$(iii)$] Let $\boldsymbol{\varphi}\in\Conf_g(W;r)$ be of class $\mathcal{C}^{\infty}$. Assume that the function $W\times W\rightarrow \CC$, $(X,Y)\mapsto \varphi(X,Y):=\boldsymbol{\varphi}(X)(Y)$, satisfies the following: for every $k\in\NN$, there is $s\geq 0$ such that for all $m>0$
$$
\sup_{\substack{l'\leq k\\ l''\leq k}} \sup_{X,Y\in W}\sup_{\substack{T_1,\ldots,T_{l'}\in W\backslash\{0\}\\ S_1,\ldots,S_{l''}\in W\backslash\{0\}}} \frac{\left|(\prod_{j=1}^{l'}\partial_{T_j;X}) (\prod_{j=1}^{l''}\partial_{S_j;Y})\varphi(X,Y)\right| (1+g^{\sigma}_X(Y-U_{X,r}))^{m/2}} {(\prod_{j=1}^{l'} g_X(T_j)^{1/2}) (\prod_{j=1}^{l''} g_X(S_j)^{1/2})(1+|X|)^s(1+|Y|)^s}<\infty,
$$
where $|\cdot|$ is (any) norm on $W$. Then the mapping $\SSS(W)\rightarrow \SSS(W\times W)$, $\psi\mapsto \VV_{\boldsymbol{\varphi}}\psi$, is well-defined and continuous.
\end{itemize}
\end{proposition}

\begin{remark}
The condition in $(iii)$ is satisfied when $\boldsymbol{\varphi}$ is as in Example \ref{exi-of-good-par-off} $(ii)$; in fact, it is satisfied with $s=0$ for all $k\in\NN$.
\end{remark}

\begin{proof} We first address $(i)$. We claim that the second part of $(i)$ implies the first one. To see this, notice first that the second part implies that \eqref{ses-lin-mapp-stft-for-s'} is separately continuous. Indeed, it implies that for each fixed $\boldsymbol{\varphi}\in\Conf_g(W;r)$, the linear map
$$
\SSS'(W)\rightarrow \ds\lim_{\substack{\longrightarrow \\ s\rightarrow \infty}} L^{\infty}_{(1+|\cdot|)^{-s}}(W\times W),\quad f\mapsto \VV_{\boldsymbol{\varphi}}f,
$$
is well-defined and maps bounded sets into bounded sets and consequently, it is continuous since $\SSS'(W)$ is bornological. The continuity in $\boldsymbol{\varphi}\in\Conf_g(W;r)$ for fixed $f\in\SSS'(W)$ is immediate. As both $\SSS'(W)$ and $\Conf_g(W;r)$ are barrelled spaces, \cite[Theorem 5, p. 159]{kothe2} implies that \eqref{ses-lin-mapp-stft-for-s'} is hypocontinuous. It remains to prove the second part. Let $B$ be a bounded subset of $\SSS'(W)$. Then $B$ is equicontinuous. Hence, there exist $C'\geq 1$ and $k\in\ZZ_+$ such that
\begin{equation}\label{equicon-subset-s'-bou-for-dualit}
|\langle f,\psi\rangle|\leq C' \sup_{k'\leq k}\sup_{\substack{Y\in W\\ T_1,\ldots, T_{k'}\in W\backslash\{0\}}} \frac{|\partial_{T_1}\ldots\partial_{T_{k'}}\psi(Y)|(1+g_0(Y))^{k/2}} {\prod_{j=1}^{k'} g_0(T_j)^{1/2}},\;\; \psi\in\SSS(W),\, f\in B.
\end{equation}
Let $\boldsymbol{\varphi}\in\Conf_g(W;r)$ and $\varphi_X:=\boldsymbol{\varphi}(X)$, $X\in W$. For $k'\leq k$, $\partial_{T_1;Y}\ldots\partial_{T_{k'};Y}(e^{-2\pi i[\Xi,Y]}\varphi_X(Y))$ is a sum of $2^{k'}$ terms of the form
\begin{equation}\label{est-par-forfirst-part-cont-s'}
e^{-2\pi i[\Xi,Y]}\left(\prod_{j\in K'_1}(-2\pi i)[\Xi,T_j]\right)\left(\left(\prod_{j\in K'_2}\partial_{T_j;Y}\right)\varphi_X(Y)\right),
\end{equation}
where the sets $K'_1$ and $K'_2$ are disjoint and their union is $\{1,\ldots,k'\}$. Observe that \eqref{est-par-forfirst-part-cont-s'} is bounded by
\begin{align*}
&\frac{(2\pi)^{k'}\|\varphi_X\|^{(k)}_{g_X,U_{X,r}}}{(1+g^{\sigma}_X(Y-U_{X,r}))^{k/2}} \left(\prod_{j\in K'_1}g_0(T_j)^{1/2}\right)\left(\prod_{j\in K'_1} g^{\sigma}_0(\Xi)^{1/2}\right)\left(\prod_{j\in K'_2}g_X(T_j)^{1/2}\right)\\
&\leq \frac{C'_1\|\varphi_X\|^{(k)}_{g_X,U_{X,r}} (1+g^{\sigma}_0(\Xi))^{k/2} (1+g^{\sigma}_0(X))^{kN_0/2} \prod_{j=1}^{k'}g_0(T_j)^{1/2}}{(1+g^{\sigma}_X(Y-U_{X,r}))^{k/2}}.
\end{align*}
Since $g_0(Y)\leq 2g_0(X-Y)+2g_0(X)$, we deduce
\begin{align}
g_0(Y)&\leq 2C_0g_X(X-Y)(1+g^{\sigma}_0(X))^{N_0}+2g_0(X)\leq 2C_0(1+g_X(X-Y))(1+g^{\sigma}_0(X))^{N_0+1}\nonumber \\
&\leq 2C_0(1+2r^2+2g^{\sigma}_X(Y-U_{X,r})) (1+g^{\sigma}_0(X))^{N_0+1}.\label{est-for-met-at-poi-sing}
\end{align}
In view of \eqref{equicon-subset-s'-bou-for-dualit}, this implies
$$
|\VV_{\overline{\boldsymbol{\varphi}}}f(X,\Xi)|\leq C'_2 \|\varphi_X\|^{(k)}_{g_X,U_{X,r}} (1+g^{\sigma}_0(\Xi))^{k/2} (1+g^{\sigma}_0(X))^{k(2N_0+1)/2},\;\; X,\Xi\in W,
$$
which, in turn, proves the second part of $(i)$.\\
\indent We turn our attention to $(ii)$. For each $k\in\ZZ_+$, denote by $\|\psi\|_k$ the norm of $\psi\in \SSS(W)$ given on the right-hand side of \eqref{equicon-subset-s'-bou-for-dualit}. Let $k\in\ZZ_+$ be arbitrary but fixed. Pick $N,N_1\in\ZZ_+$ such that $N\geq kN_0+k$ and $N_1\geq 2kN_0+k+2n+2$. For each fixed $\Xi\in W$ choose $\theta=\theta(\Xi)\in W$ such that $g_0(\theta)=1$ and $[\Xi,\theta]=g^{\sigma}_0(\Xi)^{1/2}$. Let $\psi\in \SSS(W)$ and $\boldsymbol{\varphi}\in\Conf_g(W;r)$. Employing
\begin{equation}\label{equ-for-chn-est-der-fun}
(1-(2\pi i)^{-1}\partial_{\theta;Y})^ke^{-2\pi i[\Xi, Y]}=(1+g^{\sigma}_0(\Xi)^{1/2})^ke^{-2\pi i[\Xi,Y]},
\end{equation}
we integrate by parts in the integral defining $\VV_{\boldsymbol{\varphi}}\psi(X,\Xi)$ and infer
\begin{align*}
(1+&g^{\sigma}_0(\Xi)^{1/2})^k|\VV_{\boldsymbol{\varphi}}\psi(X,\Xi)|\\
&\leq C'_1\sum_{k'+k''\leq k}\int_W |\partial^{k''}_{\theta}\psi(Y)||\partial^{k'}_{\theta;Y}\varphi_X(Y)|dY\\
&\leq C'_2\|\psi\|_{N_1}\|\varphi_X\|^{(2N)}_{g_X,U_{X,r}} \sum_{k'+k''\leq k}\int_W \frac{g_X(\theta)^{k'/2}dY}{(1+g_0(Y))^{N_1/2}(1+g^{\sigma}_X(Y-U_{X,r}))^N}.
\end{align*}
As $g^{\sigma}_0(T)\leq C''g_0(T)$, $ T\in W$, for some $C''\geq 1$, employing \eqref{ineq-for-metric-p-1} we infer (recall $g_0(\theta)=1$)
\begin{align}
g_X(\theta)^{1/2}&\leq C'_3g_Y(\theta)^{1/2} (1+g^{\sigma}_X(Y-U_{X,r}))^{N_0/2}\nonumber\\
&\leq C'_4(1+g_0(Y))^{N_0/2} (1+g^{\sigma}_X(Y-U_{X,r}))^{N_0/2}.\label{ine-for-part-of-est-s'-cont}
\end{align}
In view of \eqref{ineq-for-metric-p-1}, we have
\begin{align*}
1+g_0(X-Y)^{1/2}&\leq C'_5(1+g_Y(X-Y)^{1/2})(1+g_0(Y))^{N_0/2}\\
&\leq C'_6 (1+g_X(X-Y)^{1/2}) (1+g^{\sigma}_X(Y-U_{X,r}))^{N_0/2}(1+g_0(Y))^{N_0/2}\\
&\leq C'_7 (1+g^{\sigma}_X(Y-U_{X,r}))^{1+N_0/2}(1+g_0(Y))^{N_0/2}.
\end{align*}
Consequently, as $1+g_0(X)^{1/2}\leq (1+g_0(X-Y)^{1/2})(1+g_0(Y)^{1/2})$, we have
\begin{equation}\label{est-for-cont-of-s'-with-stft-metr}
1+g_0(X)^{1/2}\leq C'_8 (1+g^{\sigma}_X(Y-U_{X,r}))^{1+N_0/2}(1+g_0(Y))^{(N_0+1)/2}.
\end{equation}
Employing \eqref{ine-for-part-of-est-s'-cont} and \eqref{est-for-cont-of-s'-with-stft-metr} in the above estimate for $\VV_{\boldsymbol{\varphi}}\psi$, we deduce
\begin{equation*}
(1+g^{\sigma}_0(X)^{1/2})^k(1+g^{\sigma}_0(\Xi)^{1/2})^k |\VV_{\boldsymbol{\varphi}}\psi(X,\Xi)|\leq C'_9 \|\psi\|_{N_1}\|\varphi_X\|^{(2N)}_{g_X,U_{X,r}},\;\; X,\Xi\in W,
\end{equation*}
which completes the proof of $(ii)$.\\
\indent It remains to prove $(iii)$. As before, we denote by $\|\cdot\|_k$, $k\in\NN$, the norms on $\SSS(W)$ given by the right-hand side of \eqref{equicon-subset-s'-bou-for-dualit}. Let $\psi\in \SSS(W)$ and $k\in\ZZ_+$ be arbitrary but fixed. Pick $s\geq 0$ for which the condition in $(iii)$ holds for this $k$; we employ the norm $|X|:=g_0(X)^{1/2}$ on $W$ in this condition. Set $m:= (2+N_0)(2kN_0+s+k)$ and pick $N\in\ZZ_+$ such that $N\geq k+s+(N_0+1)(2kN_0+s+k)+2n+2$. For $l',l''\leq k$, $l',l''\in\NN$, we employ \eqref{equ-for-chn-est-der-fun} with $\theta\in W$ chosen as in $(ii)$ to infer
\begin{multline*}
(1+g^{\sigma}_0(\Xi)^{1/2})^k \left|\left(\prod_{j=1}^{l'}\partial_{T_j;X}\right)\left(\prod_{j=1}^{l''} \partial_{S_j;\Xi}\right) \VV_{\boldsymbol{\varphi}}\psi(X,\Xi)\right|\\
\leq C_1\sum_{\substack{k_1+k_2+k_3\leq k\\ k_1\leq l''}} \int_W \left|\partial_{\theta;Y}^{k_1}\left(\prod_{j=1}^{l''}[S_j,Y]\right)\right| |\partial_{\theta;Y}^{k_2}\psi(Y)| \left|\left(\prod_{j=1}^{l'}\partial_{T_j;X}\right) \partial_{\theta;Y}^{k_3}\varphi_X(Y)\right|dY.
\end{multline*}
Notice that
$$
\left|\partial_{\theta;Y}^{k_1}\left(\prod_{j=1}^{l''}[S_j,Y]\right)\right|\leq C_2 g_0(Y)^{(l''-k_1)/2}\prod_{j=1}^{l''}g^{\sigma}_0(S_j)^{1/2}\leq C_3 (1+g_0(Y))^{k/2}\prod_{j=1}^{l''}g_0(S_j)^{1/2}.
$$
Consequently, in view of \eqref{est-for-cont-of-s'-with-stft-metr}, we deduce
\begin{align*}
&(1+g^{\sigma}_0(X)^{1/2})^k(1+g^{\sigma}_0(\Xi)^{1/2})^k \left|\left(\prod_{j=1}^{l'}\partial_{T_j;X}\right)\left(\prod_{j=1}^{l''} \partial_{S_j;\Xi}\right) \VV_{\boldsymbol{\varphi}}\psi(X,\Xi)\right|\\
&\leq C_4\|\psi\|_N\left(\prod_{j=1}^{l''}g_0(S_j)^{1/2}\right) \sum_{k_1+k_2+k_3\leq k} \int_W \frac{g_X(\theta)^{k_3/2}(1+g_0(X))^{(s+k)/2}\prod_{j=1}^{l'}g_X(T_j)^{1/2} dY}{(1+g_0(Y))^{(N-k-s)/2} (1+g^{\sigma}_X(Y-U_{X,r}))^{m/2}}\\
&\leq C_5\|\psi\|_N\left(\prod_{j=1}^{l'}g_0(T_j)^{1/2}\right) \left(\prod_{j=1}^{l''}g_0(S_j)^{1/2}\right) \\
&{}\quad \cdot\sum_{k_1+k_2+k_3\leq k} \int_W \frac{(1+g_0(X))^{(k_3N_0+l'N_0+s+k)/2} dY}{(1+g_0(Y))^{(N-k-s)/2} (1+g^{\sigma}_X(Y-U_{X,r}))^{m/2}}\\
&\leq C_6\|\psi\|_N\left(\prod_{j=1}^{l'}g_0(T_j)^{1/2}\right) \left(\prod_{j=1}^{l''}g_0(S_j)^{1/2}\right) \int_W (1+g_0(Y))^{-n-1}dY.
\end{align*}
As the very last integral is a positive constant, the proof is complete.
\end{proof}

If $r<r_0$ is sufficiently small, one can always regularise the elements of $\Conf_g(W;r)$ in the following way such that they satisfy the condition in Proposition \ref{lemma-for-conti-of-stft-sympl} $(iii)$.

\begin{lemma}\label{lem-for-smooth-tem-part-of-unity-in-both-coord}
Let $r',r>0$ satisfy $0<r'\sqrt{C_0}<r\leq r_0$ and let $0<r'_0\leq C_0^{-1/2}r-r'$. Let $\boldsymbol{\varphi}\in\Conf_g(W;r'_0)$ be such that $\supp\varphi_X\subseteq U_{X,r'_0}$, $ X\in W$, where $\varphi_X:=\boldsymbol{\varphi}(X)$. For $\boldsymbol{\psi}\in\Conf_g(W;r')$, set $\psi_X:=\boldsymbol{\psi}(X)$, $X\in W$, and
\begin{equation}\label{eqs11-1}
\widetilde{\boldsymbol{\psi}}(X):=\widetilde{\psi}_X,\quad \mbox{where}\quad \widetilde{\psi}_X(Y):=\int_W \psi_Z(Y)\varphi_Z(X)dv_g(Z),\; X,Y\in W.
\end{equation}
Then $\widetilde{\boldsymbol{\psi}}\in\Conf_g(W;r)$, it is of class $\mathcal{C}^{\infty}$ and it satisfies the condition in Proposition \ref{lemma-for-conti-of-stft-sympl} $(iii)$ with $s=0$ for all $k\in\NN$. Moreover, the mapping $\Conf_g(W;r')\rightarrow \Conf_g(W;r)$, $\boldsymbol{\psi}\mapsto \widetilde{\boldsymbol{\psi}}$, with $\widetilde{\boldsymbol{\psi}}$ given by \eqref{eqs11-1}, is continuous.\\
\indent Furthermore, if there exists $r''>0$ such that $\supp\psi_X\subseteq U_{X,r''}$, $X\in W$, then $\supp\tilde{\psi}_X\subseteq U_{X,(r''+r'_0)\sqrt{C_0}}$, $X\in W$.
\end{lemma}

\begin{proof} Let $k\in \NN$ be arbitrary but fixed. Pick $N\in\ZZ_+$ such that $N\geq 2(N_0+1)(n+1)+2nN_0+k$. For $l',l''\leq k$, $l',l''\in\NN$, \eqref{eqs11-1} gives
\begin{multline*}
\frac{\left|(\prod_{j=1}^{l'}\partial_{T_j;X}) (\prod_{j=1}^{l''}\partial_{S_j;Y})\widetilde{\psi}_X(Y)\right|} {(\prod_{j=1}^{l'} g_X(T_j)^{1/2}) (\prod_{j=1}^{l''} g_X(S_j)^{1/2})}\\
\leq C'_1\|\boldsymbol{\psi}\|^{(N)}_{g,r'}\int_W \frac{\left|\left(\prod_{j=1}^{l'}\partial_{T_j;X}\right)\varphi_Z(X)\right|} {(1+g^{\sigma}_Z(Y-U_{Z,r'}))^{N/2}\prod_{j=1}^{l'} g_Z(T_j)^{1/2}}dv_g(Z).
\end{multline*}
We estimate the integrand as follows. When $Z\in W$ is such that $g_Z(X-Z)> r'^2_0$, the integrand is $0$. If $g_Z(X-Z)\leq r'^2_0$, then $g^{\sigma}_Z(Y-U_{Z,r'})\geq g^{\sigma}_X(Y-U_{Z,r'})/C_0$ and $U_{Z,r'}\subseteq U_{X,r}$. To verify the latter, let $Z'\in U_{Z,r'}$. Then
\begin{align*}
g_X(X-Z')^{1/2}&\leq C_0^{1/2}g_Z(X-Z')^{1/2}\leq C_0^{1/2}(g_Z(X-Z)^{1/2}+g_Z(Z-Z')^{1/2})\\
&\leq C_0^{1/2}(r'_0+r')\leq r,
\end{align*}
which proves that $U_{Z,r'}\subseteq U_{X,r}$. Hence, in view of \eqref{ineq-for-metric-p-3-1}, we deduce
$$
\frac{\left|(\prod_{j=1}^{l'}\partial_{T_j;X}) (\prod_{j=1}^{l''}\partial_{S_j;Y})\widetilde{\psi}_X(Y)\right|} {(\prod_{j=1}^{l'} g_X(T_j)^{1/2}) (\prod_{j=1}^{l''} g_X(S_j)^{1/2})}
\leq C'_2\|\boldsymbol{\psi}\|^{(N)}_{g,r'}(1+g^{\sigma}_X(Y-U_{X,r}))^{-k/2}.
$$
Employing standard arguments and these bounds, one shows that $\widetilde{\boldsymbol{\psi}}\in\mathcal{C}^{\infty}(W;\SSS(W))$. The rest of the claimed properties of $\widetilde{\boldsymbol{\psi}}$ are an immediate consequence of the above bounds. The proof of the very last statement is straightforward and we omit it.
\end{proof}

\begin{remark}
The above lemma is applicable with $\boldsymbol{\varphi}\in\Conf_g(W;r)$ constructed in Example \ref{exi-of-good-par-off} $(i)$ and Example \ref{exi-of-good-par-off} $(ii)$.
\end{remark}

Similarly as the classical STFT, $\VV_{\boldsymbol{\varphi}}$ satisfies the following orthogonality relation.

\begin{proposition}
Let $\boldsymbol{\varphi},\boldsymbol{\psi}\in\Conf_g(W;r)$ and $f_1,f_2\in L^2(W)$. Then $\VV_{\boldsymbol{\varphi}}f_1, \VV_{\boldsymbol{\psi}}f_2\in L^2(W\times W, dv_gd\lambda)$ and
\begin{equation}\label{equality-orth-for-fam}
(\VV_{\boldsymbol{\varphi}}f_1,\VV_{\boldsymbol{\psi}}f_2)_{L^2(W\times W, dv_gd\lambda)}= (f_1,I_{\boldsymbol{\varphi}\overline{\boldsymbol{\psi}}}f_2)_{L^2(W)}.
\end{equation}
\end{proposition}

\begin{proof} It is enough to show \eqref{equality-orth-for-fam} for $f_1,f_2\in\SSS(W)$ as the rest will follow by density (cf. Proposition \ref{lemma-for-conti-of-stft-sympl}). Setting $\varphi_X:=\boldsymbol{\varphi}(X)$, $\psi_X:=\boldsymbol{\psi}(X)$, $X\in W$, Parseval's identity gives
\begin{align*}
(\VV_{\boldsymbol{\varphi}}f_1,\VV_{\boldsymbol{\psi}}f_2)_{L^2(W\times W, dv_gd\lambda)}&= \int_W\left(\int_W \mathcal{F}_{\sigma}(f_1\overline{\varphi_X})(\Xi) \overline{\mathcal{F}_{\sigma}(f_2\overline{\psi_X})(\Xi)}d\Xi\right)dv_g(X)\\
&=\int_W\left(\int_W f_1(Y)\overline{\varphi_X(Y)}\, \overline{f_2(Y)}\psi_X(Y)dY\right)dv_g(X)\\
&=(f_1,I_{\boldsymbol{\varphi}\overline{\boldsymbol{\psi}}}f_2)_{L^2(W)}.
\end{align*}
\end{proof}

In view of Lemma \ref{rem-about-part-of-unity} $(iii)$, we have the following consequence.

\begin{corollary}\label{cor-for-thestft-sympl}
Let $\boldsymbol{\varphi}\in\Conf_g(W;r)$. Then
\begin{equation}\label{shot-time-onl2-sympl}
\VV_{\boldsymbol{\varphi}}:L^2(W)\rightarrow L^2(W\times W, dv_gd\lambda)
\end{equation}
is continuous. If in addition $\boldsymbol{\varphi}$ is non-degenerate, then \eqref{shot-time-onl2-sympl} is a topological imbedding.
\end{corollary}

In view of Proposition \ref{lemma-for-conti-of-stft-sympl} $(i)$, for any $\boldsymbol{\varphi}\in\Conf_g(W;r)$, it holds that
$$
\VV_{\boldsymbol{\varphi}}:\SSS'(W)\rightarrow \lim_{\substack{\longrightarrow\\ s\rightarrow \infty}} L^1_{(1+|\cdot|)^{-s}}(W\times W)\,\, \mbox{is well-defined and continuous.}
$$
For $\ds G\in \lim_{\substack{\longrightarrow\\ s\rightarrow \infty}} L^1_{(1+|\cdot|)^{-s}}(W\times W)$ and $\boldsymbol{\varphi}\in\Conf_g(W;r)$, we define $\VV^*_{\boldsymbol{\varphi}}G\in \SSS'(W)$ by
$$
\langle \VV^*_{\boldsymbol{\varphi}}G,\chi\rangle:=\int_{W\times W} G(X,\Xi)\overline{\VV_{\boldsymbol{\varphi}}\overline{\chi}(X,\Xi)}dv_g(X)d\Xi=\langle G,|g|^{1/2}\overline{\VV_{\boldsymbol{\varphi}}\overline{\chi}}\rangle,\;\; \chi\in\SSS(W),
$$
where the last dual pairing is in the sense of \eqref{top-iso-l1linf-spa-reg-lim-topplastr}. The right-hand side makes sense in view of Proposition \ref{lemma-for-conti-of-stft-sympl} $(ii)$ and \eqref{ineq-for-metric-p-7}. Furthermore, Proposition \ref{lemma-for-conti-of-stft-sympl} $(ii)$ implies that $\VV^*_{\boldsymbol{\varphi}}G$ is indeed a well-defined element of $\SSS'(W)$. We collect the properties which we need  for $\VV^*_{\boldsymbol{\varphi}}$ in the following proposition.

\begin{proposition}\label{lemma-for-the-adj-of-stft-sympl}
Let $0<r\leq r_0$.
\begin{itemize}
\item[$(i)$] The bilinear mapping
    \begin{equation}\label{bil-map-adj-stft-gen-def}
    \lim_{\substack{\longrightarrow\\ s\rightarrow\infty}} L^1_{(1+|\cdot|)^{-s}}(W\times W) \times \Conf_g(W;r)\rightarrow \SSS'(W),\quad (G,\boldsymbol{\varphi})\mapsto \VV^*_{\boldsymbol{\varphi}}G,
    \end{equation}
    is hypocontinuous. Furthermore,
    \begin{equation}\label{compo-stft-with-adj-symp}
    \VV^*_{\boldsymbol{\varphi}}\VV_{\boldsymbol{\psi}}f=I_{\boldsymbol{\varphi} \overline{\boldsymbol{\psi}}}f,\quad f\in \SSS'(W),\, \boldsymbol{\varphi},\boldsymbol{\psi}\in\Conf_g(W;r).
    \end{equation}
\item[$(ii)$] Let $\boldsymbol{\varphi}\in\Conf_g(W;r)$. For each $G\in L^2(W\times W, dv_gd\lambda)$, the mapping
\begin{equation}\label{mapp-for-theadj-ofthe-stft-sympl}
W\times W\rightarrow L^2(W),\quad (X,\Xi)\mapsto G(X,\Xi) e^{2\pi i[\Xi,\cdot]} \boldsymbol{\varphi}(X),
\end{equation}
is strongly measurable and Pettis integrable on $W\times W$ with respect to $dv_gd\lambda$ and
$$
\VV^*_{\boldsymbol{\varphi}}G=\int_{W\times W} e^{2\pi i[\Xi,\cdot]}G(X,\Xi)\boldsymbol{\varphi}(X) dv_g(X)d\Xi.
$$
Furthermore, the mapping
$$
L^2(W\times W,dv_gd\lambda)\rightarrow L^2(W),\quad G\mapsto \VV^*_{\boldsymbol{\varphi}}G,
$$
is the adjoint to \eqref{shot-time-onl2-sympl}.
\item[$(iii)$] The bilinear mapping \eqref{bil-map-adj-stft-gen-def} restricts to a well-defined and continuous bilinear mapping
$$
\lim_{\substack{\longleftarrow\\ s\rightarrow\infty}} L^1_{(1+|\cdot|)^s}(W\times W) \times \Conf_g(W;r)\rightarrow \SSS(W),\quad (G,\boldsymbol{\varphi})\mapsto \VV^*_{\boldsymbol{\varphi}}G.
$$
\item[$(iv)$] Assume that $g$ is smooth and satisfies the following condition: for every $S\in W$ there exists a $g$-admissible weight $M_S$ such that the function $X\mapsto g_X(S)$ belongs to $S(M_S,g)$. If $\boldsymbol{\varphi}\in \Conf_g(W;r)$ satisfies the condition in Proposition \ref{lemma-for-conti-of-stft-sympl} $(iii)$, then $\VV^*_{\boldsymbol{\varphi}}$ uniquely extends to a continuous mapping $\VV^*_{\boldsymbol{\varphi}}:\SSS'(W\times W)\rightarrow \SSS'(W)$.
\end{itemize}
\end{proposition}

\begin{proof} We first address $(i)$. Proposition \ref{lemma-for-conti-of-stft-sympl} $(ii)$ verifies that \eqref{bil-map-adj-stft-gen-def} is separately continuous. Since both spaces in the domain of \eqref{bil-map-adj-stft-gen-def} are barrelled, \cite[Theorem 5, p. 159]{kothe2} verifies that \eqref{bil-map-adj-stft-gen-def} is hypocontinuous. The validity of \eqref{compo-stft-with-adj-symp} for $f\in\SSS(W)$ is a consequence of \eqref{equality-orth-for-fam} and the general case follows from the density of $\SSS(W)$ in $\SSS'(W)$ and Proposition \ref{lemma-for-conti-of-stft-sympl} $(i)$. We turn our attention to $(ii)$. The fact that \eqref{mapp-for-theadj-ofthe-stft-sympl} is strongly measurable follows from the comments before Remark \ref{rem-for-euc-metric-stand}. The rest of $(ii)$ follows from Corollary \ref{cor-for-thestft-sympl} and a straightforward computation. To prove $(iii)$, let $k,l\in\NN$ be arbitrary but fixed. For $\boldsymbol{\varphi}\in \Conf_g(W;r)$, setting $\varphi_X:=\boldsymbol{\varphi}(X)$, $X\in W$, and employing the same technique as in the proof of Proposition \ref{lemma-for-conti-of-stft-sympl} $(i)$, one can show that
\begin{multline*}
\left|\left(\prod_{j=1}^k\partial_{T_j;Y}\right)(e^{2\pi i [\Xi,Y]}\varphi_X(Y))\right|\\
\leq \frac{C_1\|\varphi_X\|^{(k+l)}_{g_X,U_{X,r}} (1+g^{\sigma}_0(\Xi))^{k/2}(1+g^{\sigma}_0(X))^{kN_0/2} \prod_{j=1}^k g_0(T_j)^{1/2}} {(1+g^{\sigma}_X(Y-U_{X,r}))^{l/2}}.
\end{multline*}
Let $\ds G\in \lim_{\substack{\longleftarrow\\ s\rightarrow\infty}} L^1_{(1+|\cdot|)^s}(W\times W)$. Employing \eqref{inequ-for-met-sim-vol-der} and \eqref{est-for-met-at-poi-sing}, we deduce
\begin{multline*}
\frac{(1+g_0(Y))^{l/2}\left|\left(\prod_{j=1}^k\partial_{T_j}\right) \VV^*_{\boldsymbol{\varphi}}G(Y)\right|}{\prod_{j=1}^k g_0(T_j)^{1/2}}\\
\leq C_2\|\boldsymbol{\varphi}\|^{(k+l)}_{g,r}\int_{W\times W} |G(X,\Xi)| (1+g^{\sigma}_0(\Xi))^{k/2}(1+g^{\sigma}_0(X))^{(kN_0+l(N_0+1)+2nN_0)/2}dXd\Xi,
\end{multline*}
which implies the claim in $(iii)$.\\
\indent We now address $(iv)$. For every $S_1,S_2\in W$, it holds that $g_X(S_1,S_2)=(g_X(S_1+S_2)-g_X(S_1)-g_X(S_2))/2$. Let $E_j$, $j=1,\ldots,2n$, be a symplectic basis for $W$. For $k\in\NN$, the condition in $(iv)$ gives the following bound for some $m_0\geq 0$:
\begin{align*}
|\partial^k_Tg_X(E_j,E_l)|&\leq C'g_X(T)^{k/2}(M_{E_j+E_l}(X)+M_{E_j}(X)+M_{E_l}(X))\\
&\leq C''g_X(T)^{k/2}(1+g^{\sigma}_0(X))^{m_0},
\end{align*}
for all $X\in W$, $j,l=1,\ldots,2n$. Consequently,
$$
|\partial^k_T|g_X||\leq \tilde{C}'g_X(T)^{k/2}(1+g^{\sigma}_0(X))^{2nm_0},\quad  X,T\in W.
$$
The Fa\'a di Bruno formula (see \cite[Section 4.3.1]{lernerB}) applied to the composition of $X\mapsto |g_X|$ with $t\mapsto \sqrt{t}$ together with \eqref{inequ-for-met-sim-vol-der} yield (for $k\in\ZZ_+$)
\begin{align*}
|\partial^k_T(|g_X|^{1/2})|&\leq C'_1\sum_{l=1}^k|g_X|^{-(2l-1)/2}\sum_{\substack{k_1+\ldots +k_l=k\\ k_j\geq 1}}\prod_{j=1}^l |\partial^{k_j}_T|g_X||\\
&\leq C'_2 g_X(T)^{k/2}(1+g^{\sigma}_0(X))^{2nm_0k}\sum_{l=1}^k(1+g^{\sigma}_0(X))^{(2l-1)nN_0}\\
&\leq C'_3 g_X(T)^{k/2}(1+g^{\sigma}_0(X))^{2nm_0k+2knN_0}.
\end{align*}
Now, \cite[Lemma 4.2.3, p. 302]{lernerB} verifies the following bound for $k\in\ZZ_+$:
$$
\left|\left(\prod_{j=1}^k\partial_{T_j}\right)|g_X|^{1/2}\right|\leq C'_3(1+g^{\sigma}_0(X))^{2nm_0k+2knN_0}\prod_{j=1}^kg_X(T_j)^{1/2},\;\; X,T_1,\ldots,T_k\in W.
$$
Let $B$ be a bounded subset of $\SSS(W)$. The above estimate together with Proposition \ref{lemma-for-conti-of-stft-sympl} $(iii)$ imply that $B_1=\{|g|^{1/2} \overline{\VV_{\boldsymbol{\varphi}}\overline{\chi}}\,|\, \chi\in B\}$ is a bounded subset of $\SSS(W\times W)$. For any $G\in \SSS(W\times W)$, we have $\sup_{\chi\in B}|\langle \VV^*_{\boldsymbol{\varphi}}G,\chi\rangle|=\sup_{\phi\in B_1}|\langle G, \phi\rangle|$, which proves the claim in $(iv)$.
\end{proof}

\begin{remark}
Given any H\"ormander metric $g$, one can always find a smooth H\"ormander metric $\tilde{g}$ which satisfies the condition in Proposition \ref{lemma-for-the-adj-of-stft-sympl} $(iv)$ and is equivalent to $g$ (i.e., there exists $C\geq 1$ such that $C^{-1}g_X(T)\leq \tilde{g}_X(T)\leq Cg_X(T)$, $X,T\in W$). In fact, the smooth H\"ormander metric $\tilde{g}$ constructed in Example \ref{exi-of-good-par-off} $(ii)$ is equivalent to $g$ and \eqref{ine-for-met-der-on-all-var} proves that for each $S\in W\backslash\{0\}$, the function $X\mapsto \tilde{g}_X(S)$ belongs to $S(M_S,\tilde{g})$ with $M_S(X):=\tilde{g}_X(S)$, $X\in W$; clearly $M_S$ is $\tilde{g}$-admissible.
\end{remark}

\begin{remark}
The H\"ormander metrics of the commonly used calculi  almost always satisfy the assumption in Proposition \ref{lemma-for-the-adj-of-stft-sympl} $(iv)$ (like the Shubin calculus, the SG-calculus, the H\"ormander $S_{\rho,\delta}$-calculus etc.). In fact, if $g_{x,\xi}=f(x,\xi)^{-2}|dx|^2+F(x,\xi)^{-2}|d\xi|^2$ is a H\"ormander metric on $\RR^{2n}$ such that $f$ and $F$ are smooth, positive and $f\in S(f,g)$ and $F\in S(F,g)$, then $g$ satisfies the assumption in Proposition \ref{lemma-for-the-adj-of-stft-sympl} $(iv)$.
\end{remark}

\section{Geometric modulation spaces}\label{gms}

In this section, we define a class of generalised modulation spaces with the help of the geometric short-time Fourier transform which we introduced in Subsection \ref{GSTFT}. Besides being of independent interest, they will play a crucial role in the proof of the main result.\\
\indent We start by introducing the following class of admissible weights. A positive measurable function $\eta:W\times W\rightarrow (0,\infty)$ is called a \textit{uniformly admissible weight with respect to the metric} $g$ if there are constants $\tilde{C}\geq 1$, $\tilde{r}>0$ and $\tilde{\tau}\geq 0$ such that
\begin{gather}
g_X(X-Y)\leq\tilde{r}^2\Rightarrow \tilde{C}^{-1}\eta(Y,\Xi)\leq \eta(X,\Xi)\leq \tilde{C}\eta(Y,\Xi),\quad X,Y,\Xi\in W,\label{slow-variation-with-resp-to-the-metric-uni}\\
\eta(X+X',Y)\leq \tilde{C}\eta(X,Y)(1+g^{\sigma}_X(X'))^{\tilde{\tau}},\quad X,X',Y\in W,\label{tem-with-res-to-first}\\
\eta(X,Y+Y')\leq \tilde{C}\eta(X,Y)(1+g^{\sigma}_X(Y'))^{\tilde{\tau}},\quad X,Y,Y'\in W.\label{tem-with-res-to-second}
\end{gather}
We call any number $\tilde{r}>0$ for which \eqref{slow-variation-with-resp-to-the-metric-uni} holds true a \textit{slow variation constant} for $\eta$ and we call any number $\tilde{\tau}\geq0$ for which \eqref{tem-with-res-to-first} and \eqref{tem-with-res-to-second} hold true a \textit{temperance constant} for $\eta$. In such case, we call the pair $(\tilde{r},\tilde{\tau})$ \textit{admissibility constants} for $\eta$.

\begin{remark}\label{rem-for-adm-func-for-fromad}
The reason for this notation is the following. If $M$ is a $g$-admissible weight with admissibility constants $r$ and $N$, then the function $(X,Y)\mapsto M(X)$ is uniformly admissible with respect to $g$ with admissibility constants $(r,N)$.\\
\indent Notice that \eqref{ineq-for-metric-p-5} and \eqref{ineq-for-metric-p-6} imply that the functions $(X,Y)\mapsto |g_X|$ and $(X,Y)\mapsto |g^{\sigma}_X|$ are uniformly admissible with respect to $g$ with admissibility constants $(r_0,2nN_0)$.
\end{remark}

\begin{remark}\label{uni-adm-for-stand-euc-metr}
If $g$ is the standard Euclidean metric on $\RR^{2n}$, then $\eta:\RR^{4n}\rightarrow (0,\infty)$ is uniformly admissible for $g$ if and only if $\eta$ is moderate with respect to the Beurling weight $(1+|\cdot|)^{\tau}$, for some $\tau\geq 0$ (cf. \cite[Definition 11.1.1, p. 217]{Gr1}).
\end{remark}

\begin{lemma}\label{admisibil-weig-mul}
Let $\eta$ and $\eta_1$ be two uniformly admissible weights with respect to $g$ with admissibility constants $(\tilde{r},\tilde{\tau})$ and $(\tilde{r}_1,\tilde{\tau}_1)$, respectively. Then $\eta\eta_1$ is uniformly admissible weight with respect to $g$ with admissibility constants $(\min\{\tilde{r},\tilde{r}_1\},\tilde{\tau}+\tilde{\tau}_1)$. Furthermore, $\eta^s$, $s\in\RR$, is also uniformly admissible with respect to $g$ with admissibility constants $(\tilde{r}, s\tilde{\tau})$ when $s\geq 0$ and $(\tilde{r}, |s|\tilde{\tau}(N_0+1))$ when $s<0$.
\end{lemma}

\begin{proof} The only non-trivial part is to verify \eqref{tem-with-res-to-first} for $\eta^{s}$, $s<0$, with the constant $|s|\tilde{\tau}(N_0+1)$ in place of $\tilde{\tau}$. This follows from the following chain of inequalities:
\begin{align*}
\eta(X,Y)&\leq \tilde{C}\eta(X+X',Y)(1+g^{\sigma}_{X+X'}(X'))^{\tilde{\tau}}\\
&\leq \tilde{C}\eta(X+X',Y)(1+C_0g^{\sigma}_X(X')(1+g^{\sigma}_X(X'))^{N_0})^{\tilde{\tau}}\\
&\leq \tilde{C}C_0^{\tilde{\tau}}\eta(X+X',Y)(1+g^{\sigma}_X(X'))^{\tilde{\tau}(N_0+1)}.
\end{align*}
\end{proof}

Let $w$ be a positive measurable function on $W\times W$ which satisfies \eqref{pol-b-d-wei-s}. For $1\leq p,q\leq\infty$, we denote by $L^{p,q}_w(W\times W, dv_gd\lambda)$ the Banach space of all measurable functions $f$ on $W\times W$ which satisfy
$$
\left(\int_W\left(\int_W |f(X,\Xi)|^p w(X,\Xi)^p dv_g(X)\right)^{q/p}d\Xi\right)^{1/q}<\infty
$$
(with the obvious modifications when $p=\infty$ or $q=\infty$). Similarly, we denote by $\widetilde{L}^{p,q}_w(W\times W, dv_gd\lambda)$ the Banach space of all measurable functions $f$ on $W\times W$ which satisfy
$$
\left(\int_W\left(\int_W |f(X,\Xi)|^q w(X,\Xi)^q d\Xi\right)^{p/q}dv_g(X)\right)^{1/p}<\infty
$$
(again, with the obvious modifications when $p=\infty$ or $q=\infty$). When $g$ is the Euclidean metric, we will simply denote these spaces by $L^{p,q}_w(W\times W)$ and $\widetilde{L}^{p,q}_w(W\times W)$, respectively. Clearly,
$$
L^{p,p}_w(W\times W,dv_gd\lambda)=\widetilde{L}^{p,p}_w(W\times W,dv_gd\lambda)= L^p_w(W\times W,dv_gd\lambda),\quad p\in[1,\infty].
$$
We have the following continuous inclusions for all $1\leq p,q\leq \infty$:
\begin{gather*}
\lim_{\substack{\longleftarrow \\ s\rightarrow\infty}} L^{\infty}_{(1+|\cdot|)^s}(W\times W)\subseteq L^{p,q}_w(W\times W,dv_g d\lambda)\subseteq \lim_{\substack{\longrightarrow \\ s\rightarrow\infty}} L^1_{(1+|\cdot|)^{-s}}(W\times W),\\
\lim_{\substack{\longleftarrow \\ s\rightarrow\infty}} L^{\infty}_{(1+|\cdot|)^s}(W\times W)\subseteq \widetilde{L}^{p,q}_w(W\times W,dv_g d\lambda)\subseteq \lim_{\substack{\longrightarrow \\ s\rightarrow\infty}} L^1_{(1+|\cdot|)^{-s}}(W\times W).
\end{gather*}
When $p<\infty$ and $q<\infty$, the inclusions are also dense. If $\eta$ is uniformly admissible weight with respect to $g$, than (cf. Lemma \ref{admisibil-weig-mul}),
$$
\eta(X,Y)^{\pm 1}\leq C\eta(0,0)^{\pm1}(1+g^{\sigma}_0(X))^N(1+g^{\sigma}_0(Y))^N,\quad X,Y\in W,
$$
for some $C,N>0$; consequently, it satisfies \eqref{pol-b-d-wei-s}.\\
\indent The generalised modulation spaces which we are going to define will be modelled on $\widetilde{L}^{p,q}_{\eta}(W\times W,dv_gd\lambda)$ rather than on $L^{p,q}_{\eta}(W\times W,dv_gd\lambda)$. As it turns out, the former spaces are better suited as co-domains and domains for the GSTFT and its adjoint.\\
\indent Anticipating what follows, we prove the following result.

\begin{theorem}\label{com-adj-stft-sympl1}
Let $\eta$ be a uniformly admissible weight with respect to $g$ with slow variation constant $\tilde{r}$ and let $0<r'_0\leq\min\{r_0,\tilde{r}\}$. For any $p,q\in[1,\infty]$ and any $\boldsymbol{\varphi},\boldsymbol{\psi}\in\Conf_g(W;r'_0)$, the mapping
$$
\widetilde{L}^{p,q}_{\eta}(W\times W,dv_gd\lambda)\rightarrow \widetilde{L}^{p,q}_{\eta}(W\times W,dv_gd\lambda),\quad G\mapsto \VV_{\boldsymbol{\psi}}\VV^*_{\boldsymbol{\varphi}} G,
$$
is well-defined and continuous. Furthermore, for each $p,q\in[1,\infty]$ there exist $C\geq 1$ and $k\in\ZZ_+$ such that
\begin{equation}\label{uni-bound-for-cont-mml}
\|\VV_{\boldsymbol{\psi}}\VV^*_{\boldsymbol{\varphi}} G\|_{\widetilde{L}^{p,q}_{\eta}(W\times W,dv_gd\lambda)}\leq C \|\boldsymbol{\psi}\|^{(k)}_{g,r'_0}\|\boldsymbol{\varphi}\|^{(k)}_{g,r'_0} \|G\|_{\widetilde{L}^{p,q}_{\eta}(W\times W,dv_gd\lambda)},
\end{equation}
for all $G\in \widetilde{L}^{p,q}_{\eta}(W\times W,dv_gd\lambda)$, $\boldsymbol{\varphi}, \boldsymbol{\psi}\in \Conf_g(W;r'_0)$.
\end{theorem}

Before we prove Theorem \ref{com-adj-stft-sympl1}, we show the following special case.

\begin{proposition}\label{com-adj-stft-sympl}
Let $\eta$, $\tilde{r}$ and $r'_0$ be as in Theorem \ref{com-adj-stft-sympl1}. For any $p\in[1,\infty]$ and any $\boldsymbol{\varphi},\boldsymbol{\psi}\in\Conf_g(W;r'_0)$, the mapping
$$
L^p_{\eta}(W\times W,dv_gd\lambda)\rightarrow L^p_{\eta}(W\times W,dv_gd\lambda),\quad G\mapsto \VV_{\boldsymbol{\psi}}\VV^*_{\boldsymbol{\varphi}} G,
$$
is well-defined and continuous. Furthermore, for each $p\in[1,\infty]$ there exist $C\geq 1$ and $k\in\ZZ_+$ such that
\begin{equation}\label{uni-bound-for-cont-mml2}
\|\VV_{\boldsymbol{\psi}}\VV^*_{\boldsymbol{\varphi}} G\|_{L^p_{\eta}(W\times W,dv_gd\lambda)}\leq C \|\boldsymbol{\psi}\|^{(k)}_{g,r'_0}\|\boldsymbol{\varphi}\|^{(k)}_{g,r'_0} \|G\|_{L^p_{\eta}(W\times W,dv_gd\lambda)},
\end{equation}
for all $G\in L^p_{\eta}(W\times W,dv_gd\lambda)$, $\boldsymbol{\varphi}, \boldsymbol{\psi}\in \Conf_g(W;r'_0)$.
\end{proposition}

\begin{proof} Let $\tilde{\tau}\geq 0$ be a temperance constant for $\eta$. Let $\boldsymbol{\varphi},\boldsymbol{\psi}\in\Conf_g(W;r'_0)$ and set $\varphi_X:=\boldsymbol{\varphi}(X)$, $\psi_X:=\boldsymbol{\psi}(X)$, $X\in W$. In view of Proposition \ref{lemma-for-conti-of-stft-sympl} $(i)$ and Proposition \ref{lemma-for-the-adj-of-stft-sympl} $(i)$,
\begin{align}\label{conti-com-st-adj}
\VV_{\boldsymbol{\psi}}\VV^*_{\boldsymbol{\varphi}}: \lim_{\substack{\longrightarrow \\ s\rightarrow\infty}} L^1_{(1+|\cdot|)^{-s}}(W\times W)\rightarrow \lim_{\substack{\longrightarrow \\ s\rightarrow \infty}} L^{\infty}_{(1+|\cdot|)^{-s}}(W\times W)
\end{align}
is well-defined and continuous. For $G\in \ds\lim_{\substack{\longrightarrow \\ s\rightarrow\infty}} L^1_{(1+|\cdot|)^{-s}}(W\times W)$, we infer
$$
\VV_{\boldsymbol{\psi}}\VV^*_{\boldsymbol{\varphi}}G(X,\Xi)=\langle G, |g|^{1/2}\overline{\VV_{\boldsymbol{\varphi}}(e^{2\pi i[\Xi,\cdot]}\psi_X)}\rangle.
$$
Notice that $\overline{\VV_{\boldsymbol{\varphi}}(e^{2\pi i[\Xi,\cdot]}\psi_X)(Z_1,Z_2)}=\VV_{\boldsymbol{\psi}}\varphi_{Z_1}(X,\Xi-Z_2)$ and consequently,
$$
|\VV_{\boldsymbol{\psi}}\VV^*_{\boldsymbol{\varphi}}G(X,\Xi)|\leq \int_{W\times W} |G(Z_1,\Xi-Z_2)||\VV_{\boldsymbol{\psi}}\varphi_{Z_1}(X,Z_2)||g_{Z_1}|^{1/2}dZ_1dZ_2,\;\; X,\Xi\in W.
$$
We make the following\\
\\
\noindent \textbf{Claim.} For every $N>0$ there are $C\geq 1$ and $k\in\ZZ_+$ such that
\begin{multline*}
|\VV_{\boldsymbol{\psi}}\varphi_{Z_1}(X,Z_2)||g_{Z_1}|^{1/2}\leq C \|\varphi_{Z_1}\|^{(k)}_{g_{Z_1}, U_{Z_1,r'_0}}\|\psi_X\|^{(k)}_{g_X,U_{X,r'_0}} (1+g^{\sigma}_{Z_1}(Z_2))^{-N}(1+g^{\sigma}_X(Z_2))^{-N}\\
\cdot (1+g^{\sigma}_{Z_1}(U_{X,r'_0}-U_{Z_1,r'_0}))^{-N} (1+g^{\sigma}_X(U_{X,r'_0}-U_{Z_1,r'_0}))^{-N},
\end{multline*}
for all $X,Z_1,Z_2\in W$, $\boldsymbol{\varphi},\boldsymbol{\psi}\in\Conf_g(W;r'_0)$.\\
\\
We postpone its proof for later and continue with the proof of the proposition. Let $X'\in U_{X,r'_0}$ and $Z'_1\in U_{Z_1,r'_0}$ be arbitrary. Then
\begin{align*}
\eta(X,\Xi)&\leq \tilde{C}\eta(X',\Xi)\leq \tilde{C}^2\eta(Z'_1,\Xi)(1+g^{\sigma}_{Z'_1}(X'-Z'_1))^{\tilde{\tau}}\\
&\leq C_0^{\tilde{\tau}}\tilde{C}^3 \eta(Z_1,\Xi)(1+g^{\sigma}_{Z_1}(X'-Z'_1))^{\tilde{\tau}}\\
&\leq C_0^{\tilde{\tau}}\tilde{C}^4 \eta(Z_1,\Xi-Z_2)(1+g^{\sigma}_{Z_1}(Z_2))^{\tilde{\tau}} (1+g^{\sigma}_{Z_1}(X'-Z'_1))^{\tilde{\tau}}.
\end{align*}
As $X'\in U_{X,r'_0}$ and $Z'_1\in U_{Z_1,r'_0}$ are arbitrary, we deduce
\begin{equation*}
\eta(X,\Xi)\leq C_0^{\tilde{\tau}}\tilde{C}^4 \eta(Z_1,\Xi-Z_2)(1+g^{\sigma}_{Z_1}(Z_2))^{\tilde{\tau}} (1+g^{\sigma}_{Z_1}(U_{X,r'_0}-U_{Z_1,r'_0}))^{\tilde{\tau}}.
\end{equation*}
We employ the bounds in the Claim to infer that there exists $k\in\ZZ_+$ such that
\begin{align}
|&\VV_{\boldsymbol{\psi}}\VV^*_{\boldsymbol{\varphi}}G(X,\Xi)|\eta(X,\Xi)\nonumber\\
&\leq C_1\|\boldsymbol{\varphi}\|^{(k)}_{g,r'_0}\|\psi_X\|^{(k)}_{g_X,U_{X,r'_0}} \int_{W\times W} |G(Z_1,\Xi-Z_2)|\eta(Z_1,\Xi-Z_2)\nonumber\\
&\qquad\cdot (1+g^{\sigma}_{Z_1}(Z_2))^{-n-1}(1+g^{\sigma}_X(Z_2))^{-n-1} (1+g^{\sigma}_{Z_1}(U_{X,r'_0}-U_{Z_1,r'_0}))^{-(2n+1)(N_0+1)}\nonumber \\
&\qquad \cdot (1+g^{\sigma}_X(U_{X,r'_0}-U_{Z_1,r'_0}))^{-(2n+1)(N_0+1)}dZ_1dZ_2,\label{bounds-for-gen}
\end{align}
for all $X,\Xi\in W$. Let $X'\in U_{X,r'_0}$ and $Z'_1\in U_{Z_1,r'_0}$ be arbitrary. We have
\begin{align*}
g_X(X-Z_1)&\leq 3g_X(X-X')+3g_X(Z'_1-Z_1)+3g_X(X'-Z'_1)\\
&\leq 3r'^2_0+3C_0g_{X'}(Z'_1-Z_1)+3g^{\sigma}_X(X'-Z'_1)\\
&\leq 3r'^2_0+3C_0^2g_{Z'_1}(Z'_1-Z_1)(1+g^{\sigma}_{X'}(X'-Z'_1))^{N_0} +3g^{\sigma}_X(X'-Z'_1)\\
&\leq 3r'^2_0+3r'^2_0C_0^{N_0+3}(1+g^{\sigma}_X(X'-Z'_1))^{N_0} +3g^{\sigma}_X(X'-Z'_1),
\end{align*}
and consequently,
\begin{align}\label{bounds-metric-x}
g_X(X-Z_1)\leq 3(2r'^2_0+1)C_0^{N_0+3}(1+g^{\sigma}_X(U_{X,r'_0}-U_{Z_1,r'_0}))^{N_0+1}.
\end{align}
Analogously,
\begin{align}\label{bounds-metric-z1}
g_{Z_1}(X-Z_1)\leq 3(2r'^2_0+1)C_0^{N_0+3}(1+g^{\sigma}_{Z_1}(U_{X,r'_0}-U_{Z_1,r'_0}))^{N_0+1}.
\end{align}
First we prove the claim in the proposition when $p=1$ and $p=\infty$. When $p=\infty$, \eqref{bounds-for-gen} and \eqref{bounds-metric-x} together with the identity $|g_X||g^{\sigma}_X|=1$ imply
\begin{align*}
|&\VV_{\boldsymbol{\psi}}\VV^*_{\boldsymbol{\varphi}}G(X,\Xi)|\eta(X,\Xi)\\
&\leq C_2\|\boldsymbol{\varphi}\|^{(k)}_{g,r'_0}\|\psi_X\|^{(k)}_{g_X,U_{X,r'_0}} \|G\|_{L^{\infty}_{\eta}(W\times W)} \int_{W\times W} \frac{|g_X|^{1/2}|g^{\sigma}_X|^{1/2}dZ_1dZ_2} {(1+g^{\sigma}_X(Z_2))^{n+1}(1+g_X(X-Z_1))^{n+1}}\\
&= C_2\|\boldsymbol{\varphi}\|^{(k)}_{g,r'_0}\|\psi_X\|^{(k)}_{g_X,U_{X,r'_0}} \|G\|_{L^{\infty}_{\eta}(W\times W)} \int_W\frac{|g_X|^{1/2}dZ_1}{(1+g_X(Z_1))^{n+1}} \int_W\frac{|g^{\sigma}_X|^{1/2}dZ_2}{(1+g^{\sigma}_X(Z_2))^{n+1}}\\
&\leq C_3 \|\boldsymbol{\varphi}\|^{(k)}_{g,r'_0}\|\psi_X\|^{(k)}_{g_X,U_{X,r'_0}} \|G\|_{L^{\infty}_{\eta}(W\times W)},
\end{align*}
which proves the claim in the proposition when $p=\infty$. Next, we consider the case $p=1$. For $X'\in U_{X,r'_0}$, in view of \eqref{ineq-for-metric-p-5} and \eqref{ineq-for-metric-p-6}, we have
$$
|g_X|\leq C_0^{6n+2nN_0}|g_{Z_1}|(1+g^{\sigma}_{Z_1}(X'-U_{Z_1,r'_0}))^{2nN_0},
$$
hence
\begin{equation}\label{bound-for-v-jj}
|g_X|\leq C_0^{2nN_0+6n} |g_{Z_1}|(1+g^{\sigma}_{Z_1}(U_{X,r'_0}-U_{Z_1,r'_0}))^{2nN_0}.
\end{equation}
Now, \eqref{bounds-for-gen} and \eqref{bounds-metric-z1} imply
\begin{multline}\label{ine-for-the-case-1-b}
\|\VV_{\boldsymbol{\psi}}\VV^*_{\boldsymbol{\varphi}}G\|_{L^1_{\eta}(W\times W,dv_gd\lambda)}\\
\leq C'_2\|\boldsymbol{\varphi}\|^{(k)}_{g,r'_0}\|\boldsymbol{\psi}\|^{(k)}_{g,r'_0} \int_{W\times W} |G(Z_1,Z_2)|\eta(Z_1,Z_2)F(Z_1,Z_2)|g_{Z_1}|^{1/2}dZ_1dZ_2,
\end{multline}
where
$$
F(Z_1,Z_2)=\int_{W\times W} (1+g^{\sigma}_{Z_1}(\Xi-Z_2))^{-n-1}(1+g_{Z_1}(X-Z_1))^{-n-1}dXd\Xi.
$$
Employing $|g_{Z_1}||g^{\sigma}_{Z_1}|=1$, similarly as above, we infer that $F\in L^{\infty}(W\times W)$ and the bound \eqref{ine-for-the-case-1-b} proves the claim in the proposition when $p=1$. We prove the case when $p\in (1,\infty)$ by interpolation. Fix $p\in(1,\infty)$ and notice that the above implies that
\begin{align}
&\VV_{\boldsymbol{\psi}}\VV^*_{\boldsymbol{\varphi}}: L^1(W\times W,d\mu_p)\rightarrow L^1(W\times W,d\mu_p)\quad \mbox{and}\label{boun-ope-for-int}\\
&\VV_{\boldsymbol{\psi}}\VV^*_{\boldsymbol{\varphi}}: L^{\infty}(W\times W,d\mu_p)\rightarrow L^{\infty}(W\times W,d\mu_p)\label{boun-ope-for-int-anoth}
\end{align}
are well-defined and continuous, where $d\mu_p$ is the measure $\eta^pdv_gd\lambda$ (in view of Lemma \ref{admisibil-weig-mul}, $\eta^p$ is uniformly admissible with a slow variation constant $\tilde{r}$); notice that $L^{\infty}(W\times W,d\mu_p)=L^{\infty}(W\times W,dv_gd\lambda)=L^{\infty}(W\times W)$. Now, in view of \eqref{conti-com-st-adj}, the claim in the proposition follows from the Riesz-Thorin interpolation theorem \cite[Theorem 1.3.4, p. 37]{Grafakos}; \eqref{uni-bound-for-cont-mml2} follows from the bounds we showed for \eqref{boun-ope-for-int} and \eqref{boun-ope-for-int-anoth} and the bounds in the Riesz-Thorin theorem.\\
\indent It remains to prove the Claim. Let $N\in\ZZ_+$ be arbitrary and pick $N_1\in\ZZ_+$ such that $N_1\geq N(2N_0+1)+n+1$. For every fixed $Z_1,Z_2\in W$, there exists $\theta=\theta(Z_1,Z_2)\in W$ such that $g_{Z_1}(\theta)=1$ and $[Z_2,\theta]=g^{\sigma}_{Z_1}(Z_2)^{1/2}$. Then
$$
(1-(2\pi i)^{-1}\partial_{\theta;Y})^{2N}e^{-2\pi i[Z_2,Y]}= (1+g^{\sigma}_{Z_1}(Z_2)^{1/2})^{2N}e^{-2\pi i[Z_2,Y]}.
$$
We estimate as follows
\begin{align*}
(1+&g^{\sigma}_{Z_1}(Z_2))^N|\VV_{\boldsymbol{\psi}}\varphi_{Z_1}(X,Z_2)||g_{Z_1}|^{1/2}\\
&\leq C'_1|g_{Z_1}|^{1/2}\sum_{N'+N''\leq 2N}\int_W |\partial^{N'}_{\theta;Y}\varphi_{Z_1}(Y)||\partial^{N''}_{\theta;Y}\psi_X(Y)| dY\\
&\leq C'_1\|\varphi_{Z_1}\|^{(2N_1)}_{g_{Z_1}, U_{Z_1,r'_0}}\|\psi_X\|^{(2N_1)}_{g_X,U_{X,r'_0}}\\
&{}\quad\cdot\sum_{N'+N''\leq 2N}\int_W \frac{g_X(\theta)^{N''/2}|g_{Z_1}|^{1/2}} {(1+g^{\sigma}_{Z_1}(Y-U_{Z_1,r'_0}))^{N_1}(1+g^{\sigma}_X(Y-U_{X,r'_0}))^{N_1}}dY.
\end{align*}
Applying \eqref{ineq-for-metric-p-1} twice, we infer (recall, $g_{Z_1}(\theta)=1$)
\begin{align*}
g_X(\theta)&\leq C_0^{N_0+2} g_Y(\theta)(1+g^{\sigma}_X(Y-U_{X,r'_0}))^{N_0}\\
&\leq C_0^{2N_0+4} (1+g^{\sigma}_{Z_1}(Y-U_{Z_1,r'_0}))^{N_0}(1+g^{\sigma}_X(Y-U_{X,r'_0}))^{N_0}.
\end{align*}
Hence,
\begin{multline}\label{bounds-for-int-metrr}
(1+g^{\sigma}_{Z_1}(Z_2))^N|\VV_{\boldsymbol{\psi}}\varphi_{Z_1}(X,Z_2)||g_{Z_1}|^{1/2} \leq C'_2 \|\varphi_{Z_1}\|^{(2N_1)}_{g_{Z_1}, U_{Z_1,r'_0}}\|\psi_X\|^{(2N_1)}_{g_X,U_{X,r'_0}}\\
\cdot\int_W (1+g^{\sigma}_{Z_1}(Y-U_{Z_1,r'_0}))^{-N_1+NN_0} (1+g^{\sigma}_X(Y-U_{X,r'_0}))^{-N(N_0+1)}|g_{Z_1}|^{1/2}dY.
\end{multline}
Let $X'\in U_{X,r'_0}$ and $Z'_1\in U_{Z_1,r'_0}$. Then
\begin{align*}
g^{\sigma}_{Z_1}(X'-Z'_1)\leq 2C_0g^{\sigma}_{Z'_1}(X'-Y)+2g^{\sigma}_{Z_1}(Y-Z'_1).
\end{align*}
We estimate the first term as follows
\begin{align*}
g^{\sigma}_{Z'_1}(X'-Y)&\leq C_0g^{\sigma}_Y(X'-Y)(1+g^{\sigma}_{Z'_1}(Y-Z'_1))^{N_0}\\
&\leq C_0^2g^{\sigma}_{X'}(X'-Y)(1+g^{\sigma}_{X'}(X'-Y))^{N_0} (1+g^{\sigma}_{Z'_1}(Y-Z'_1))^{N_0}\\
&\leq C_0^{2N_0+3}g^{\sigma}_X(X'-Y)(1+g^{\sigma}_X(X'-Y))^{N_0} (1+g^{\sigma}_{Z_1}(Y-Z'_1))^{N_0}.
\end{align*}
Consequently,
$$
1+g^{\sigma}_{Z_1}(X'-Z'_1)\leq 2C_0^{2N_0+4}(1+g^{\sigma}_X(X'-Y))^{N_0+1}(1+g^{\sigma}_{Z_1}(Y-Z'_1))^{N_0+1}.
$$
As $X'\in U_{X,r'_0}$ and $Z'_1\in U_{Z_1,r'_0}$ are arbitrary, we deduce
\begin{multline*}
(1+g^{\sigma}_{Z_1}(U_{X,r'_0}-U_{Z_1,r'_0}))^N\\
\leq 2^NC_0^{N(2N_0+4)}(1+g^{\sigma}_X(U_{X,r'_0}-Y))^{N(N_0+1)} (1+g^{\sigma}_{Z_1}(Y-U_{Z_1,r'_0}))^{N(N_0+1)}.
\end{multline*}
Employing this bound in \eqref{bounds-for-int-metrr}, we infer
\begin{align*}
(1+&g^{\sigma}_{Z_1}(Z_2))^N(1+g^{\sigma}_{Z_1}(U_{X,r'_0}-U_{Z_1,r'_0}))^N |\VV_{\boldsymbol{\psi}}\varphi_{Z_1}(X,Z_2)||g_{Z_1}|^{1/2}\\
&\leq C'_3 \|\varphi_{Z_1}\|^{(2N_1)}_{g_{Z_1}, U_{Z_1,r'_0}}\|\psi_X\|^{(2N_1)}_{g_X,U_{X,r'_0}} \int_W \frac{|g_{Z_1}|^{1/2}dY} {(1+g^{\sigma}_{Z_1}(Y-U_{Z_1,r'_0}))^{n+1}}\\
& \leq C'_4\|\varphi_{Z_1}\|^{(2N_1)}_{g_{Z_1}, U_{Z_1,r'_0}}\|\psi_X\|^{(2N_1)}_{g_X,U_{X,r'_0}}\int_W \frac{|g_{Z_1}|^{1/2}dY} {(1+g_{Z_1}(Y-Z_1))^{n+1}}.
\end{align*}
Since the very last integral is uniformly bounded by a constant for all $Z_1\in W$, we deduce the following bounds: for every $N>0$ there are $C\geq 1$ and $k\in\ZZ_+$ such that
\begin{multline}\label{intme-bound-for-theshorttime-sympl}
|\VV_{\boldsymbol{\psi}}\varphi_{Z_1}(X,Z_2)||g_{Z_1}|^{1/2}\leq C \|\varphi_{Z_1}\|^{(k)}_{g_{Z_1}, U_{Z_1,r'_0}}\|\psi_X\|^{(k)}_{g_X,U_{X,r'_0}}\\
(1+g^{\sigma}_{Z_1}(Z_2))^{-N}(1+g^{\sigma}_{Z_1}(U_{X,r'_0}-U_{Z_1,r'_0}))^{-N},
\end{multline}
for all $X,Z_1,Z_2\in W$, $\boldsymbol{\varphi},\boldsymbol{\psi}\in\Conf_g(W;r'_0)$. We claim that \eqref{intme-bound-for-theshorttime-sympl} implies the desired bounds. To see this, let $X'\in U_{X,r'_0}$ and $Z'_1\in U_{Z_1,r'_0}$ and notice that
\begin{align*}
1+g^{\sigma}_X(Z_2)&\leq 1+C_0g^{\sigma}_{X'}(Z_2)\leq 1+C_0^2g^{\sigma}_{Z'_1}(Z_2)(1+g^{\sigma}_{Z'_1}(X'-Z'_1))^{N_0}\\
&\leq 1+C_0^{N_0+3}g^{\sigma}_{Z_1}(Z_2)(1+g^{\sigma}_{Z_1}(X'-Z'_1))^{N_0},
\end{align*}
which gives
\begin{equation}\label{inequ-met-1}
1+g^{\sigma}_X(Z_2)\leq C_0^{N_0+3}(1+g^{\sigma}_{Z_1}(Z_2))(1+g^{\sigma}_{Z_1}(U_{X,r'_0}-U_{Z_1,r'_0}))^{N_0}.
\end{equation}
Similarly, for $X'\in U_{X,r'_0}$ and $Z'_1\in U_{Z_1,r'_0}$, we infer
\begin{align*}
1+g^{\sigma}_X(X'-Z'_1)&\leq 1+C_0g^{\sigma}_{X'}(X'-Z'_1)\leq 1+C^2_0g^{\sigma}_{Z'_1}(X'-Z'_1)(1+g^{\sigma}_{Z'_1}(X'-Z'_1))^{N_0}\\
&\leq 1+C^{N_0+3}_0g^{\sigma}_{Z_1}(X'-Z'_1)(1+g^{\sigma}_{Z_1}(X'-Z'_1))^{N_0},
\end{align*}
whence
\begin{equation}\label{inequ-met-2}
1+g^{\sigma}_X(U_{X,r'_0}-U_{Z_1,r'_0})\leq C^{N_0+3}_0(1+g^{\sigma}_{Z_1}(U_{X,r'_0}-U_{Z_1,r'_0}))^{N_0+1}.
\end{equation}
Combining \eqref{inequ-met-1} and \eqref{inequ-met-2} with \eqref{intme-bound-for-theshorttime-sympl}, we deduce the desired bounds in the Claim.
\end{proof}

\begin{proof}[Proof of Theorem \ref{com-adj-stft-sympl1}] The case when $p=q\in[1,\infty]$ follows from Proposition \ref{com-adj-stft-sympl}. We divide the proof of the rest in seven cases. Let $\boldsymbol{\varphi},\boldsymbol{\psi}\in\Conf_g(W;r'_0)$ and set $\varphi_X:=\boldsymbol{\varphi}(X)$, $\psi_X:=\boldsymbol{\psi}(X)$, $X\in W$.\\
\indent \underline{Case 1: $p=\infty$, $q\in[1,\infty)$.} Let $G\in \widetilde{L}^{\infty,q}_{\eta}(W\times W,dv_gd\lambda)$. Then, \eqref{bounds-for-gen} together with the Minkowski integral inequality implies
\begin{multline*}
\left(\int_W|\VV_{\boldsymbol{\psi}}\VV^*_{\boldsymbol{\varphi}}G(X,\Xi)|^q \eta(X,\Xi)^qd\Xi\right)^{1/q}\leq C_1 \|\boldsymbol{\varphi}\|^{(k)}_{g,r'_0}\|\psi_X\|^{(k)}_{g_X,U_{X,r'_0}} \|G\|_{\widetilde{L}^{\infty,q}_{\eta}(W\times W,dv_gd\lambda)}\\
\cdot\int_{W\times W} (1+g^{\sigma}_X(Z_2))^{-n-1} (1+g^{\sigma}_X(U_{X,r'_0}-U_{Z_1,r'_0}))^{-(2n+1)(N_0+1)}dZ_1dZ_2,
\end{multline*}
a.a. $X$. Employing $|g_X||g^{\sigma}_X|=1$ and \eqref{bounds-metric-x}, by the same technique as in the proof of Proposition \ref{com-adj-stft-sympl} one can show that the integral on the right is uniformly bounded by a constant for all $X\in W$, which completes the proof of Case 1.\\
\indent \underline{Case 2: $p=1$, $q=\infty$.} Let $G\in \widetilde{L}^{1,\infty}_{\eta}(W\times W, dv_gd\lambda)$. Then \eqref{bounds-for-gen} together with $|g_{Z_1}||g^{\sigma}_{Z_1}|=1$ implies
\begin{multline*}
|\VV_{\boldsymbol{\psi}}\VV^*_{\boldsymbol{\varphi}}G(X,\Xi)|\eta(X,\Xi)\leq C_1\|\boldsymbol{\varphi}\|^{(k)}_{g,r'_0}\|\psi_X\|^{(k)}_{g_X,U_{X,r'_0}}\\
\cdot \int_{W\times W} \frac{|G(Z_1,\Xi-Z_2)|\eta(Z_1,\Xi-Z_2)|g_{Z_1}|^{1/2} |g^{\sigma}_{Z_1}|^{1/2}} {(1+g^{\sigma}_{Z_1}(Z_2))^{n+1} (1+g^{\sigma}_{Z_1}(U_{X,r'_0}-U_{Z_1,r'_0}))^{(2n+1)(N_0+1)}}dZ_1dZ_2.
\end{multline*}
Set $F(Z_1):=\esssup_{Z_2\in W}|G(Z_1,Z_2)|\eta(Z_1,Z_2)$. Notice that \eqref{bound-for-v-jj} implies
$$
|g^{\sigma}_{Z_1}|\leq C_0^{2nN_0+6n}|g^{\sigma}_X|(1+g^{\sigma}_{Z_1}(U_{X,r'_0}-U_{Z_1,r'_0}))^{2nN_0}.
$$
In view of \eqref{bounds-metric-z1}, we deduce
\begin{multline*}
\esssup_{\Xi\in W} |\VV_{\boldsymbol{\psi}}\VV^*_{\boldsymbol{\varphi}}G(X,\Xi)|\eta(X,\Xi)\leq C'_2 \|\boldsymbol{\varphi}\|^{(k)}_{g,r'_0}\|\psi_X\|^{(k)}_{g_X,U_{X,r'_0}}\\
\cdot |g^{\sigma}_X|^{1/2}\int_W \frac{F(Z_1)|g_{Z_1}|}{(1+g_{Z_1}(X-Z_1))^{n+1}}\left(\int_W \frac{|g^{\sigma}_{Z_1}|^{1/2}dZ_2} {(1+g^{\sigma}_{Z_1}(Z_2))^{n+1}}\right) dZ_1,
\end{multline*}
a.a. $X$. Consequently,
\begin{align*}
\|&\VV_{\boldsymbol{\psi}}\VV^*_{\boldsymbol{\varphi}} G\|_{\widetilde{L}^{1,\infty}_{\eta}(W\times W,dv_gd\lambda)}\\
&\leq C'_3\|\boldsymbol{\varphi}\|^{(k)}_{g,r'_0}\|\boldsymbol{\psi}\|^{(k)}_{g,r'_0} \int_W F(Z_1)|g_{Z_1}|^{1/2}\left(\int_W\frac{|g_{Z_1}|^{1/2}dX} {(1+g_{Z_1}(X-Z_1))^{n+1}}\right)dZ_1\\
&\leq C'_4 \|\boldsymbol{\varphi}\|^{(k)}_{g,r'_0}\|\boldsymbol{\psi}\|^{(k)}_{g,r'_0} \|G\|_{\widetilde{L}^{1,\infty}_{\eta}(W\times W,dv_gd\lambda)},
\end{align*}
which completes the proof of Case 2.\\
\indent Before we consider the remaining cases, we introduce the following notations. For any positive measurable functions $w$ on $W\times W$ which satisfies \eqref{pol-b-d-wei-s}, denote by $P_w$ the topological isomorphism
$$
P_w:\lim_{\substack{\longrightarrow \\ s\rightarrow\infty}} L^1_{(1+|\cdot|)^{-s}}(W\times W)\rightarrow \lim_{\substack{\longrightarrow \\ s\rightarrow\infty}} L^1_{(1+|\cdot|)^{-s}}(W\times W),\quad P_w(F)=Fw;
$$
its inverse is $P_{1/w}$. Furthermore, we denote by $Q$ the topological isomorphism
$$
Q:\lim_{\substack{\longrightarrow \\ s\rightarrow\infty}} L^1_{(1+|\cdot|)^{-s}}(W\times W)\rightarrow \lim_{\substack{\longrightarrow \\ s\rightarrow\infty}} L^1_{(1+|\cdot|)^{-s}}(W\times W),\quad Q(F)(X,\Xi)=F(\Xi,X);
$$
clearly $Q\circ Q=\operatorname{Id}$. For any measurable $w_1:W\times W\rightarrow (0,\infty)$ which satisfies \eqref{pol-b-d-wei-s}, $P_w$ and $Q$ restrict to the following bijective isometries for all $p,q\in[1,\infty]$:
\begin{gather*}
P_w: L^{p,q}_{w_1}(W\times W)\rightarrow L^{p,q}_{w_1/w}(W\times W),\quad P_w: \widetilde{L}^{p,q}_{w_1}(W\times W)\rightarrow \widetilde{L}^{p,q}_{w_1/w}(W\times W),\\
Q:L^{p,q}_{w_1}(W\times W)\rightarrow \widetilde{L}^{q,p}_{Qw_1}(W\times W).
\end{gather*}
\indent \underline{Case 3: $p\in(1,\infty)$, $q=1$.} The proof is by interpolation. Set
$$
\eta_1(X,\Xi):=|g_X|^{\frac{1}{2p}-\frac{1}{2}}\eta(X,\Xi),\quad \eta_2(X,\Xi):=|g_X|^{1/(2p)}\eta(X,\Xi),\qquad X,\Xi\in W.
$$
Lemma \ref{admisibil-weig-mul} and Remark \ref{rem-for-adm-func-for-fromad} verify that $\eta_1$ and $\eta_2$ are uniformly admissible with slow variation constant $\min\{r_0,\tilde{r}\}$. In view of Case 1 and Proposition \ref{com-adj-stft-sympl},
\begin{align}
&\VV_{\boldsymbol{\psi}}\VV^*_{\boldsymbol{\varphi}}:\widetilde{L}^{1,1}_{\eta_1}(W\times W, dv_gd\lambda)\rightarrow \widetilde{L}^{1,1}_{\eta_1}(W\times W, dv_gd\lambda)\quad \mbox{and} \label{boun-interpol-firs}\\
&\VV_{\boldsymbol{\psi}}\VV^*_{\boldsymbol{\varphi}}: \widetilde{L}^{\infty,1}_{\eta_2}(W\times W, dv_gd\lambda)\rightarrow \widetilde{L}^{\infty,1}_{\eta_2}(W\times W, dv_gd\lambda)\label{bound-interpol-anoth-map}
\end{align}
are well-defined and continuous. Consequently,
\begin{align*}
&QP_{\eta_2}(\VV_{\boldsymbol{\psi}}\VV^*_{\boldsymbol{\varphi}})P_{1/\eta_2}Q: L^{1,1}(W\times W)\rightarrow L^{1,1}(W\times W)\quad \mbox{and}\\
&QP_{\eta_2}(\VV_{\boldsymbol{\psi}}\VV^*_{\boldsymbol{\varphi}})P_{1/\eta_2}Q: L^{1,\infty}(W\times W)\rightarrow L^{1,\infty}(W\times W)
\end{align*}
are well-defined and continuous. Now, the Riesz-Thorin interpolation theorem for $L^{p,q}$-spaces \cite[Section 7, Theorem 2]{ben-pen} implies that $QP_{\eta_2}(\VV_{\boldsymbol{\psi}}\VV^*_{\boldsymbol{\varphi}})P_{1/\eta_2}Q: L^{1,p}(W\times W)\rightarrow L^{1,p}(W\times W)$ is well-defined and continuous, which, in turn, implies the claim in the theorem in view of the identity
\begin{align}\label{equ-for-com-via-m}
\VV_{\boldsymbol{\psi}}\VV^*_{\boldsymbol{\varphi}}= P_{1/\eta_2}Q(QP_{\eta_2}(\VV_{\boldsymbol{\psi}} \VV^*_{\boldsymbol{\varphi}})P_{1/\eta_2}Q)QP_{\eta_2};
\end{align}
\eqref{uni-bound-for-cont-mml} follows from the bounds that we proved for \eqref{boun-interpol-firs} and \eqref{bound-interpol-anoth-map} in Proposition \ref{com-adj-stft-sympl} and Case 1 and the bounds in the Riesz-Thorin theorem \cite[Section 7, Theorem 2]{ben-pen}.
\indent \underline{Case 4: $p=1$, $q\in(1,\infty)$.} The proof is similar to the proof of Case 3.\\
\indent \underline{Case 5: $1<q<p<\infty$.} Let $\eta_2$ be as in Case 3. In view of Case 3 and Proposition \ref{com-adj-stft-sympl}, we infer that
\begin{align}
&\VV_{\boldsymbol{\psi}}\VV^*_{\boldsymbol{\varphi}}:\widetilde{L}^{p,1}_{\eta}(W\times W, dv_gd\lambda)\rightarrow \widetilde{L}^{p,1}_{\eta}(W\times W, dv_gd\lambda)\quad \mbox{and} \label{inter-bound-first-map-adv}\\
&\VV_{\boldsymbol{\psi}}\VV^*_{\boldsymbol{\varphi}}:\widetilde{L}^{p,p}_{\eta}(W\times W, dv_gd\lambda)\rightarrow \widetilde{L}^{p,p}_{\eta}(W\times W, dv_gd\lambda)\label{interp-bound-secon-map-adit-elm}
\end{align}
are well-defined and continuous and hence,
\begin{align*}
&QP_{\eta_2}(\VV_{\boldsymbol{\psi}}\VV^*_{\boldsymbol{\varphi}})P_{1/\eta_2}Q: L^{1,p}(W\times W)\rightarrow L^{1,p}(W\times W)\quad \mbox{and}\\
&QP_{\eta_2}(\VV_{\boldsymbol{\psi}}\VV^*_{\boldsymbol{\varphi}})P_{1/\eta_2}Q: L^{p,p}(W\times W)\rightarrow L^{p,p}(W\times W)
\end{align*}
are well-defined and continuous. The Riesz-Thorin interpolation theorem \cite[Section 7, Theorem 2]{ben-pen} implies that $QP_{\eta_2}(\VV_{\boldsymbol{\psi}}\VV^*_{\boldsymbol{\varphi}})P_{1/\eta_2}Q: L^{q,p}(W\times W)\rightarrow L^{q,p}(W\times W)$ is well-defined and continuous. In view of \eqref{equ-for-com-via-m}, this yields the claim; \eqref{uni-bound-for-cont-mml} follows from the bounds we proved for \eqref{inter-bound-first-map-adv} and \eqref{interp-bound-secon-map-adit-elm} in Case 3 and Proposition \ref{com-adj-stft-sympl} and the bounds in the Riesz-Thorin theorem \cite[Section 7, Theorem 2]{ben-pen}.\\
\indent \underline{Case 6: $1<p<q<\infty$.} In view of Case 4 and Proposition \ref{com-adj-stft-sympl},
\begin{align*}
&\VV_{\boldsymbol{\psi}}\VV^*_{\boldsymbol{\varphi}}:\widetilde{L}^{1,q}_{\eta_1}(W\times W, dv_gd\lambda)\rightarrow \widetilde{L}^{1,q}_{\eta_1}(W\times W, dv_gd\lambda)\quad \mbox{and}\\
&\VV_{\boldsymbol{\psi}}\VV^*_{\boldsymbol{\varphi}}:\widetilde{L}^{q,q}_{\eta_3}(W\times W, dv_gd\lambda)\rightarrow \widetilde{L}^{q,q}_{\eta_3}(W\times W, dv_gd\lambda)
\end{align*}
are well-defined and continuous with $\eta_1$ as in Case 3 and $\eta_3(X,\Xi):=|g_X|^{\frac{1}{2p}-\frac{1}{2q}}\eta(X,\Xi)$, $X,\Xi\in W$ (cf. Lemma \ref{admisibil-weig-mul} and Remark \ref{rem-for-adm-func-for-fromad}). The rest of the proof is analogous to the proof of Case 5; of course, now one interpolates between $L^{q,1}$ and $L^{q,q}$.\\
\indent \underline{Case 7: $p\in(1,\infty)$, $q=\infty$.} We claim that
\begin{align}\label{st-ad-dual-a-a}
\langle \VV_{\boldsymbol{\psi}}\VV^*_{\boldsymbol{\varphi}}G,\chi\rangle=\int_{W\times W} G(X,\Xi)\overline{\VV_{\boldsymbol{\varphi}}\VV^*_{\boldsymbol{\psi}} (|g|^{-1/2}\overline{\chi})(X,\Xi)} dv_g(X)d\Xi,
\end{align}
for all $G\in \ds\lim_{\substack{\longrightarrow \\ s\rightarrow\infty}} L^1_{(1+|\cdot|)^{-s}}(W\times W)$, $\chi\in\SSS(W\times W)$. Notice that the integral on the right is absolutely convergent in view of Proposition \ref{lemma-for-conti-of-stft-sympl} $(ii)$ and Proposition \ref{lemma-for-the-adj-of-stft-sympl} $(iii)$. It is straightforward to check that \eqref{st-ad-dual-a-a} holds when $G\in \SSS(W\times W)$ and the general case follows by density. Let $G\in \widetilde{L}^{p,\infty}_{\eta}(W\times W,dv_gd\lambda)$. Denote by $p'\in(1,\infty)$ the H\"older conjugate index to $p$ and let $\eta_2$ be as in Case 3. In view of Case 3 and \eqref{st-ad-dual-a-a}, we have the following chain of inequalities for all $\chi\in\SSS(W\times W)$ (Lemma \ref{admisibil-weig-mul} verifies that $1/\eta$ is uniformly admissible with slow variation constant $\tilde{r}$):
\begin{align*}
|\langle \VV_{\boldsymbol{\psi}}\VV^*_{\boldsymbol{\varphi}}G,\chi\rangle|&\leq \|G\|_{\widetilde{L}^{p,\infty}_{\eta}(W\times W,dv_gd\lambda)} \|\VV_{\boldsymbol{\varphi}}\VV^*_{\boldsymbol{\psi}} (|g|^{-1/2}\overline{\chi})\|_{\widetilde{L}^{p',1}_{1/\eta} (W\times W,dv_gd\lambda)}\\
&\leq C''\|\boldsymbol{\varphi}\|^{(k)}_{g,r'_0}\|\boldsymbol{\psi}\|^{(k)}_{g,r'_0} \|G\|_{\widetilde{L}^{p,\infty}_{\eta}(W\times W,dv_gd\lambda)} \||g|^{-1/2}\chi\|_{\widetilde{L}^{p',1}_{1/\eta} (W\times W,dv_gd\lambda)}\\
&= C''\|\boldsymbol{\varphi}\|^{(k)}_{g,r'_0}\|\boldsymbol{\psi}\|^{(k)}_{g,r'_0} \|G\|_{\widetilde{L}^{p,\infty}_{\eta}(W\times W,dv_gd\lambda)} \|\chi\|_{\widetilde{L}^{p',1}_{1/\eta_2} (W\times W)}.
\end{align*}
Since the strong dual of $\widetilde{L}^{p',1}_{1/\eta_2}(W\times W)$ is $\widetilde{L}^{p,\infty}_{\eta_2}(W\times W)= \widetilde{L}^{p,\infty}_{\eta}(W\times W,dv_gd\lambda)$, the claim follows.
\end{proof}

We are now ready to define the generalised modulation spaces. Let $\eta$ be a uniformly admissible weight with respect to $g$ with slow variation constant $\tilde{r}>0$ and let $\boldsymbol{\varphi}$ be a non-degenerate element of $\Conf_g(W;\min\{r_0,\tilde{r}\})$. For $p,q\in[1,\infty]$, we define the geometric modulation space $\widetilde{\MM}^{p,q}_{\eta}(W)$ by
$$
\widetilde{\MM}^{p,q}_{\eta}(W):= \{f\in\SSS'(W)\, |\, \VV_{\boldsymbol{\varphi}}f\in \widetilde{L}^{p,q}_{\eta}(W\times W, dv_gd\lambda)\}
$$
and we equip it with the norm $\|f\|_{\widetilde{\MM}^{p,q}_{\eta}}:= \|\VV_{\boldsymbol{\varphi}}f\|_{\widetilde{L}^{p,q}_{\eta}(W\times W,dv_gd\lambda)}$; the latter is indeed a norm in view of \eqref{compo-stft-with-adj-symp} and Lemma \ref{rem-about-part-of-unity} $(iii)$.

\begin{theorem}\label{mod-spa-donot-depend-on-partuni}
Let $p,q\in[1,\infty]$ and let $\eta$ be a uniformly admissible weight with respect to $g$ with slow variation constant $\tilde{r}>0$. Different choices of non-degenerate elements of $\Conf_g(W;\min\{r_0,\tilde{r}\})$ produce the same space $\widetilde{\MM}^{p,q}_{\eta}$ with equivalent norms. Furthermore, $\widetilde{\MM}^{p,q}_{\eta}$ is a Banach space and for any $\boldsymbol{\theta}\in\Conf_g(W;\min\{r_0,\tilde{r}\})$ the mappings
\begin{equation}\label{mappings-for-mod-spaces-sym-bbb}
\VV_{\boldsymbol{\theta}}: \widetilde{\MM}^{p,q}_{\eta}\rightarrow \widetilde{L}^{p,q}_{\eta}(W\times W,dv_gd\lambda)\quad \mbox{and}\quad \VV^*_{\boldsymbol{\theta}}: \widetilde{L}^{p,q}_{\eta}(W\times W,dv_gd\lambda)\rightarrow \widetilde{\MM}^{p,q}_{\eta}
\end{equation}
are well-defined and continuous.
\end{theorem}

\begin{proof} Set $r'_0:=\min\{r_0,\tilde{r}\}$. Let $\boldsymbol{\varphi}$ and $\boldsymbol{\psi}$ be two non-degenerate elements of $\Conf_g(W;r'_0)$. Lemma \ref{rem-about-part-of-unity} implies that $I_{|\boldsymbol{\varphi}|^2},1/I_{|\boldsymbol{\varphi}|^2}\in S(1,g)$ and $\boldsymbol{\psi}/I_{|\boldsymbol{\varphi}|^2}\in\Conf_g(W;r'_0)$. In view of \eqref{compo-stft-with-adj-symp}, we infer
\begin{equation}\label{equ-for-com-mod-inv}
\VV_{\boldsymbol{\psi}/I_{|\boldsymbol{\varphi}|^2}} \VV^*_{\boldsymbol{\varphi}}\VV_{\boldsymbol{\varphi}}f= \VV_{\boldsymbol{\psi}/I_{|\boldsymbol{\varphi}|^2}}(I_{|\boldsymbol{\varphi}|^2}f)= \VV_{\boldsymbol{\psi}}f,\quad f\in\SSS'(W).
\end{equation}
Hence, if $f\in\SSS'(W)$ is such that $\VV_{\boldsymbol{\varphi}}f\in \widetilde{L}^{p,q}_{\eta}(W\times W,dv_gd\lambda)$, then Theorem \ref{com-adj-stft-sympl1} together with \eqref{equ-for-com-mod-inv} imply that $\VV_{\boldsymbol{\psi}}f\in \widetilde{L}^{p,q}_{\eta}(W\times W,dv_gd\lambda)$ and
$$
\|\VV_{\boldsymbol{\psi}}f\|_{\widetilde{L}^{p,q}_{\eta}(W\times W,dv_gd\lambda)}\leq C' \|\VV_{\boldsymbol{\varphi}} f\|_{\widetilde{L}^{p,q}_{\eta}(W\times W,dv_gd\lambda)}.
$$
This proves the independence of $\widetilde{\MM}^{p,q}_{\eta}$ of the non-degenerate element of $\Conf_g(W;r'_0)$ employed as well as the equivalence of the norms induced on $\widetilde{\MM}^{p,q}_{\eta}$. Now, the second mapping in \eqref{mappings-for-mod-spaces-sym-bbb} is well-defined and continuous in view of Theorem \ref{com-adj-stft-sympl1}. To prove this for the first mapping in \eqref{mappings-for-mod-spaces-sym-bbb}, pick a non-degenerate $\boldsymbol{\varphi}\in\Conf_g(W;r'_0)$. Then, in the same way as for \eqref{equ-for-com-mod-inv}, we infer $\VV_{\boldsymbol{\theta}}f=\VV_{\boldsymbol{\theta}/I_{|\boldsymbol{\varphi}|^2}} \VV^*_{\boldsymbol{\varphi}}\VV_{\boldsymbol{\varphi}}f$, $ f\in\SSS'(W)$. Consequently, Theorem \ref{com-adj-stft-sympl1} implies that the first mapping in \eqref{mappings-for-mod-spaces-sym-bbb} is also well-defined and continuous. It remains to prove that $\widetilde{\MM}^{p,q}_{\eta}$ is a Banach space. Let $f_j$, $j\in\ZZ_+$, be a Cauchy sequence in $\widetilde{\MM}^{p,q}_{\eta}$. There exists $G\in \widetilde{L}^{p,q}_{\eta}(W\times W,dv_gd\lambda)$ such that $\VV_{\boldsymbol{\varphi}}f_j\rightarrow G$ in $\widetilde{L}^{p,q}_{\eta}(W\times W,dv_gd\lambda)$. In view of \eqref{compo-stft-with-adj-symp}, we infer
$$
f_j=\VV^*_{\boldsymbol{\varphi}/I_{|\boldsymbol{\varphi}|^2}}\VV_{\boldsymbol{\varphi}}f_j \rightarrow \VV^*_{\boldsymbol{\varphi}/I_{|\boldsymbol{\varphi}|^2}}G\in \widetilde{\MM}^{p,q}_{\eta}\quad \mbox{in}\,\, \widetilde{\MM}^{p,q}_{\eta},
$$
which completes the proof.
\end{proof}

\begin{remark}\label{con-incl-of-s-in-genmodspace}
If $\eta_1$ and $\eta_2$ are uniformly admissible and $\eta_1/\eta_2\in L^{\infty}(W\times W)$, then
$$
\SSS(W)\subseteq \widetilde{\MM}^{p,q}_{\eta_2}(W)\subseteq \widetilde{\MM}^{p,q}_{\eta_1} \subseteq \SSS'(W)\quad \mbox{continuously, for all}\,\, p,q\in[1,\infty].
$$
\end{remark}

\begin{remark}\label{rem-for-mod-spa-clas-and-geo-sympl-trs}
Let $p,q\in[1,\infty]$ and let $\eta:\RR^{4n}\rightarrow (0,\infty)$ be a measurable function which is moderate with respect to the Beurling weight $(1+|\cdot|)^{\tau}$, for some $\tau\geq 0$. Similarly as in the case of the classical modulation space $M^{p,q}_{\eta}(\RR^{2n})$, one can define the modulation space $\widetilde{M}^{p,q}_{\eta}(\RR^{2n})$ associated to $\widetilde{L}^{p,q}_{\eta}(\RR^{2n}\times \RR^{2n})$ (see \cite[Section 4.2]{D-P-P-V}); of course $M^{p,p}_{\eta}(\RR^{2n})= \widetilde{M}^{p,p}_{\eta}(\RR^{2n})$. If $g$ is the standard Euclidean metric on $\RR^{2n}$, setting $\widetilde{\eta}(X,\Xi):=\eta(X,-\sigma\Xi)$, $X,\Xi\in\RR^{2n}$, Remark \ref{rem-for-euc-metric-stand} and Remark \ref{uni-adm-for-stand-euc-metr} give
$$
\widetilde{\MM}^{p,q}_{\eta}(\RR^{2n})= \widetilde{M}^{p,q}_{\widetilde{\eta}}(\RR^{2n}),\, p,q\in[1,\infty],\quad \mbox{and}\quad \widetilde{\MM}^{p,p}_{\eta}(\RR^{2n})= M^{p,p}_{\widetilde{\eta}}(\RR^{2n}),\, p\in[1,\infty].
$$
\end{remark}

\section{Almost diagonalisation}\label{sec dia}

\subsection{Structural properties of $S(M,g)$}\label{structur}

As a first benefit of the above construction, we show that the Weyl-H\"ormander symbol class $S(M,g)$ can be represented as an intersection of weighted geometric modulation spaces. Besides being of independent interest, this is one of the key ingredients in the proof of our main result. We start by introducing the following notation. Notice that the function $W\times W\rightarrow (0,\infty)$, $(X,\Xi)\mapsto 1+g^{\sigma}_X(\Xi)$, is uniformly admissible with slow variation constant $r_0$. For $1\leq p\leq \infty$ and $s\geq 0$, define
$$
u_{p,s}:W\times W\rightarrow (0,\infty),\quad u_{p,s}(X,\Xi):=|g_X|^{\frac{1}{2}(1-\frac{1}{p})}(1+g^{\sigma}_X(\Xi))^s.
$$
In view of Remark \ref{rem-for-adm-func-for-fromad} and Lemma \ref{admisibil-weig-mul}, $u_{p,s}$ is uniformly admissible with slow variation constant $r_0$. Consequently, if $M$ is a $g$-admissible weight with slow variation constant $r$, Remark \ref{rem-for-adm-func-for-fromad} and Lemma \ref{admisibil-weig-mul} verify that $(X,\Xi)\mapsto u_{p,s}(X,\Xi)/ M(X)$ is uniformly admissible with slow variation constant $\min\{r,r_0\}$. Notice that
\begin{equation}\label{inclusion-for-mod-space-for-weight}
\widetilde{\MM}^{\infty,p}_{u_{p,s_2}/M}\subseteq \widetilde{\MM}^{\infty,p}_{u_{p,s_1}/M}\,\, \mbox{continuously, when}\,\, s_2\geq s_1.
\end{equation}

\begin{proposition}\label{rep-res-hor}
Let $M$ be a $g$-admissible weight and $1\leq p\leq \infty$. Then
\begin{equation}\label{ident-hormand-class-modspac-prl}
S(M,g)=\lim_{\substack{\longleftarrow\\ s\rightarrow \infty}}\widetilde{\MM}^{\infty,p}_{u_{p,s}/M}\quad \mbox{topologically},
\end{equation}
where the linking mappings in the projective limit are the canonical inclusions \eqref{inclusion-for-mod-space-for-weight}.
\end{proposition}

\begin{remark}\label{new151}
In view of Remark \ref{rem-for-mod-spa-clas-and-geo-sympl-trs}, this generalises \cite[Lemma 6.1]{gro-rz}.
\end{remark}

\begin{proof} For simplicity, denote the projective limit on the right in \eqref{ident-hormand-class-modspac-prl} by $\mathfrak{M}_p$. In view of Remark \ref{con-incl-of-s-in-genmodspace}, $\SSS(W)\subseteq \mathfrak{M}_p\subseteq \SSS'(W)$ continuously. Let $r'_0\leq \min\{r_0,1\}$ be a slow variation constant for $u_{p,s}/M$, $s\geq 0$, and let $\boldsymbol{\varphi}$ be the smooth non-degenerate element of $\Conf_g(W;r'_0)$ constructed in Example \ref{exi-of-good-par-off} $(ii)$; recall that $I_{\boldsymbol{\varphi}}(Y)=1$, $Y\in W$, and $\supp\varphi_X\subseteq U_{X,r'_0}$, $X\in W$, where we set $\varphi_X:=\boldsymbol{\varphi}(X)$, $X\in W$. We employ $\boldsymbol{\varphi}$ to define the norm on $\widetilde{\MM}^{\infty,p}_{u_{p,s}/M}$, $s\geq 0$, $p\in[1,\infty]$. We claim that $\mathfrak{M}_{p_2}\subseteq \mathfrak{M}_{p_1}$ continuously when $1\leq p_1\leq p_2\leq \infty$. To see this, assume that $1\leq p_1< p_2<\infty$; the proof of the case $p_2=\infty$ is analogous. Let $f\in \mathfrak{M}_{p_2}$. Employing H\"older's inequality with $\widetilde{p}=p_2/p_1$ and $\widetilde{q}=p_2/(p_2-p_1)$ together with the fact $|g_X||g^{\sigma}_X|=1$, we infer (for any $s\geq 0$)
\begin{align*}
\frac{|g_X|^{\frac{1}{2}(1-\frac{1}{p_1})}}{M(X)}&\left(\int_W |\VV_{\boldsymbol{\varphi}}f(X,\Xi)|^{p_1}(1+g^{\sigma}_X(\Xi))^{sp_1}d\Xi\right)^{1/p_1}\\
&\leq\frac{|g_X|^{\frac{1}{2}(1-\frac{1}{p_2})}}{M(X)}\left(\int_W |\VV_{\boldsymbol{\varphi}}f(X,\Xi)|^{p_2} (1+g^{\sigma}_X(\Xi))^{(s+n+1)p_2}d\Xi\right)^{1/p_2}\\
&\quad\cdot\left(\int_W (1+g^{\sigma}_X(\Xi))^{-(n+1)p_1p_2/(p_2-p_1)} |g^{\sigma}_X|^{1/2}d\Xi\right)^{(p_2-p_1)/(p_1p_2)}.
\end{align*}
Since the very last integral is uniformly bounded by a single constant for all $X\in W$, the claim follows. Consequently, to prove the claim in the proposition, it suffices to show that
\begin{align}\label{inc-to-be-proved-for-the-por-thssm}
S(M,g)\subseteq \mathfrak{M}_{\infty}\quad \mbox{and}\quad \mathfrak{M}_1\subseteq S(M,g),\quad \mbox{continuously}.
\end{align}
\indent To prove the first inclusion in \eqref{inc-to-be-proved-for-the-por-thssm}, let $f\in S(M,g)$ and $k\in \ZZ_+$ be arbitrary. Write
$$
\VV_{\boldsymbol{\varphi}}f(X,\Xi)=\int_W e^{-2\pi i[\Xi,Y]}f(Y)\varphi_X(Y)dY.
$$
For every $X,\Xi\in W$ fixed, there exists $\theta=\theta(X,\Xi)\in W$ such that $g_X(\theta)=1$ and $[\Xi,\theta]=g^{\sigma}_X(\Xi)^{1/2}$. Notice that
$$
(1-(2\pi i)^{-1}\partial_{\theta;Y})^{2k}e^{-2\pi i[\Xi,Y]}=(1+g^{\sigma}_X(\Xi)^{1/2})^{2k}e^{-2\pi i[\Xi,Y]}.
$$
We estimate as follows:
\begin{align*}
|\VV_{\boldsymbol{\varphi}}&f(X,\Xi)|u_{\infty,k}(X,\Xi)/M(X)\\
&\leq C'_1|g_X|^{1/2}\sum_{k'+k''\leq 2k}\int_W M(Y)^{-1}|\partial_{\theta}^{k'}f(Y)||\partial_{\theta;Y}^{k''}\varphi_X(Y)|dY\\
&\leq C'_1\|f\|^{(2k)}_{S(M,g)}|g_X|^{1/2}\sum_{k'+k''\leq 2k}\int_W g_Y(\theta)^{k'/2} |\partial_{\theta;Y}^{k''}\varphi_X(Y)| dY\\
&\leq C'_2\|f\|^{(2k)}_{S(M,g)}\|\varphi_X\|^{(N)}_{g_X,U_{X,r'_0}}\int_W (1+g^{\sigma}_X(Y-U_{X,r'_0}))^{-n-1}|g_X|^{1/2}dY,
\end{align*}
for sufficiently large $N\in\ZZ_+$. Since $1+g_X(X-Y)\leq C'_3(1+g^{\sigma}_X(Y-U_{X,r'_0}))$, the very last integral is uniformly bounded by a single constant for all $X\in W$ and the validity of the first part of \eqref{inc-to-be-proved-for-the-por-thssm} follows.\\
\indent We turn our attention to the second inclusion in \eqref{inc-to-be-proved-for-the-por-thssm}. Let $f\in \mathfrak{M}_1$. For $T_1,\ldots, T_k\in W\backslash\{0\}$, $k\in\NN$, we infer
$$
\mathcal{F}_{\sigma}(\partial_{T_1}\ldots\partial_{T_k}(f\varphi_X))=(2\pi i)^k[\cdot,T_1]\cdot\ldots \cdot[\cdot,T_k]\mathcal{F}_{\sigma}(f\varphi_X)\in L^1(W),\;\; \mbox{a.a.}\; X,
$$
and consequently $\partial_{T_1}\ldots\partial_{T_k}(f\varphi_X)\in\mathcal{F}_{\sigma}L^1(W)\subseteq \mathcal{C}_0(W)$, a.a. $X$. Thus, in view of \eqref{est-from-belo-and-above-part-unit}, $\partial_{T_1}\ldots \partial_{T_k}(f\psi)\in\mathcal{C}(W)$, for all $\psi\in\DD(W)$ (since the interiors of the sets in \eqref{est-from-belo-and-above-part-unit} cover $W$ when $X$ varies in $W\backslash(\mbox{a nullset})$). We deduce $f\in\mathcal{C}^{\infty}(W)$. Let $k\in \NN$ and $l\leq k$ be arbitrary but fixed and set $N:=(N_0+1)(n+1)+N_0n+N_0k/2$. Since $\supp\varphi_X\subseteq U_{X,r'_0}$, for $T_1,\ldots,T_l\in W\backslash\{0\}$, we infer
\begin{align}
|\partial_{T_1}\ldots\partial_{T_l}f(Y)|/M(Y)&\leq C'_1\int_W |\partial_{T_1;Y}\ldots\partial_{T_l;Y}(f(Y)\varphi_X(Y))|/M(X)dv_g(X)\nonumber\\
&= C'_1\int_W \frac{|\partial_{T_1;Y}\ldots\partial_{T_l;Y}(f(Y)\varphi_X(Y))|} {(1+g^{\sigma}_X(Y-U_{X,r'_0}))^NM(X)} dv_g(X).\label{est-for-the-opp-incls}
\end{align}
Employing $\mathcal{F}_{\sigma}\mathcal{F}_{\sigma}=\operatorname{Id}$ and \eqref{ineq-for-metric-p-1}, we estimate as follows:
\begin{align*}
|\partial_{T_1;Y}&\ldots\partial_{T_l;Y}(f(Y)\varphi_X(Y))|/M(X)\\
&= \frac{(2\pi)^l}{M(X)}\left|\int_W e^{-2\pi i[Y,\Xi]}\mathcal{F}_{\sigma}(f\varphi_X)(\Xi)\prod_{j=1}^l[\Xi,T_j]d\Xi\right|\\
&\leq \frac{(2\pi)^l\prod_{j=1}^lg_Y(T_j)^{1/2}}{M(X)}\int_W |\VV_{\boldsymbol{\varphi}}f(X,\Xi)|g^{\sigma}_Y(\Xi)^{l/2}d\Xi\\
&\leq \frac{C'_2(1+g^{\sigma}_X(Y-U_{X,r'_0}))^{kN_0/2} \prod_{j=1}^lg_Y(T_j)^{1/2}}{M(X)}\int_W |\VV_{\boldsymbol{\varphi}}f(X,\Xi)|g^{\sigma}_X(\Xi)^{l/2}d\Xi.
\end{align*}
Plugging this estimate into \eqref{est-for-the-opp-incls}, we deduce
$$
\frac{|\partial_{T_1}\ldots \partial_{T_l}f(Y)|}{M(Y)\prod_{j=1}^lg_Y(T_j)^{1/2}}\leq C'_3\|f\|_{\widetilde{\MM}^{\infty,1}_{u_{1,k/2}/M}} \int_W \frac{dv_g(X)}{(1+g^{\sigma}_X(Y-U_{X,r'_0}))^{(N_0+1)(n+1)+N_0n}}.
$$
In view of \eqref{ineq-for-metric-p-3-1}, the integral is uniformly bounded for all $Y\in W$ by a constant and the proof of the second part of \eqref{inc-to-be-proved-for-the-por-thssm} is complete.
\end{proof}

\subsection{Characterisation of $S(M,g)$}\label{char-M,g}

Before we state the main result of the article, we introduce the following notation. As standard, we denote by $\Sp(W)$ the Lie group of symplectic maps on $W$ and we denote by $\Mp(W)$ the metaplectic group as defined by Weil \cite{weil}; i.e. it is a two-fold cover of $\Sp(W)$ (the Weil representation).\footnote{Some authors construct $\Mp(W)$ as a circle cover of $\Sp(W)$.} Let $(\cdot,\cdot)_W$ be an inner product on $W$ and let $L:W\rightarrow W'$ be the induced isomorphism. Notice that $\langle LX,Y\rangle=(X,Y)_W$, for all $X,Y\in W$. Denote by $\Sym(W,(\cdot,\cdot)_W)$ the vector space of all symmetric operators on the Hilbert space $(W,(\cdot,\cdot)_W)$ and let $\Sym_+(W,(\cdot,\cdot)_W)$ be the subset of $\Sym(W,(\cdot,\cdot)_W)$ consisting of all symmetric positive-definite operators; $\Sym_+(W,(\cdot,\cdot)_W)$ is an open convex cone in $\Sym(W,(\cdot,\cdot)_W)$. For any $A\in\Sym_+(W,(\cdot,\cdot)_W)$ we denote by $A^{1/2}$ the unique symmetric positive-definite square root of $A$. Recall that the mapping $\Sym_+(W,(\cdot,\cdot)_W)\rightarrow \Sym_+(W,(\cdot,\cdot)_W)$, $A\mapsto A^{1/2}$, is smooth. If $\gamma$ is any positive definite quadratic form on $W$ with $Q:W\rightarrow W'$ being the induced linear map, then $L^{-1}Q\in\Sym_+(W,(\cdot,\cdot)_W)$ and hence its square root $(L^{-1}Q)^{1/2}$ is a well-defined element of $\Sym_+(W,(\cdot,\cdot)_W)$.

\begin{lemma}\label{lema-for-sym-root-quadr}${}$
\begin{itemize}
\item[$(i)$] Let $A:W\rightarrow W'$ be an isomorphism which satisfies ${}^tA=A$. Then $\sigma^{-1}A$ is symplectic if and only if $\sigma^{-1}A\sigma^{-1}A=-I$.
\item[$(ii)$] Let $(\cdot,\cdot)_W$ be an inner product on $W$ with $L:W\rightarrow W'$ the induced isomorphism. Assume that $\sigma^{-1}L$ is symplectic (or, equivalently, $\sigma^{-1}L\sigma^{-1}L=-I$ in view of $(i)$). Let $\gamma$ be a positive-definite quadratic form on $W$ with $Q:W\rightarrow W'$ the induced isomorphism. Then
    \begin{equation}\label{equ-for-square-root-of-sympl-dual}
    (L^{-1}Q^{\sigma})^{1/2}=-L^{-1}\sigma(L^{-1}Q)^{-1/2}L^{-1}\sigma.
    \end{equation}
    Furthermore, for all $X,Y\in W$, it holds that
    \begin{equation}\label{equ-for-movin-accrr-sym-form-quad-root}
    [(L^{-1}Q)^{1/2}X,Y]=[X,(L^{-1}Q^{\sigma})^{-1/2}Y]\quad \mbox{and}\quad [(L^{-1}Q^{\sigma})^{1/2}X,Y]=[X,(L^{-1}Q)^{-1/2}Y].
    \end{equation}
\end{itemize}
\end{lemma}

\begin{proof} To verify $(i)$, notice that
$$
[\sigma^{-1}AX,\sigma^{-1}AY]=-\langle A\sigma^{-1}AX,Y\rangle=-[\sigma^{-1}A\sigma^{-1}AX,Y],\quad X,Y\in W.
$$
The last term is equal to $[X,Y]$, for all $X,Y\in W$, if and only if $\sigma^{-1}A\sigma^{-1}A=-I$.\\
\indent We turn our attention to $(ii)$. It is straightforward to verify that the right-hand side of \eqref{equ-for-square-root-of-sympl-dual} is symmetric positive-definite operator on $W$ and that the composition of this operator with itself is $L^{-1}Q^{\sigma}$ (recall $Q^{\sigma}=-\sigma Q^{-1}\sigma$). Consequently, \eqref{equ-for-square-root-of-sympl-dual} holds true. To prove \eqref{equ-for-movin-accrr-sym-form-quad-root} it suffices to show the second equality as then the first will follow from this and the fact $(Q^{\sigma})^{\sigma}=Q$. For $X,Y\in W$, in view of $(i)$, we infer
\begin{align*}
[(L^{-1}Q^{\sigma})^{1/2}X,Y]&=-\langle (L^{-1}Q)^{-1/2}L^{-1}\sigma X, \sigma L^{-1} \sigma Y\rangle\\
&=((L^{-1}Q)^{-1/2}L^{-1}\sigma X,Y)_W= (L^{-1}\sigma X,(L^{-1}Q)^{-1/2}Y)_W\\
&=[X,(L^{-1}Q)^{-1/2}Y].
\end{align*}
\end{proof}

\begin{remark}
If $(\cdot,\cdot)_V$ is any inner product on $V$ with $(\cdot,\cdot)_{V'}$ being the dual inner product on $V'$ then
$$
((x,\xi),(y,\eta))_W:=(x,y)_V+(\xi,\eta)_{V'},\quad (x,\xi),(y,\eta)\in W,
$$
satisfies the assumption in Lemma \ref{lema-for-sym-root-quadr} $(ii)$.
\end{remark}

\begin{remark}\label{rem-for-det-of-root-of-qua}
If $(\cdot,\cdot)_W$ is such that $\sigma^{-1}L:W\rightarrow W$ has determinant $1$, then $\det (L^{-1}Q_X)=|g_X|$. To see this, let $E_j$, $j=1,\ldots,2n$, be a symplectic basis on $W$ with $E'_j$, $j=1,\ldots,2n$, being the dual basis on $W'$ and denote by $P:W\rightarrow W'$ the isomorphism that sends $E_j$ to $E'_j$, $j=1,\ldots,2n$; clearly ${}^tP=P$. Then
\begin{align*}
|g_X|&=\det((\langle E_j,Q_XE_k\rangle)_{j,k})= \det((\langle E'_j,P^{-1}Q_XE_k\rangle)_{j,k})=\det(P^{-1}Q_X)\\
&= \det(P^{-1}\sigma)\det(\sigma^{-1}L)\det(L^{-1}Q_X)=\det(L^{-1}Q_X);
\end{align*}
the very last equality follows from the fact $\det(P^{-1}\sigma)=1$ since $P^{-1}\sigma\in\Sp(W)$. Similarly, in this case we also have $|g^{\sigma}_X|=\det(L^{-1}Q^{\sigma}_X)$. In particular, $\det (L^{-1}Q_X)=|g_X|$ and $\det(L^{-1}Q^{\sigma}_X)=|g^{\sigma}_X|$ when $\sigma^{-1}L$ is symplectic, i.e. when $\sigma^{-1}L\sigma^{-1}L=-I$ (cf. Lemma \ref{lema-for-sym-root-quadr} $(i)$).
\end{remark}

Given an inner product $(\cdot,\cdot)_W$ on $W$ with $L:W\rightarrow W'$ the induced isomorphism, for each $X\in W$, we denote by $\Psi^{g,L}_X$ the topological isomorphism
\begin{align*}
\Psi^{g,L}_X:\SSS(W)\rightarrow \SSS(W),\quad (\Psi^{g,L}_X\varphi)(Y)=\varphi((L^{-1}Q_X)^{-1/2}Y),\; Y\in W,
\end{align*}
and we extend it by duality to the topological isomorphism
\begin{align}\label{change-of-var-for-dualit}
\Psi^{g,L}_X:\SSS'(W)\rightarrow \SSS'(W),\quad \langle\Psi^{g,L}_Xf,\varphi\rangle=|\det (L^{-1}Q_X)|^{1/2} \langle f,\varphi\circ(L^{-1}Q_X)^{1/2}\rangle.
\end{align}
As standard, we denote by $\pi(x,\xi)$, $(x,\xi)\in W$, the time-frequency shift $\pi(x,\xi)f=e^{2\pi i \langle \xi,\cdot\rangle} f(\cdot-x)$, $f\in\SSS'(V)$. For $\chi_1,\chi_2\in\SSS(V)$, $W(\chi_1,\chi_2)$ stands for the Wigner function
$$
W(\chi_1,\chi_2)(x,\xi)=\int_V e^{-2\pi i \langle \xi,y\rangle}\chi_1(x+y/2)\overline{\chi_2(x-y/2)}dy,\quad (x,\xi)\in W;
$$
recall that $W(\chi_1,\chi_2)\in \SSS(W)$.\\
\indent We are ready to state and prove the main result of the article.

\begin{theorem}\label{main-theorem-dia}
Let $a\in \SSS'(W)$ and $M$ be a $g$-admissible weight with slow variation constant $r$. Let $(\cdot,\cdot)_W$ be an inner product on $W$ whose induced isomorphism $L:W\rightarrow W'$ is such that $\sigma^{-1}L$ is a symplectic map and let $\Psi^{g,L}_X$, $X\in W$, be the isomorphisms defined in \eqref{change-of-var-for-dualit}. Let $\{\theta_X\,|\, X\in W\}\subseteq \mathcal{O}_{\mathcal{M}}(W)$ and $\chi\in\SSS(V)$. Consider the following conditions:
\begin{itemize}
\item[$(i)$] $a\in S(M,g)$;
\item[$(ii)$] for each $N\in\NN$, the function
    \begin{multline*}
    (X,\Xi)\mapsto M((X+\Xi)/2)^{-1}(1+g_{\frac{X+\Xi}{2}}(X-\Xi))^N\\
    \cdot\left\langle \left(\Psi^{g,L}_{\frac{X+\Xi}{2}}(a\theta_{\frac{X+\Xi}{2}})\right)^w \pi\left((L^{-1}Q_{\frac{X+\Xi}{2}})^{1/2}X\right)\chi, \overline{\pi\left((L^{-1}Q_{\frac{X+\Xi}{2}})^{1/2}\Xi\right)\chi}\right\rangle
    \end{multline*}
    belongs to $L^{\infty}(W\times W)$.
\end{itemize}
Set
\begin{equation}\label{def-for-var-for-non-deg-ele-confmainthe}
\varphi_X(Y):=\overline{\theta_X(Y)} W(\chi,\chi)((L^{-1}Q_X)^{1/2}(Y-X)),\quad X,Y\in W.
\end{equation}
If $\boldsymbol{\varphi}:W\mapsto \SSS(W)$, $\boldsymbol{\varphi}(X)=\varphi_X$, belongs to $\Conf_g(W;\min\{r_0,r\})$, then $(i)$ implies $(ii)$. If, furthermore, $\boldsymbol{\varphi}$ is non-degenerate, then $(ii)$ implies $(i)$.
\end{theorem}

\begin{remark}
Since $\theta_X\in \mathcal{O}_{\mathcal{M}}(W)$, $X\in W$, we have $a\theta_X\in \SSS'(W)$, for all $a\in \SSS'(W)$ and $X\in W$. Thus, the right hand side of $(ii)$ is well-defined for each $X,\Xi\in W$.
\end{remark}

\begin{remark}
In Lemma \ref{fam-sat-con-maithforvar} below we give a class of families $\{\theta_X\,|\, X\in W\}$ which satisfy all of the assumptions in Theorem \ref{main-theorem-dia} for all $\chi\in\SSS(V)$ such that $W(\chi,\chi)(0)\neq 0$ (see also Remark \ref{rem-for-exa-famisatisallofassummainthe}).
\end{remark}

\begin{proof} Set $\tilde{Q}_X:=L^{-1}Q_X$ and $\tilde{Q}^{\sigma}_X=L^{-1}Q^{\sigma}_X$, $X\in W$. Assume that $\boldsymbol{\varphi}\in\Conf_g(W;\min\{r_0,r\})$. First we show that for each $a\in \SSS'(W)$, the function
\begin{equation}\label{fun-for-measur}
W\times W\rightarrow \CC,\quad (X,\Xi)\mapsto \left\langle\left(\Psi^{g,L}_{\frac{X+\Xi}{2}}(a\theta_{\frac{X+\Xi}{2}})\right)^w \pi\left(\tilde{Q}^{1/2}_{\frac{X+\Xi}{2}}X\right)\chi, \overline{\pi\left(\tilde{Q}^{1/2}_{\frac{X+\Xi}{2}}\Xi\right)\chi}\right\rangle,
\end{equation}
is measurable. Assume first that $a\in\SSS(W)$. Then \cite[Lemma 3.1]{Gr2} (especially its proof; cf. \cite[Equation (4.4)]{gro-rz}) gives
\begin{equation}\label{eq2-ss11}
\langle a^w\pi(X)\chi,\overline{\pi(\Xi)\chi}\rangle=e^{if(X,\Xi)}\int_W e^{-2\pi i[Y,X-\Xi]}a(Y) W(\chi,\chi)\left(Y-\frac{X+\Xi}{2}\right) dY,
\end{equation}
where $f$ is a real-valued smooth function on $W\times W$ which does not depend on $a$ and $\chi$. In view of Remark \ref{rem-for-det-of-root-of-qua} and Lemma \ref{lema-for-sym-root-quadr} $(ii)$, we infer ($W(\chi,\chi)$ is real-valued)
\begin{align}
&\left\langle\left(\Psi^{g,L}_{\frac{X+\Xi}{2}}(a\theta_{\frac{X+\Xi}{2}})\right)^w \pi\left(\tilde{Q}^{1/2}_{\frac{X+\Xi}{2}}X\right)\chi, \overline{\pi\left(\tilde{Q}^{1/2}_{\frac{X+\Xi}{2}}\Xi\right)\chi}\right\rangle\nonumber\\
&=e^{if\left(\tilde{Q}^{1/2}_{\frac{X+\Xi}{2}}X,\tilde{Q}^{1/2}_{\frac{X+\Xi}{2}}\Xi\right)} |g_{\frac{X+\Xi}{2}}|^{1/2}\int_W e^{-2\pi i\left[Y\, , \,\tilde{Q}^{\sigma\, -1/2}_{\frac{X+\Xi}{2}}\tilde{Q}^{1/2}_{\frac{X+\Xi}{2}}(X-\Xi)\right]}a(Y) \overline{\varphi_{\frac{X+\Xi}{2}}(Y)}dY\nonumber\\
&=e^{if\left(\tilde{Q}^{1/2}_{\frac{X+\Xi}{2}}X,\tilde{Q}^{1/2}_{\frac{X+\Xi}{2}}\Xi\right)} |g_{\frac{X+\Xi}{2}}|^{1/2} \left\langle a,e^{-2\pi i\left[\,\cdot\, ,\,\tilde{Q}^{\sigma\, -1/2}_{\frac{X+\Xi}{2}}\tilde{Q}^{1/2}_{\frac{X+\Xi}{2}}(X-\Xi)\right]} \overline{\varphi_{\frac{X+\Xi}{2}}}\right\rangle.\label{equ-for-the-mapping-for-main-res}
\end{align}
Since $X\mapsto \tilde{Q}_X^{1/2}$ and $X\mapsto \tilde{Q}^{\sigma\, -1/2}_X$ are measurable maps from $W$ into $\Sym_+(W,(\cdot,\cdot)_W)$ (recall, the square root is smooth on $\Sym_+(W,(\cdot,\cdot)_W)$), this integral representation together with the fact that $W\rightarrow \SSS(W)$, $X\mapsto \varphi_X$, is strongly measurable implies that the function \eqref{fun-for-measur} is measurable when $a\in \SSS(W)$. Now, the case when $a\in \SSS'(W)$ follows from the fact that $\SSS(W)$ is sequentially dense in $\SSS'(W)$. In view of \eqref{equ-for-the-mapping-for-main-res}, this also verifies the following key identity:
\begin{multline}\label{ident-sympstft-actpsid}
\left|\left\langle\left(\Psi^{g,L}_{\frac{X+\Xi}{2}}(a\theta_{\frac{X+\Xi}{2}})\right)^w \pi\left(\tilde{Q}^{1/2}_{\frac{X+\Xi}{2}}X\right)\chi, \overline{\pi\left(\tilde{Q}^{1/2}_{\frac{X+\Xi}{2}}\Xi\right)\chi}\right\rangle\right|\\
=|g_{\frac{X+\Xi}{2}}|^{1/2}\left|\VV_{\boldsymbol{\varphi}}a\left(\frac{X+\Xi}{2}, \tilde{Q}^{\sigma\, -1/2}_{\frac{X+\Xi}{2}}\tilde{Q}^{1/2}_{\frac{X+\Xi}{2}}(\Xi-X)\right)\right|,\;\;  X,\Xi\in W,
\end{multline}
for all $a\in \SSS'(W)$ (\eqref{ident-sympstft-actpsid} transfers the problem from $\Psi$DOs to the theory of the generalised modulation spaces which we developed in the previous sections). Notice that $g^{\sigma}_X(\tilde{Q}^{\sigma\, -1/2}_X\tilde{Q}^{1/2}_X\Xi)=g_X(\Xi)$, $X,\Xi\in W$. Consequently, the condition $(ii)$ is equivalent to the following: for every $N\in\NN$ there exist $C_N>0$ and a nullset $K_N\subseteq W\times W$ such that
\begin{multline}\label{condi-equiv-toc-in-theore}
\left|\VV_{\boldsymbol{\varphi}}a\left(\frac{X+\Xi}{2}, \tilde{Q}^{\sigma\, -1/2}_{\frac{X+\Xi}{2}}\tilde{Q}^{1/2}_{\frac{X+\Xi}{2}}(\Xi-X)\right)\right|\\
\leq \frac{C_NM((X+\Xi)/2)}{|g_{\frac{X+\Xi}{2}}|^{1/2} \left(1+g^{\sigma}_{\frac{X+\Xi}{2}}\left(\tilde{Q}^{\sigma\, -1/2}_{\frac{X+\Xi}{2}} \tilde{Q}^{1/2}_{\frac{X+\Xi}{2}}(\Xi-X)\right)\right)^N},\;\; (X,\Xi)\in (W\times W)\backslash K_N.
\end{multline}
Consider the bijection $A:W\times W\rightarrow W\times W$, $A(Z_1,Z_2)=(Z_1,\tilde{Q}^{\sigma\, -1/2}_{Z_1}\tilde{Q}^{1/2}_{Z_1}Z_2)$ with inverse $A^{-1}(Z_1,Z_2)=(Z_1,\tilde{Q}^{-1/2}_{Z_1}\tilde{Q}^{\sigma\, 1/2}_{Z_1}Z_2)$. Both $A$ and $A^{-1}$ are measurable maps. We claim that the images of every nullset $\tilde{K}\subseteq W\times W$ under $A$ and $A^{-1}$ are again nullsets. To see this, pick a Borel nullset $K\supseteq \tilde{K}$ and denote by $\kappa_K$ the characteristic function of $K$. Since $\kappa_K\circ A^{-1}$ is the characteristic function of $A(K)$, it follows that $\kappa_K\circ A^{-1}$ is measurable. Remark \ref{rem-for-det-of-root-of-qua}, Fubini's theorem and the fact $|g_X||g^{\sigma}_X|=1$ give
\begin{align*}
(\lambda\times\lambda)(A(K))&=\int_W\left(\int_W \kappa_K(A^{-1}(Z_1,Z_2)) dZ_2\right)dZ_1\\
&= \int_W|g_{Z_1}|\left(\int_W \kappa_K(Z_1,Z_2) dZ_2\right)dZ_1=\int_K |g_{Z_1}|dZ_1dZ_2=0.
\end{align*}
Consequently, $A(\tilde{K})$ is a nullset; the proof that $A^{-1}(\tilde{K})$ is a nullset is analogous. Thus, in view of \eqref{condi-equiv-toc-in-theore}, the condition $(ii)$ is equivalent to
$$
(Z_1,Z_2)\mapsto M(Z_1)^{-1}|g_{Z_1}|^{1/2}(1+g^{\sigma}_{Z_1}(Z_2))^N |\VV_{\boldsymbol{\varphi}}a(Z_1,Z_2)|\quad \mbox{belongs to}\,\, L^{\infty}(W\times W),
$$
for all $N\in\NN$. Now, both claims in the theorem follow from Theorem \ref{mod-spa-donot-depend-on-partuni} and Proposition \ref{rep-res-hor}.
\end{proof}

\begin{remark}\label{rem-for-con-smooth-metr-fun}
If both $g$ and $\boldsymbol{\varphi}$ are of class $\mathcal{C}^k$, $0\leq k\leq\infty$, then \eqref{equ-for-the-mapping-for-main-res} implies that the function \eqref{fun-for-measur} is also of class $\mathcal{C}^k$, $0\leq k\leq\infty$, for each $a\in\SSS'(W)$ (see the comments before Remark \ref{rem-for-euc-metric-stand}).
\end{remark}

\begin{remark}\label{sim-con-iim-for-evewo}
Introducing the change of variables $(X,\Xi)\mapsto (X+\Xi,X-\Xi)$, one verifies that the condition $(ii)$ is equivalent to the following:
\begin{itemize}
\item[$(ii)'$] for each $N\in\NN$, the function
\begin{multline*}
(X,\Xi)\mapsto M(X)^{-1}(1+g_X(\Xi))^N\\
    \cdot\left\langle \left(\Psi^{g,L}_X(a\theta_X)\right)^w \pi\left((L^{-1}Q_X)^{1/2}(X+\Xi)\right)\chi, \overline{\pi\left((L^{-1}Q_X)^{1/2}(X-\Xi)\right)\chi}\right\rangle
    \end{multline*}
    belongs to $L^{\infty}(W\times W)$.
\end{itemize}
\end{remark}

Employing analogous technique as in the proof of Theorem \ref{main-theorem-dia}, we can characterise the symbols in $\widetilde{\MM}^{p,p}_{\eta}$ in similar fashion.

\begin{theorem}\label{charc-of-symplec-mod-space-with-growth}
Let $L$ and $\Psi^{g,L}_X$, $X\in W$, be as in Theorem \ref{main-theorem-dia}. Let $\eta$ be a uniformly admissible weight with respect to $g$ with slow variation constant $r$ and let $1\leq p\leq \infty$. Let $\{\theta_X\,|\, X\in W\}\subseteq \mathcal{O}_{\mathcal{M}}(W)$ and $\chi\in\SSS(V)$. Consider the following conditions for $a\in\SSS'(W)$:
\begin{itemize}
\item[$(i)$] $a\in \widetilde{\MM}^{p,p}_{\eta}$;
\item[$(ii)$] the function
    \begin{multline*}
    (X,\Xi)\mapsto |g_{\frac{X+\Xi}{2}}|^{-\frac{1}{2}+\frac{3}{2p}}\, \eta\left(\frac{X+\Xi}{2}, (L^{-1}Q^{\sigma}_{\frac{X+\Xi}{2}})^{-1/2}(L^{-1}Q_{\frac{X+\Xi}{2}})^{1/2} (\Xi-X)\right)\\
    \cdot\left\langle \left(\Psi^{g,L}_{\frac{X+\Xi}{2}}(a\theta_{\frac{X+\Xi}{2}})\right)^w \pi\left((L^{-1}Q_{\frac{X+\Xi}{2}})^{1/2}X\right)\chi, \overline{\pi\left((L^{-1}Q_{\frac{X+\Xi}{2}})^{1/2}\Xi\right)\chi}\right\rangle
    \end{multline*}
    belongs to $L^p(W\times W)$.
\end{itemize}
If $\boldsymbol{\varphi}:W\mapsto \SSS(W)$, $\boldsymbol{\varphi}(X)=\varphi_X$, with $\varphi_X$ given by \eqref{def-for-var-for-non-deg-ele-confmainthe}, belongs to $\Conf_g(W;\min\{r_0,r\}),$ then $(i)$ implies $(ii)$. If, furthermore, $\boldsymbol{\varphi}$ is non-degenerate, then $(ii)$ implies $(i)$.
\end{theorem}

\begin{proof} We claim that for each $a\in\SSS'(W)$, the function in $(ii)$ is measurable. Since \eqref{fun-for-measur} is measurable, we only need to prove that the part with the $\eta$ function is measurable; notice that it equals $\eta\circ A((X+\Xi)/2,\Xi-X)$ with $A$ being the measurable bijection we defined in the proof of Theorem \ref{main-theorem-dia}. To verify the measurability of the latter, it suffices to show that $(X,\Xi)\mapsto \eta\circ A(X,\Xi)$ is measurable. This follows from the fact that the image under $A^{-1}$ of every nullset is again a nullset (see the proof of Theorem \ref{main-theorem-dia}).\\
\indent When $p=\infty$, the proof of the theorem is the same as the proof of Theorem \ref{main-theorem-dia}. Assume that $1\leq p<\infty$. The identity \eqref{ident-sympstft-actpsid} together with a change of variables (cf. Remark \ref{rem-for-det-of-root-of-qua}) verifies that $(ii)$ is equivalent to $\VV_{\boldsymbol{\varphi}}a\in L^p_{\eta}(W\times W,dv_gd\lambda)$. Now, both assertions in the theorem follow from Theorem \ref{mod-spa-donot-depend-on-partuni}.
\end{proof}

One can simplify the condition $(ii)$ by introducing the change of variables from Remark \ref{sim-con-iim-for-evewo}.

\subsection{Examples and consequences of Theorem \ref{main-theorem-dia}}\label{ex111}

Our first goal is to give examples of $\chi\in\SSS(V)$ and $\{\theta_X\,|\, X\in W\}\subseteq \mathcal{O}_{\mathcal{M}}(W)$ which satisfy the assumptions of Theorem \ref{main-theorem-dia}. We start with the following technical result.

\begin{lemma}\label{lemma-for-wig-gen-dec}
Let $(\cdot,\cdot)_W$ be an inner product on $W$ with $L:W\rightarrow W'$ the induced isomorphism. Let $\psi\in\SSS(W)$ and, for each $X\in W$, define $\tilde{\psi}_X:=\psi((L^{-1}Q_X)^{1/2}(\cdot -X))$. Then $\tilde{\psi}_X\in\SSS(W)$, $X\in W$, and the map $W\rightarrow \SSS(W)$, $X\mapsto \tilde{\psi}_X$, is strongly measurable. Furthermore,
\begin{equation}\label{est-for-con-par-als}
\sup_{l\leq k}\sup_{\substack{X,Y\in W\\ T_1,\ldots, T_l\in W\backslash \{0\}}}\frac{|\partial_{T_1;Y}\ldots \partial_{T_l;Y}\tilde{\psi}_X(Y)|(1+g_X(X-Y))^N}{\prod_{j=1}^l g_X(T_j)^{1/2}}<\infty, \;\; k,N\in\NN.
\end{equation}
In particular, if $g$ is symplectic, then $W\rightarrow \SSS(W)$, $X\mapsto \tilde{\psi}_X$, belongs to $\Conf_g(W;r)$, for all $r\in(0,r_0]$. If, in addition, $\psi(0)\neq 0,$ then $X\mapsto \tilde{\psi}_X$ is non-degenerate.
\end{lemma}

\begin{proof} Denote by $|\cdot|$ the norm on $W$ induced by $(\cdot,\cdot)_W$ and set $\tilde{Q}_X:=L^{-1}Q_X$, $X\in W$. Clearly $\tilde{\psi}_X\in\SSS(W)$, $X\in W$, and the strong measurability of $W\rightarrow \SSS(W)$, $X\mapsto \tilde{\psi}_X$, follows from Lemma \ref{lemma-regularity-sta}. To prove the bounds \eqref{est-for-con-par-als}, let $T_1,\ldots,T_l\in W\backslash\{0\}$ and $N\in\NN$. Then
$$
|\partial_{T_1;Y}\ldots \partial_{T_l;Y}\tilde{\psi}_X(Y)|= |\psi^{(l)}(\tilde{Q}_X^{1/2}(Y-X);\tilde{Q}_X^{1/2}T_1,\ldots,\tilde{Q}^{1/2}_XT_l)|\leq \frac{C\prod_{j=1}^l|\tilde{Q}^{1/2}_XT_j|}{(1+|\tilde{Q}^{1/2}_X(Y-X)|)^N}.
$$
Since $|\tilde{Q}^{1/2}_XT|=g_X(T)^{1/2}$, $ T\in W$, the bounds \eqref{est-for-con-par-als} immediately follow.\\
\indent Assume now that $g$ is symplectic. The fact that $W\rightarrow \SSS(W)$, $X\mapsto \tilde{\psi}_X$, belongs to $\Conf_g(W;r)$, for all $r\in(0,r_0]$, is an immediate consequence of \eqref{est-for-con-par-als}. To prove that it is non-degenerate when $\psi(0)\neq 0$, pick $r_1,c_1>0$ such that $|\psi(X)|\geq c_1$, for all $|X|\leq r_1$. Set $\varepsilon:=\min\{r_0,r_1/\sqrt{C_0}\}$. Let $Y\in W$ be arbitrary but fixed. Then $|\psi(\tilde{Q}^{1/2}_X(Y-X))|\geq c_1$, $ X\in U_{Y,\varepsilon}$. To verify this, notice that when $X\in U_{Y,\varepsilon}$ we have
$$
|\tilde{Q}^{1/2}_X(Y-X)|^2=g_X(Y-X)\leq C_0g_Y(Y-X)\leq r_1^2,
$$
which implies the claim. In view of Remark \ref{equ-for-sym-metr-meas-lebesmes}, we infer
\begin{align*}
\int_W|\tilde{\psi}_X(Y)|^2dv_g(X)&\geq \int_{U_{Y,\varepsilon}} |\psi(\tilde{Q}^{1/2}_X(Y-X))|^2 |g_Y|^{1/2}dX\\
&\geq c_1^2\int_{U_{Y,\varepsilon}} |g_Y|^{1/2}dX= c'.
\end{align*}
Consequently, $X\mapsto \tilde{\psi}_X$ is non-degenerate.
\end{proof}

\begin{remark}\label{cont-metr-par-co}
If $g$ is continuous (not necessarily symplectic), it is straightforward to show that $W\rightarrow \SSS(W)$, $X\mapsto \tilde{\psi}_X$, is also continuous (by employing Taylor expansion).
\end{remark}

The following lemma gives an example of a family $\{\theta_X\,|\, X\in W\}$ which satisfies all of the assumptions in Theorem \ref{main-theorem-dia}.

\begin{lemma}\label{fam-sat-con-maithforvar}
Let $(\cdot,\cdot)_W$ and $L$ be as in Theorem \ref{main-theorem-dia}. Let $r>0$ and let $\theta_0\in \mathcal{C}^{\infty}([0,\infty))\backslash\{0\}$ be any non-negative and non-increasing function which satisfies $\supp \theta_0\subseteq [0,1]$. Set $\tilde{r}:=\min\{r_0,r\}$ and define $\theta_X(Y):=\theta_0(\tilde{r}^{-2}g_X(X-Y))$, $X,Y\in W$. For any $\chi\in\SSS(V)$ satisfying
\begin{equation}\label{equ-for-con-on-thefun-tobe-fam}
W(\chi,\chi)(0)=2^n\int_V \chi(x)\overline{\chi(-x)}dx\neq 0,
\end{equation}
the map $\boldsymbol{\varphi}:W\rightarrow \SSS(W)$, $\boldsymbol{\varphi}(X)=\varphi_X$, with $\varphi_X$ given by \eqref{def-for-var-for-non-deg-ele-confmainthe}, is a non-degenerate element of $\Conf_g(W;\min\{r_0,r\})$.
\end{lemma}

\begin{remark}\label{rem-for-exa-famisatisallofassummainthe}
Given a $g$-admissible weight $M$ (resp., a uniformly admissible weight $\eta$ with respect to $g$) with a slow variation constant $r$, the lemma shows that $\{\theta_X\,|\, X\in W\}$ satisfies all of the assumptions in Theorem \ref{main-theorem-dia} (resp., in Theorem \ref{charc-of-symplec-mod-space-with-growth}) for any $\chi\in\SSS(V)$ that satisfies \eqref{equ-for-con-on-thefun-tobe-fam}.
\end{remark}

\begin{proof} Denote by $|\cdot|$ the norm on $W$ induced by $(\cdot,\cdot)_W$. Set $\tilde{Q}_X:=L^{-1}Q_X$, $X\in W$. The map $W\rightarrow \SSS(W)$, $X\mapsto \theta_X$, is strongly measurable in view of Lemma \ref{lemma-regularity-sta1} and, similarly as in the proof of \cite[Theorem 2.2.7, p. 70]{lernerB}, one can show that it belongs to $\Conf_g(W;\tilde{r})$. Consequently, Lemma \ref{lemma-for-wig-gen-dec} implies that $\boldsymbol{\varphi}\in\Conf_g(W;\tilde{r})$. It remains to show that $\boldsymbol{\varphi}$ is non-degenerate. There is $0<\varepsilon'<1$ such that $\theta_0\geq \varepsilon'$ on $[0,\varepsilon']$. Pick $r_1,c_1>0$ such that $|W(\chi,\chi)(X)|\geq c_1$, for all $|X|\leq r_1$, and set $\varepsilon:=\sqrt{\varepsilon'/C_0}\min\{\tilde{r},r_1\}$. For $Y\in W$ fixed, we employ \eqref{ineq-for-metric-p-5} and estimate as follows:
\begin{align*}
\int_W|\varphi_X(Y)|^2dv_g(X)&\geq \int_{U_{Y,\varepsilon}}\theta_0(\tilde{r}^{-2}g_X(X-Y))^2 |W(\chi,\chi)(\tilde{Q}^{1/2}_X(Y-X))|^2 |g_X|^{1/2}dX\\
&\geq C^{-n}_0\int_{U_{Y,\varepsilon}} \theta_0(\tilde{r}^{-2}C_0g_Y(X-Y))^2|W(\chi,\chi)(\tilde{Q}^{1/2}_X(Y-X))|^2 |g_Y|^{1/2}dX.
\end{align*}
When $X\in U_{Y,\varepsilon}$, we have $\tilde{r}^{-2}C_0g_Y(X-Y)\leq \varepsilon'$ and
\begin{equation*}
|\tilde{Q}^{1/2}_X(Y-X)|=g_X(X-Y)^{1/2}\leq \sqrt{C_0}g_Y(X-Y)^{1/2}\leq r_1.
\end{equation*}
Consequently,
$$
\int_W|\varphi_X(Y)|^2dv_g(X)\geq C^{-n}_0\varepsilon'^2c_1^2\int_{U_{Y,\varepsilon}}|g_Y|^{1/2}dX= c'',\quad Y\in W,
$$
which verifies that $\boldsymbol{\varphi}$ is non-degenerate.
\end{proof}

A typical example for $\chi$ is to take a Gaussian. To be precise, let $e_1,\ldots,e_n$ be a basis for $V$ such that the volume with respect to the measure $dx$ of the parallelepiped formed with these vectors is $1$, and let $e'_1,\ldots,e'_n$ be the dual basis of $V'$. If we set
$$
\chi(x):=e^{-\pi\sum_{j=1}^n|\langle e'_j,x\rangle|^2},\quad x\in V,
$$
then $W(\chi,\chi)(x,\xi)=2^{n/2}e^{-2\pi\sum_{j=1}^n(|\langle e'_j,x\rangle|^2+|\langle \xi,e_j\rangle|^2)}$ (see \cite[Equation (4.20), p. 72]{Gr1}).\\
\indent Assume now that $g$ is a symplectic H\"ormander metric on $W$ and let $(\cdot,\cdot)_W$ and $L$ be as in Theorem \ref{main-theorem-dia}. As an immediate consequence of Lemma \ref{lemma-for-wig-gen-dec}, we can even take $\theta_X(Y)=1$ and $X\mapsto \varphi_X:=W(\chi,\chi)((L^{-1}Q_X)^{1/2}(\cdot-X))$ will still belong to $\Conf_g(W;r)$ for any $r\in(0,r_0]$ and any $\chi\in \SSS(V)$. Moreover, if $\chi$ satisfies \eqref{equ-for-con-on-thefun-tobe-fam}, then Lemma \ref{lemma-for-wig-gen-dec} also yields that $X\mapsto \varphi_X$ is non-degenerate. Furthermore, Lemma \ref{lema-for-sym-root-quadr} $(ii)$ implies that $\tilde{Q}_X^{1/2}:=(L^{-1}Q_X)^{1/2}\in \Sp(W)$, $X\in W$. Denote by $\Pi$ the surjective homomorphism $\Pi:\Mp(W)\rightarrow \Sp(W)$. The metaplectic covariance of the Weyl quantisation \cite[Theorem 7.13, p. 205]{deGosson} (cf. \cite[Theorem 4.3]{hormander}) yields that there is $\Phi^{g,L}_X\in\Mp(W)$ such that
$$
\Pi(\Phi^{g,L}_X)=\tilde{Q}^{-1/2}_X\quad \mbox{and}\quad (\Psi^{g,L}_Xa)^w=(\Phi^{g,L}_X)^*a^w\Phi^{g,L}_X,\quad  X\in W,\, a\in\SSS'(W)
$$
(there are two such operators and they only differ by a sign). Let $\tau_X$, $X\in W$, be the unitary operator $\tau_{(x,\xi)}\chi(y)=e^{2\pi i\langle \xi,y-x/2\rangle}\chi(y-x)$, $\chi\in \SSS(V)$, $(x,\xi)\in W$. Clearly, $\tau_{(x,\xi)}=e^{-\pi i\langle \xi,x\rangle}\pi(x,\xi)$. Since $\Omega\tau_X\Omega^*=\tau_{\Pi(\Omega)X}$, $\Omega\in\Mp(W)$, $X\in W$
(see \cite[Theorem 7.13, p. 205]{deGosson}), we infer that for each $X,Y\in W$, $\pi(\tilde{Q}^{1/2}_Y X)$ is equal to $(\Phi^{g,L}_Y)^*\pi(X)\Phi^{g,L}_Y$ up to a constant of modulus $1$. Consequently, the above discussion gives
\begin{multline}\label{equ-sym-metaplect-operators-for-weylq}
\left|\left\langle \left(\Psi^{g,L}_{\frac{X+\Xi}{2}}a\right)^w \pi\left(\tilde{Q}_{\frac{X+\Xi}{2}}^{1/2}X\right)\chi, \overline{\pi\left(\tilde{Q}_{\frac{X+\Xi}{2}}^{1/2}\Xi\right)\chi}\right\rangle\right|\\
=\left|\left\langle a^w\pi(X) \Phi^{g,L}_{\frac{X+\Xi}{2}}\chi, \overline{\pi(\Xi)\Phi^{g,L}_{\frac{X+\Xi}{2}}\chi}\right\rangle\right|, \; X,\Xi\in W,
\end{multline}
for all $a\in\SSS'(W)$, $\chi\in\SSS(V)$. We proved the following consequence of Theorem \ref{main-theorem-dia}.

\begin{corollary}\label{cor-for-sym-metric-metplectope}
Assume that $g$ is a symplectic H\"ormander metric and let $a\in \SSS'(W)$ and $\chi\in\SSS(V)$. Let $M$, $r$, $L$ and $\Psi^{g,L}_X$, $X\in W$, be as in Theorem \ref{main-theorem-dia}. Then $(L^{-1}Q_X)^{1/2}\in\Sp(W)$, for all $X\in W$. For each $X\in W$, choose $\Phi^{g,L}_X\in\Mp(W)$ such that $\Pi(\Phi^{g,L}_X)=(L^{-1}Q_X)^{-1/2}$ and consider the following conditions:
\begin{itemize}
\item[$(i)$] $a\in S(M,g)$;
\item[$(ii)$] for each $N\in\NN$, the function
    \begin{multline*}
    (X,\Xi)\mapsto M((X+\Xi)/2)^{-1}(1+g_{\frac{X+\Xi}{2}}(X-\Xi))^N\\
    \cdot\left\langle \left(\Psi^{g,L}_{\frac{X+\Xi}{2}}a\right)^w \pi\left((L^{-1}Q_{\frac{X+\Xi}{2}})^{1/2}X\right)\chi, \overline{\pi\left((L^{-1}Q_{\frac{X+\Xi}{2}})^{1/2}\Xi\right)\chi}\right\rangle
    \end{multline*}
    belongs to $L^{\infty}(W\times W)$;
\item[$(iii)$] for each $N\in\NN$, the function
    \begin{equation*}
    (X,\Xi)\mapsto M((X+\Xi)/2)^{-1}(1+g_{\frac{X+\Xi}{2}}(X-\Xi))^N \left|\left\langle a^w\pi(X) \Phi^{g,L}_{\frac{X+\Xi}{2}}\chi, \overline{\pi(\Xi)\Phi^{g,L}_{\frac{X+\Xi}{2}}\chi}\right\rangle\right|
    \end{equation*}
    belongs to $L^{\infty}(W\times W)$.
\end{itemize}
Then $(i)\Rightarrow (ii)\Leftrightarrow(iii)$. If, furthermore, $\chi$ satisfies \eqref{equ-for-con-on-thefun-tobe-fam}, then $(i)\Leftrightarrow (ii)\Leftrightarrow(iii)$.
\end{corollary}

\begin{remark}
In view of \eqref{fun-for-measur} and \eqref{equ-sym-metaplect-operators-for-weylq}, the functions in $(ii)$ and $(iii)$ are measurable for all $a\in\SSS'(W)$ and $\chi\in\SSS(V)$. If both $g$ and $M$ are continuous, then these functions are also continuous for all $a\in\SSS'(W)$ and $\chi\in\SSS(V)$ (cf. Remark \ref{rem-for-con-smooth-metr-fun} and Remark \ref{cont-metr-par-co}).
\end{remark}

The above discussion also proves the following consequence of Theorem \ref{charc-of-symplec-mod-space-with-growth} (cf. Remark \ref{equ-for-sym-metr-meas-lebesmes}).

\begin{corollary}
Assume that $g$ is a symplectic H\"ormander metric and let $a\in \SSS'(W)$, $\chi\in\SSS(V)$ and $1\leq p\leq \infty$. Let $\eta$, $r$, $L$ and $\Psi^{g,L}_X$, $X\in W$, be as in Theorem \ref{charc-of-symplec-mod-space-with-growth}. Then $(L^{-1}Q_X)^{1/2}\in\Sp(W)$, for all $X\in W$. For each $X\in W$, choose $\Phi^{g,L}_X\in\Mp(W)$ such that $\Pi(\Phi^{g,L}_X)=(L^{-1}Q_X)^{-1/2}$ and consider the following conditions:
\begin{itemize}
\item[$(i)$] $a\in \widetilde{\MM}^{p,p}_{\eta}$;
\item[$(ii)$] the function
    \begin{multline*}
    (X,\Xi)\mapsto \eta((X+\Xi)/2, \Xi-X)\\
    \cdot\left\langle \left(\Psi^{g,L}_{\frac{X+\Xi}{2}}a\right)^w \pi\left((L^{-1}Q_{\frac{X+\Xi}{2}})^{1/2}X\right)\chi, \overline{\pi\left((L^{-1}Q_{\frac{X+\Xi}{2}})^{1/2}\Xi\right)\chi}\right\rangle
    \end{multline*}
    belongs to $L^p(W\times W)$;
\item[$(iii)$] the function
    \begin{equation*}
    (X,\Xi)\mapsto \eta((X+\Xi)/2, \Xi-X)\left|\left\langle a^w\pi(X) \Phi^{g,L}_{\frac{X+\Xi}{2}}\chi, \overline{\pi(\Xi)\Phi^{g,L}_{\frac{X+\Xi}{2}}\chi}\right\rangle\right|
    \end{equation*}
    belongs to $L^p(W\times W)$.
\end{itemize}
Then $(i)\Rightarrow (ii)\Leftrightarrow(iii)$. If, furthermore, $\chi$ satisfies \eqref{equ-for-con-on-thefun-tobe-fam}, then $(i)\Leftrightarrow (ii)\Leftrightarrow(iii)$.
\end{corollary}

We now specialise the main results to the commonly used calculi. These are all generated by H\"ormander metrics on $\RR^{2n}$ of the form
\begin{equation}\label{beals-fferf-metric-gen}
g_{(x,\xi)}=f(x,\xi)^{-2}|dx|^2+F(x,\xi)^{-2}|d\xi|^2,
\end{equation}
with $f$ and $F$ positive measurable functions on $\RR^{2n}$ (notice that \eqref{beals-fferf-metric-gen} is symplectic if and only if $F=1/f$). In this case, employing the standard inner product on $\RR^{2n}$, $\tilde{Q}^{1/2}_{(x,\xi)}$ is just a diagonal matrix whose first $n$ entries along the diagonal are $1/f(x,\xi)$ and the second $n$ entries are $1/F(x,\xi)$. The construction in Lemma \ref{fam-sat-con-maithforvar} amounts to
\begin{equation}\label{def-of-tht-for-ord-cal}
\theta_{(x,\xi)}(y,\eta)=\theta_0\left(\frac{|x-y|^2}{\tilde{r}^2f(x,\xi)^2}+ \frac{|\xi-\eta|^2}{\tilde{r}^2F(x,\xi)^2}\right),\quad (x,\xi),(y,\eta)\in\RR^{2n},
\end{equation}
and the isomorphism \eqref{change-of-var-for-dualit} is given by
\begin{equation}\label{iso-ind-by-the-met-in-cas-ordmecal}
\Psi_{(x,\xi)}\psi(y,\eta)=\psi(f(x,\xi)y,F(x,\xi)\eta),\quad \psi\in\SSS(\RR^{2n}).
\end{equation}
If $g$ is a symplectic H\"ormander metric on $\RR^{2n}$ of the form \eqref{beals-fferf-metric-gen} (hence, $F=1/f$), then the metaplectic operators $\Phi_{(x,\xi)}$, $(x,\xi)\in\RR^{2n}$, which satisfy $\Pi(\Phi_{(x,\xi)})=\tilde{Q}^{-1/2}_{(x,\xi)}$ are given by \cite[Proposition 7.8, p. 202]{deGosson}
\begin{equation}\label{metap-op-for-sym-hor-met-onr}
\Phi_{(x,\xi)}\chi(y)=\pm f(x,\xi)^{-n/2}\chi(y/f(x,\xi)).
\end{equation}
Combining Theorem \ref{main-theorem-dia}, Corollary \ref{cor-for-sym-metric-metplectope}, Remark \ref{sim-con-iim-for-evewo}, Lemma \ref{lemma-for-wig-gen-dec} and Lemma \ref{fam-sat-con-maithforvar}, we deduce the following result.

\begin{corollary}
Let $g$ be a H\"ormander metric on $\RR^{2n}$ given by \eqref{beals-fferf-metric-gen} with slow variation constant $r_0$ and let $M$ be a $g$-admissible weight with slow variation constant $r$. For each $(x,\xi)\in\RR^{2n}$, let $\Psi_{(x,\xi)}:\SSS'(\RR^{2n})\rightarrow\SSS'(\RR^{2n})$ be the isomorphism given by \eqref{iso-ind-by-the-met-in-cas-ordmecal} and define $\theta_{(x,\xi)}\in\DD(\RR^{2n})$ by \eqref{def-of-tht-for-ord-cal} with $\tilde{r}:=\min\{r_0,r\}$ and $\theta_0$ as in Lemma \ref{fam-sat-con-maithforvar}. Let $\chi\in\SSS(\RR^n)$ be such that $W(\chi,\chi)(0)\neq 0$. For any $a\in\SSS'(\RR^{2n})$, the following statements are equivalent:
\begin{itemize}
\item[$(i)$] $a\in S(M,g)$, i.e. $a\in\mathcal{C}^{\infty}(\RR^{2n})$ and $\|M^{-1} f^{|\beta|}F^{|\alpha|}\partial^{\alpha}_{\xi} \partial^{\beta}_xa\|_{L^{\infty}(\RR^{2n})}<\infty$, $\alpha,\beta\in\NN$;
\item[$(ii)$] for each $N\in\NN$, the function
\begin{multline*}
(x,\xi,y,\eta)\mapsto M(x,\xi)^{-1} \left(1+\frac{|y|}{f(x,\xi)}+ \frac{|\eta|}{F(x,\xi)}\right)^N\\
\left\langle \left(\Psi_{(x,\xi)}(a \theta_{(x,\xi)})\right)^w \pi\left(\frac{x+y}{f(x,\xi)}, \frac{\xi+\eta}{F(x,\xi)}\right) \chi, \overline{\pi\left(\frac{x-y}{f(x,\xi)}, \frac{\xi-\eta}{F(x,\xi)}\right) \chi}\right\rangle
\end{multline*}
belongs to $L^{\infty}(\RR^{4n})$.
\end{itemize}
If, in addition, $g$ is symplectic (i.e., $F=1/f$) then $(i)$ (and consequently, also $(ii)$) is equivalent to each of the following statements:
\begin{itemize}
\item[$(iii)$] for each $N\in\NN$, the function
\begin{multline*}
(x,\xi,y,\eta)\mapsto M(x,\xi)^{-1} (1+|y|f(x,\xi)^{-1}+ |\eta|f(x,\xi))^N\\
\left\langle (\Psi_{(x,\xi)}a)^w \pi\left(\frac{x+y}{f(x,\xi)}, (\xi+\eta)f(x,\xi)\right) \chi, \overline{\pi\left(\frac{x-y}{f(x,\xi)}, (\xi-\eta)f(x,\xi)\right) \chi}\right\rangle
\end{multline*}
belongs to $L^{\infty}(\RR^{4n})$;
\item[$(iv)$] for each $N\in\NN$, the function
 \begin{multline*}
    (x,\xi,y,\eta)\mapsto M(x,\xi)^{-1} (1+|y|f(x,\xi)^{-1}+ |\eta|f(x,\xi))^N\\
    |\langle a^w\pi(x+y,\xi+\eta) \Phi_{(x,\xi)}\chi, \overline{\pi(x-y,\xi-\eta)\Phi_{(x,\xi)}\chi}\rangle|
    \end{multline*}
    belongs to $L^{\infty}(\RR^{4n})$, where $\Phi_{(x,\xi)}$ is the metaplectic operator \eqref{metap-op-for-sym-hor-met-onr}.
\end{itemize}
\end{corollary}

The corollary characterises the symbol classes of all of the commonly used calculi. Taking $g$ to be the Euclidian metric on $\RR^{2n}$ (i.e. $f(x,\xi)=F(x,\xi)=1$) and $M(x,\xi)=1$, both $(iii)$ and $(iv)$ reduce to the characterisation of $S^0_{0,0}(\RR^{2n})$ given in \cite[Theorem 6.2 $(i)$-$(ii)$]{gro-rz} (after changing variables with the transformation inverse to the one employed in Remark \ref{sim-con-iim-for-evewo}).

\appendix
\section{Proofs of the facts stated in Subsection \ref{properties}}\label{appendix1-proof-subwithhor-metfa}

The fact that the functions \eqref{funco-for-met-meas-cont-whe-gisco} are always measurable and continuous when the metric is continuous will follow from the following result (see the remarks following it).

\begin{lemma}\label{lemma-meas-cont-metr-obtf}
Let $\mathcal{B}(W)$ be the real vector space of all bilinear mappings $W\times W\rightarrow \RR$. Let $\tilde{g}$ be a Riemannian metric on $W$ and let $\tilde{\tilde{g}}:W\times W\rightarrow \mathcal{B}(W)$, $(X,Y)\mapsto\tilde{\tilde{g}}_{(X,Y)}$, be a measurable map that is symmetric and positive-definite at every point. Denote the corresponding quadratic forms by the same symbols: $\tilde{g}_X(T):=\tilde{g}_X(T,T)$ and $\tilde{\tilde{g}}_{(X,Y)}(T):=\tilde{\tilde{g}}_{(X,Y)}(T,T)$. Let $r_1,r_2\in[0,\infty)$ and, for each $X\in W$, set $\tilde{U}_{X,r_j}:=\{Y\in W\,|\, \tilde{g}_X(X-Y)\leq r^2_j\}$, $j=1,2$. Then the function
\begin{equation}\label{equ-for-fun-met-measur}
W\times W\rightarrow [0,\infty),\quad (X,Y)\mapsto \tilde{\tilde{g}}_{(X,Y)}(\tilde{U}_{X,r_1}-\tilde{U}_{Y,r_2}),
\end{equation}
is measurable. If both $\tilde{g}$ and $\tilde{\tilde{g}}$ are continuous, then \eqref{equ-for-fun-met-measur} is also continuous.
\end{lemma}

\begin{proof} Given $X\in W$ and $r\geq 0$, denote by $\tilde{B}_X(r)$ the closed $\tilde{g}_X$-ball with centre at the origin and radius $r$, i.e. $\tilde{B}_X(r):=\{Y\in W\,|\, \tilde{g}_X(Y)\leq r^2\}$. Notice that $\tilde{U}_{X,r}=X+\tilde{B}_X(r)$. For each fixed $Z\in W$ and $r\geq 0$, define the function:
$$
\kappa_{r;Z}:W\rightarrow [0,\infty),\quad
\left\{
\begin{array}{l}
\kappa_{r;Z}(X)=1,\,\, \mbox{if}\,\, Z\in \tilde{B}_X(r),\\
\kappa_{r;Z}(X)=0,\,\, \mbox{if}\,\, Z\not\in \tilde{B}_X(r).
\end{array}
\right.
$$
Notice that $\{X\in W\,|\, \kappa_{r;Z}(X)>c\}$ is the empty set if $c\geq 1$, is the whole $W$ if $c<0$ and is equal to $\{X\in W\,|\, \tilde{g}_X(Z)\leq r^2\}$ when $c\in[0,1)$. Consequently, $\kappa_{r;Z}$ is a measurable function on $W$ for each fixed $Z\in W$ and $r\geq 0$. Let $\{Z_j\in W\,|\, j\in\NN\}$ be a countable dense subset of $W$ which contains the origin and, for each $X\in W$, set $\tilde{J}_{l;X}:=\{j\in\NN\,|\, Z_j\in\tilde{B}_X(r_l)\}$, $l=1,2$. Clearly, $\tilde{J}_{l;X}\neq \emptyset$ and $\NN\backslash \tilde{J}_{l;X}\neq \emptyset$, for all $X\in W$, $l=1,2$. For $j,k\in\NN$, define
$$
f_{j,k}:W\times W\rightarrow [0,\infty),\quad f_{j,k}(X,Y):=\tilde{\tilde{g}}_{(X,Y)}(X-Y+\kappa_{r_1;Z_j}(X)Z_j- \kappa_{r_2;Z_k}(Y)Z_k).
$$
In view of the above, $f_{j,k}$ is measurable on $W\times W$ for all $j,k\in\NN$. Notice that
\begin{align}
\inf_{j\in\tilde{J}_{1;X},\, k\in \tilde{J}_{2;Y}} f_{j,k}(X,Y)&=\inf_{j\in \tilde{J}_{1;X},\, k\in\tilde{J}_{2;Y}}\tilde{\tilde{g}}_{(X,Y)}(X-Y+Z_j-Z_k)\nonumber\\
&= \inf_{Z'\in\tilde{B}_X(r_1),\, Z''\in\tilde{B}_Y(r_2)}\tilde{\tilde{g}}_{(X,Y)}(X-Y+Z'-Z'')\nonumber\\
&=\tilde{\tilde{g}}_{(X,Y)}(\tilde{U}_{X,r_1}-\tilde{U}_{Y,r_2}).\label{equ-for-meas-lesee}
\end{align}
When $j\not\in \tilde{J}_{1;X}$ and $k\in \tilde{J}_{2;Y}$, \eqref{equ-for-meas-lesee} implies
$$
f_{j,k}(X,Y)=\tilde{\tilde{g}}_{(X,Y)}(X-Y-Z_k)\geq \tilde{\tilde{g}}_{(X,Y)}(\tilde{U}_{X,r_1}-\tilde{U}_{Y,r_2})= \inf_{j\in\tilde{J}_{1;X},\, k\in\tilde{J}_{2;Y}}f_{j,k}(X,Y).
$$
Similarly, $f_{j,k}(X,Y)\geq \inf_{j\in\tilde{J}_{1;X},\, k\in\tilde{J}_{2;Y}}f_{j,k}(X,Y)$, when $j\in \tilde{J}_{1;X}$ and $k\not\in \tilde{J}_{2;Y}$, as well as in the case when $j\not\in \tilde{J}_{1;X}$ and $k\not\in \tilde{J}_{2;Y}$. Consequently, $\inf_{j,k\in\NN}f_{j,k}(X,Y)=\tilde{\tilde{g}}_{(X,Y)} (\tilde{U}_{X,r_1}-\tilde{U}_{Y,r_2})$, for all $X,Y\in W$, which proves the measurability of \eqref{equ-for-fun-met-measur}.\\
\indent Assume now that both $\tilde{g}$ and $\tilde{\tilde{g}}$ are continuous. Keeping the notations for $\tilde{B}_X(r)$, $X\in W$, $r\geq 0$, as above, it suffices to show that $W\times W\rightarrow [0,\infty)$, $(X,Y)\mapsto\tilde{\tilde{g}}_{(X,Y)}(X-Y+\tilde{B}_X(r_1)-\tilde{B}_Y(r_2))$, is continuous. Let $E_j$, $j=1,\ldots,2n$, be a basis of $W$ with $E'_j$, $j=1,\ldots, 2n$, being the dual basis of $W'$ and denote by $|\cdot|$ the following norm on $W$: $|X|:=(\sum_{j=1}^{2n} |\langle E'_j,X\rangle|^2)^{1/2}$, $X\in W$. For $X,Y\in W$, we define
$$
\tilde{g}_{j,k}(X):=\tilde{g}_X(E_j,E_k),\quad \tilde{\tilde{g}}_{j,k}(X,Y):=\tilde{\tilde{g}}_{(X,Y)}(E_j,E_k).
$$
We make the following\\
\\
\noindent \textbf{Claim.} Let $A$ be a compact subset of $W$ and $X',X'',Y',Y'',Z\in W$. Then
\begin{multline}\label{claim-ine-for-cont-met-rrs}
|\tilde{\tilde{g}}_{(X',Y')}(Z+A)-\tilde{\tilde{g}}_{(X'',Y'')}(Z+A)|\\
\leq (|Z|+\sup_{Z'\in A}|Z'|)^2 \sum_{j,k=1}^{2n}|\tilde{\tilde{g}}_{j,k}(X',Y')-\tilde{\tilde{g}}_{j,k}(X'',Y'')|.
\end{multline}

\noindent We postpone its proof for later and continue with the proof of the lemma. Set $r:=\max\{r_1,r_2\}\geq 0$. Let $X_0,Y_0\in W$ and $\varepsilon\in (0,1)$ be arbitrary but fixed. Set $Z_0:=X_0-Y_0$. There exists $C'\geq 1$ such that
\begin{gather}
\tilde{\tilde{g}}_{(X_0,Y_0)}(T)\leq C'^2\tilde{g}_{X_0}(T),\quad T\in W,\quad \mbox{and}\label{ine1-for-cont-met-c}\\
C'^{-2}\tilde{g}_X(T)\leq |T|^2\leq C'^2\tilde{g}_X(T),\quad T\in W,\; X\in \{Z\in W\,|\, |Z|\leq |X_0|+|Y_0|+1\}.\label{ine2-for-cont-met-c}
\end{gather}
Set $\varepsilon_1:=\varepsilon^2 (8C'^6(1+r)(|X_0|+|Y_0|+1))^{-2}$. There exists $\delta_1>0$ such that $\sum_{j,k=1}^{2n}|\tilde{g}_{j,k}(X)-\tilde{g}_{j,k}(X_0)|\leq \varepsilon_1$ when $|X-X_0|\leq \delta_1$ and $\sum_{j,k=1}^{2n}|\tilde{g}_{j,k}(Y)-\tilde{g}_{j,k}(Y_0)|\leq \varepsilon_1$ when $|Y-Y_0|\leq \delta_1$. Furthermore, there exists $\delta_2>0$ such that when $|X-X_0|\leq \delta_2$ and $|Y-Y_0|\leq \delta_2$ it holds that
$$
\sum_{j,k=1}^{2n}|\tilde{\tilde{g}}_{j,k}(X,Y)-\tilde{\tilde{g}}_{j,k}(X_0,Y_0)|\leq \varepsilon/(4(1+2rC'+|X_0|+|Y_0|)^2).
$$
Set $\delta:=\min\{\delta_1,\delta_2,\varepsilon/(16C'^4(|X_0|+|Y_0|+1))\}$. Let $X,Y\in W$ be such that $|X-X_0|\leq \delta$ and $|Y-Y_0|\leq \delta$; notice that both $X$ and $Y$ belong to the set on the right of \eqref{ine2-for-cont-met-c} so \eqref{ine2-for-cont-met-c} holds true. Write $Z=X-Y$ and estimate as follows:
\begin{align*}
|\tilde{\tilde{g}}_{(X,Y)}&(X-Y+\tilde{B}_X(r_1)-\tilde{B}_Y(r_2))- \tilde{\tilde{g}}_{(X_0,Y_0)}(X_0-Y_0+\tilde{B}_{X_0}(r_1)-\tilde{B}_{Y_0}(r_2))|\\
&\leq |\tilde{\tilde{g}}_{(X,Y)}(Z+\tilde{B}_X(r_1)-\tilde{B}_Y(r_2)) - \tilde{\tilde{g}}_{(X_0,Y_0)}(Z+\tilde{B}_X(r_1)-\tilde{B}_Y(r_2))|\\
&{}\quad+ |\tilde{\tilde{g}}_{(X_0,Y_0)}(Z+\tilde{B}_X(r_1)-\tilde{B}_Y(r_2)) -\tilde{\tilde{g}}_{(X_0,Y_0)}(Z_0+\tilde{B}_X(r_1)-\tilde{B}_Y(r_2))|\\
&{}\quad+|\tilde{\tilde{g}}_{(X_0,Y_0)}(Z_0+\tilde{B}_X(r_1)-\tilde{B}_Y(r_2))- \tilde{\tilde{g}}_{(X_0,Y_0)}(Z_0+\tilde{B}_{X_0}(r_1)-\tilde{B}_{Y_0}(r_2))|\\
&=:S_1+S_2+S_3.
\end{align*}
To estimate $S_1$, we apply the Claim. For $X'\in \tilde{B}_X(r_1)$ and $Y'\in \tilde{B}_Y(r_2)$, in view of \eqref{ine2-for-cont-met-c} we infer
$$
|X'|+|Y'|\leq C'\tilde{g}_X(X')^{1/2}+C'\tilde{g}_Y(Y')^{1/2}\leq 2rC'.
$$
Hence, \eqref{claim-ine-for-cont-met-rrs} implies $S_1\leq \varepsilon/4$. We estimate $S_2$ as follows (cf. \eqref{ine1-for-cont-met-c} and \eqref{ine2-for-cont-met-c})
\begin{align}
S_2&= |\tilde{\tilde{g}}_{(X_0,Y_0)}(Z+\tilde{B}_X(r_1)-\tilde{B}_Y(r_2))^{1/2} -\tilde{\tilde{g}}_{(X_0,Y_0)}(Z_0+\tilde{B}_X(r_1)-\tilde{B}_Y(r_2))^{1/2}|\nonumber\\
&{}\quad\cdot (\tilde{\tilde{g}}_{(X_0,Y_0)}(Z+\tilde{B}_X(r_1)-\tilde{B}_Y(r_2))^{1/2}+ \tilde{\tilde{g}}_{(X_0,Y_0)} (Z_0+\tilde{B}_X(r_1)-\tilde{B}_Y(r_2))^{1/2})\label{equ-for-s2con-mms}\\
&\leq \tilde{\tilde{g}}_{(X_0,Y_0)}(Z-Z_0)^{1/2}(\tilde{\tilde{g}}_{(X_0,Y_0)}(Z)^{1/2}+ \tilde{\tilde{g}}_{(X_0,Y_0)}(Z_0)^{1/2})\nonumber\\
&\leq C'^4|Z-Z_0|(|Z|+|Z_0|)\leq \varepsilon/4.\nonumber
\end{align}
To estimate $S_3$, we proceed as in \eqref{equ-for-s2con-mms} to infer
\begin{multline*}
S_3\leq 2C'^2(|X_0|+|Y_0|)\\
\cdot |\tilde{\tilde{g}}_{(X_0,Y_0)}(Z_0+\tilde{B}_X(r_1)-\tilde{B}_Y(r_2))^{1/2}- \tilde{\tilde{g}}_{(X_0,Y_0)}(Z_0+\tilde{B}_{X_0}(r_1)-\tilde{B}_{Y_0}(r_2))^{1/2}|.
\end{multline*}
Assume first that
\begin{equation}\label{equ-for-thir-part-est-spp}
\tilde{\tilde{g}}_{(X_0,Y_0)}(Z_0+\tilde{B}_{X_0}(r_1)-\tilde{B}_{Y_0}(r_2))^{1/2}\geq \tilde{\tilde{g}}_{(X_0,Y_0)}(Z_0+\tilde{B}_X(r_1)-\tilde{B}_Y(r_2))^{1/2}
\end{equation}
and pick $X'\in \tilde{B}_X(r_1)$ and $Y'\in\tilde{B}_Y(r_2)$ such that
$$
\tilde{\tilde{g}}_{(X_0,Y_0)}(Z_0+\tilde{B}_X(r_1)-\tilde{B}_Y(r_2))^{1/2}= \tilde{\tilde{g}}_{(X_0,Y_0)}(Z_0+X'-Y')^{1/2}.
$$
We claim that $(1+C'\sqrt{\varepsilon_1})^{-1}X'\in \tilde{B}_{X_0}(r_1)$ and $(1+C'\sqrt{\varepsilon_1})^{-1}Y'\in\tilde{B}_{Y_0}(r_2)$. To see this, notice that (cf. \eqref{ine2-for-cont-met-c})
\begin{equation*}
|\tilde{g}_{X_0}(X')-\tilde{g}_X(X')|\leq |X'|^2\sum_{j,k=1}^{2n}|\tilde{g}_{j,k}(X_0)-\tilde{g}_{j,k}(X)|\leq \varepsilon_1C'^2\tilde{g}_X(X')\leq \varepsilon_1C'^2r^2_1.
\end{equation*}
Hence, $\tilde{g}_{X_0}(X')\leq r^2_1(1+\varepsilon_1C'^2)\leq r^2_1(1+C'\sqrt{\varepsilon_1})^2$ which shows that $(1+C'\sqrt{\varepsilon_1})^{-1}X'\in\tilde{B}_{X_0}(r_1)$. The proof of $(1+C'\sqrt{\varepsilon_1})^{-1}Y'\in \tilde{B}_{Y_0}(r_2)$ is analogous. Thus, in view of \eqref{ine1-for-cont-met-c} and \eqref{ine2-for-cont-met-c}, we infer
\begin{align*}
S_3&\leq 2C'^2(|X_0|+|Y_0|)\\
&{}\quad \cdot (\tilde{\tilde{g}}_{(X_0,Y_0)}(Z_0+(1+C'\sqrt{\varepsilon_1})^{-1}(X'-Y'))^{1/2}- \tilde{\tilde{g}}_{(X_0,Y_0)}(Z_0+X'-Y')^{1/2})\\
&\leq 2C'^3\sqrt{\varepsilon_1}(|X_0|+|Y_0|)\tilde{\tilde{g}}_{(X_0,Y_0)}(X'-Y')^{1/2} \leq 2C'^4\sqrt{\varepsilon_1}(|X_0|+|Y_0|)\tilde{g}_{X_0}(X'-Y')^{1/2}\\
&\leq 2C'^4\sqrt{\varepsilon_1}(|X_0|+|Y_0|)(\tilde{g}_{X_0}(X')^{1/2}+ \tilde{g}_{X_0}(Y')^{1/2})\leq 4C'^6r\sqrt{\varepsilon_1}(|X_0|+|Y_0|)\leq \varepsilon/2.
\end{align*}
The proof of $S_3\leq \varepsilon/2$ when the opposite inequality in \eqref{equ-for-thir-part-est-spp} holds is analogous. The proof of the lemma is complete.\\
\indent It remains to show the Claim. Pick $T''\in A$ such that $\tilde{\tilde{g}}_{(X'',Y'')}(Z+A)=\tilde{\tilde{g}}_{(X'',Y'')}(Z+T'')$ and estimate as follows
\begin{align*}
\tilde{\tilde{g}}_{(X',Y')}(Z+T'')-\tilde{\tilde{g}}_{(X'',Y'')}(Z+A)&\leq |\tilde{\tilde{g}}_{(X',Y')}(Z+T'')-\tilde{\tilde{g}}_{(X'',Y'')}(Z+T'')|\\
&\leq |Z+T''|^2 \sum_{j,k=1}^{2n}|\tilde{\tilde{g}}_{j,k}(X',Y')-\tilde{\tilde{g}}_{j,k}(X'',Y'')|.
\end{align*}
Consequently,
\begin{align*}
\tilde{\tilde{g}}_{(X',Y')}(Z+A)&\leq \tilde{\tilde{g}}_{(X',Y')}(Z+T'')\\
&\leq (|Z|+\sup_{Z'\in A}|Z'|)^2 \sum_{j,k=1}^{2n}|\tilde{\tilde{g}}_{j,k}(X',Y')-\tilde{\tilde{g}}_{j,k}(X'',Y'')| +\tilde{\tilde{g}}_{(X'',Y'')}(Z+A).
\end{align*}
Interchanging the role of $X'$ and $Y'$ with $X''$ and $Y''$, we deduce \eqref{claim-ine-for-cont-met-rrs}.
\end{proof}

\begin{remark}\label{rem-for-meas-prod-borel}
If we strengthen the assumptions on $\tilde{\tilde{g}}$ by requiring it to be measurable with respect to the product of the Lebesgue $\sigma$-algebras on the two copies of $W$, then the first part of the proof implies that \eqref{equ-for-fun-met-measur} is also measurable with respect to the product of the Lebesgue $\sigma$-algebras. Consequently, for each $X$ (resp., for each $Y$), \eqref{equ-for-fun-met-measur} is a measurable function of $Y$ (resp., $X$) on $W$. Similarly, the first part of the proof also yields that if both $\tilde{g}$ and $\tilde{\tilde{g}}$ are Borel measurable, then the same holds for the function in \eqref{equ-for-fun-met-measur}.
\end{remark}

\begin{remark}\label{rem-for-measura-cont-when-met-ssh}
Taking $\tilde{g}_X=g_X$, $X\in W$, and $\tilde{\tilde{g}}_{(X,Y)}=g^{\sigma}_X$, $X,Y\in W$, in the above lemma, in view of Remark \ref{rem-for-meas-prod-borel}, we can deduce that
\begin{equation}\label{equ-for-fun-met-measur1}
(X,Y)\mapsto g^{\sigma}_X(Y-U_{X,r}),\quad (X,Y)\mapsto g^{\sigma}_X(X-U_{Y,r}),\quad (X,Y)\mapsto g^{\sigma}_X(U_{X,r}-U_{Y,r}),
\end{equation}
are measurable on $W\times W$ with respect to the product of the Lebesgue $\sigma$-algebras for any $r>0$; take $r_1=r$, $r_2=0$ for the first one, $r_1=0$, $r_2=r$ for the second and $r_1=r_2=r$ for the third. Hence, for every $X$ (resp., $Y$), the above are measurable functions of $Y$ (resp., $X$) on $W$. When $g$ is continuous, the lemma implies that all of the functions in \eqref{equ-for-fun-met-measur1} are also continuous.
\end{remark}

Next, we show the inequalities stated in Lemma \ref{tec-res-ine-for-hor-metr-ini}.

\begin{proof}[Proof of Lemma \ref{tec-res-ine-for-hor-metr-ini}] To prove \eqref{ineq-for-metric-p-1}, let $X'\in U_{X,r}$. Then
\begin{align*}
g_X(T)&\leq C_0g_{X'}(T)\leq C_0^2 g_Y(T)(1+g^{\sigma}_{X'}(Y-X'))^{N_0}\leq C_0^{N_0+2} g_Y(T)(1+g^{\sigma}_X(Y-X'))^{N_0},\\
g_Y(T)&\leq C_0g_{X'}(T)(1+g^{\sigma}_{X'}(Y-X'))^{N_0}\leq C_0^{N_0+2}g_X(T)(1+g^{\sigma}_X(Y-X'))^{N_0},
\end{align*}
and \eqref{ineq-for-metric-p-1} immediately follows from this. In order to prove \eqref{ineq-for-metric-p-2}, notice that \eqref{ineq-for-metric-p-1} gives
$$
g_Y(X-Y)\leq C_0^{N_0+2}g_X(X-Y)(1+g^{\sigma}_X(Y-U_{X,r}))^{N_0}.
$$
Since  $g_X(X-Y)\leq 2r^2+2g_X(X'-Y)\leq 2r^2+2g^{\sigma}_X(Y-X')$, for $X'\in U_{X,r}$, the validity of \eqref{ineq-for-metric-p-2} follows. In order to verify \eqref{ineq-for-metric-p-3}, notice that \eqref{ineq-for-metric-p-1} gives
$$
g^{\sigma}_Y(Y-U_{X,r})\leq g^{\sigma}_Y(Y-X')\leq C_0^{N_0+2}g^{\sigma}_X(Y-X')(1+g^{\sigma}_X(Y-U_{X,r}))^{N_0},\; X'\in U_{X,r}.
$$
This implies \eqref{ineq-for-metric-p-3}. We turn our attention to \eqref{ineq-for-metric-p-4}. As $|g_X|$ and $|g^{\sigma}_X|$ are the same in any symplectic basis, \eqref{ineq-for-metric-p-4} follows once we evaluate them in a basis $E_j$, $j=1,\ldots,2n$, which satisfies $0<g_X(E_j,E_j)=1/g^{\sigma}_X(E_j,E_j)\leq 1$, $j=1,\ldots,2n$, and $g_X(E_j,E_k)=g^{\sigma}_X(E_j,E_k)=0$ when $j\neq k$; see \cite[Lemma 4.4.25, p. 339]{lernerB} concerning the existence of such symplectic basis. The validity of \eqref{ineq-for-metric-p-5} and \eqref{ineq-for-metric-p-6} follow from the following classical fact: Let $A$ and $B$ be positive-definite symmetric matrices with real entries. If $A-B$ is positive semi-definite then $\det A\geq \det B$. We only show \eqref{ineq-for-metric-p-6} as the proof of \eqref{ineq-for-metric-p-5} is similar. Let $E_j$, $j=1,\ldots,2n$, be a symplectic basis of $W$ and for each $X\in W$ denote by $A^{\sigma}_X$ the matrix $(g^{\sigma}_X(E_j,E_k))_{j,k}$. Then \eqref{ineq-for-metric-p-1} implies that $C_0^{N_0+2}(1+g^{\sigma}_X(Y-U_{X,r}))^{N_0}A^{\sigma}_Y-A^{\sigma}_X$ is positive semi-definite and hence, $|g^{\sigma}_X|\leq |g^{\sigma}_Y|C_0^{2nN_0+4n}(1+g^{\sigma}_X(Y-U_{X,r}))^{2nN_0}$. The proof that $|g^{\sigma}_Y|/|g^{\sigma}_X|$ and $(|g_X|/|g_Y|)^{\pm 1}$ have the same bound is analogous. The proof of \eqref{ineq-for-metric-p-7} is similar: employing the same technique one deduces (cf. \eqref{ineq-for-metric-p-4})
\begin{equation}\label{inequ-for-met-sim-vol-der}
C_0^{-2n}|g_0|(1+g^{\sigma}_0(X))^{-2nN_0}\leq |g_X|\leq |g^{\sigma}_X|\leq C_0^{2n}|g^{\sigma}_0|(1+g^{\sigma}_0(X))^{2nN_0},\; X\in W,
\end{equation}
which implies the validity of \eqref{ineq-for-metric-p-7}. It remains to prove \eqref{ineq-for-metric-p-3-1}. In view of \eqref{ineq-for-metric-p-2} and \eqref{ineq-for-metric-p-6}, we infer that the integral in \eqref{ineq-for-metric-p-3-1} is uniformly bounded with
$$
C'\int_W (1+g_Y(X-Y))^{-n-1}|g_Y|^{1/2}dX,\quad Y\in W,
$$
and the latter integral is uniformly bounded  (by a single constant) for all $Y\in W$.
\end{proof}

\section{Proof of the property $(d)$ in Example \ref{exi-of-good-par-off} $(ii)$}\label{appendix2-proof-dexm}

First we show the following bounds: for every $k,l,m\in\NN$ there is $C\geq 1$ such that
\begin{multline}\label{ine-for-met-der-on-all-var}
\left|\left(\prod_{j=1}^k\partial_{T_j;X}\right) \left(\prod_{j=1}^l\partial_{S_j;Y}\right) \left(\prod_{j=1}^m\partial_{S'_j;Z}\right) \tilde{g}_X(Y,Z)\right|\leq Cg_X(Y)^{(1-l)/2}g_X(Z)^{(1-m)/2}\\
\cdot\left(\prod_{j=1}^k g_X(T_j)^{1/2}\right)\left(\prod_{j=1}^l g_X(S_j)^{1/2}\right) \left(\prod_{j=1}^m g_X(S'_j)^{1/2}\right),
\end{multline}
for all $X,Y,Z,T_1,\ldots, T_k, S_1,\ldots, S_l, S'_1,\ldots, S'_m\in W$. Notice that the left-hand side of \eqref{ine-for-met-der-on-all-var} is identically equal to $0$ if $\max\{l,m\}\geq 2$. We prove \eqref{ine-for-met-der-on-all-var} when $l=m=1$; the proof when one or both of $l$ and $m$ are zero is analogous and we omit it. The definition of $\tilde{g}$ together with the fact $\supp\psi_X\subseteq U_{X,r_0}$, $X\in W$, gives
\begin{equation*}
\frac{\left|\left(\prod_{j=1}^k\partial_{T_j;X}\right) \partial_{S;Y}\partial_{S';Z} \tilde{g}_X(Y,Z)\right|}{g_X(S)^{1/2} g_X(S')^{1/2}\prod_{j=1}^k g_X(T_j)^{1/2}}\leq C_1\int_W \frac{\left|\left(\prod_{j=1}^k\partial_{T_j;X}\right)\psi_{\widetilde{X}}(X)\right|} {\prod_{j=1}^k g_{\widetilde{X}}(T_j)^{1/2}}dv_g(\widetilde{X}).
\end{equation*}
In view of \eqref{ineq-for-metric-p-3-1}, the integral is uniformly bounded for all $X\in W$ and the validity of \eqref{ine-for-met-der-on-all-var} follows. As a consequence of \eqref{ine-for-met-der-on-all-var}, we will derive the following bounds: for every $k\in\NN$ there is $C\geq 1$ such that
\begin{align}
\left|\left(\prod_{j=1}^k\partial_{T_j;X}\right)\tilde{g}_X(X-Y,Z)\right|&\leq C (1+g_X(X-Y))^{1/2}g_X(Z)^{1/2}\prod_{j=1}^kg_X(T_j)^{1/2}, \label{bound-der-of-the-metr-for-add1}\\
\left|\left(\prod_{j=1}^k\partial_{T_j;X}\right)\tilde{g}_X(X-Y)\right|&\leq C (1+g_X(X-Y))\prod_{j=1}^kg_X(T_j)^{1/2},\label{bound-der-of-the-metr-for-add2}
\end{align}
for all $X,Y,Z,T_1,\ldots,T_k\in W$. To verify them, denote by $\partial_{T;1}$, $\partial_{T;2}$ and $\partial_{T;3}$ the derivations in direction $T\in W$ in the variables $X$, $Y$ and $Z$ respectively of the function $W\times W\times W\rightarrow \RR$, $(X,Y,Z)\mapsto \tilde{g}_X(Y,Z)$. Notice that $\partial_{T;X}(\tilde{g}_X(X-Y,Z))=(\partial_{T;1}+\partial_{T;2})g_X(X-Y,Z)$. Consequently, $\left(\prod_{j=1}^k\partial_{T_j;X}\right)\tilde{g}_X(X-Y,Z)$ is a sum of $2^k$ terms of the form
\begin{equation}\label{est-for-par-of-der-of-the-metr-p11}
\left(\prod_{j\in K_1}\partial_{T_j;1}\right)\left(\prod_{j\in K_2}\partial_{T_j;2}\right) \tilde{g}_X(X-Y,Z),
\end{equation}
where the sets $K_1$ and $K_2$ are disjoint and their union is $\{1,\ldots,k\}$. Now, \eqref{ine-for-met-der-on-all-var} implies that the absolute value of \eqref{est-for-par-of-der-of-the-metr-p11} is bounded by the right-hand side of \eqref{bound-der-of-the-metr-for-add1} (\eqref{est-for-par-of-der-of-the-metr-p11} is $0$ if $|K_2|\geq 2$) and the validity of \eqref{bound-der-of-the-metr-for-add1} follows. To verify \eqref{bound-der-of-the-metr-for-add2}, notice that
\begin{equation*}
\partial_{T;X}(\tilde{g}_X(X-Y))= (\partial_{T;1}+\partial_{T;2}+\partial_{T;3})\tilde{g}_X(X-Y,X-Y).
\end{equation*}
Hence, $\left(\prod_{j=1}^k\partial_{T_j;X}\right)\tilde{g}_X(X-Y)$ is a sum of $3^k$ terms of the form
\begin{equation}\label{est-for-par-of-der-of-the-metr-p12}
\left(\prod_{j\in K_1}\partial_{T_j;1}\right)\left(\prod_{j\in K_2}\partial_{T_j;2}\right) \left(\prod_{j\in K_3}\partial_{T_j;3}\right)\tilde{g}_X(X-Y,X-Y),
\end{equation}
where the sets $K_1$, $K_2$ and $K_3$ are pairwise disjoint and their union is $\{1,\ldots,k\}$. Now, similarly as above, the validity of \eqref{bound-der-of-the-metr-for-add2} follows from \eqref{ine-for-met-der-on-all-var}. We now prove $(d)$. Since $1/I_{\tilde{\boldsymbol{\varphi}}}\in S(1,g)$, in view of \eqref{ineq-for-metric-p-1} it suffices to prove the bound in $(d)$ for $\tilde{\varphi}_X$. Let $k,l\in\NN$ and let $k'\leq k$ and $l'\leq l$. By induction on $l'$, one verifies that $\left(\prod_{j=1}^{l'}\partial_{S_j;Y}\right) \tilde{\varphi}_X(Y)$ is a finite sum of terms of the form
\begin{equation*}
C'^{2p'}r^{-2p'}\left(\prod_{s=1}^{l'-p'}\tilde{g}_X(S_{j'_s},S_{j''_s})\right) \left(\prod_{s=1}^{2p'-l'}\tilde{g}_X(X-Y,S_{j_s})\right) \chi^{(p')}_0(C'^2r^{-2}\tilde{g}_X(X-Y)),
\end{equation*}
where $p'\in[l'/2,l']\cap\NN$ and $(j'_1,\ldots,j'_{l'-p'},j''_1,\ldots,j''_{l'-p'},j_1,\ldots,j_{2p'-l'})$ is a permutation of $\{1,\ldots,l'\}$. Consequently, $\left(\prod_{j=1}^{k'}\partial_{T_j;X}\right) \left(\prod_{j=1}^{l'}\partial_{S_j;Y}\right) \tilde{\varphi}_X(Y)$ is a finite sum of terms of the form
\begin{align}
C'^{2p'}r^{-2p'}&\left(\prod_{s=1}^{l'-p'}\left(\prod_{q''\in K''_s} \partial_{T_{q''};X}\right)\tilde{g}_X(S_{j'_s},S_{j''_s})\right)\nonumber\\
&\cdot\left(\prod_{s=1}^{2p'-l'}\left(\prod_{q'\in K'_s}\partial_{T_{q'};X}\right)\tilde{g}_X(X-Y,S_{j_s})\right)\nonumber\\
&\cdot\left(\prod_{q\in K}\partial_{T_q;X}\right) \chi^{(p')}_0(C'^2r^{-2}\tilde{g}_X(X-Y)),\label{bound-for-the-der-of-the-metr-sum-ter}
\end{align}
where the sets $K$, $K'_s$, $s=1,\ldots,2p'-l'$, and $K''_s$, $s=1,\ldots,l'-p'$, are pairwise disjoint and their union is $\{1,\ldots,k'\}$. Notice that \eqref{ine-for-met-der-on-all-var} immediately gives
\begin{equation}\label{est-first-part-der-of-the-metr}
\prod_{s=1}^{l'-p'}\left|\left(\prod_{q''\in K''_s} \partial_{T_{q''};X}\right)\tilde{g}_X(S_{j'_s},S_{j''_s})\right|\leq C'_1 \prod_{s=1}^{l'-p'}g_X(S_{j'_s})^{1/2}g_X(S_{j''_s})^{1/2}\left(\prod_{q''\in K''_s} g_X(T_{q''})^{1/2}\right),
\end{equation}
while \eqref{bound-der-of-the-metr-for-add1} yields
\begin{multline}\label{est-second-part-der-of-the-metr}
\prod_{s=1}^{2p'-l'}\left|\left(\prod_{q'\in K'_s}\partial_{T_{q'};X}\right)\tilde{g}_X(X-Y,S_{j_s})\right|\\
\leq C'_2 (1+g_X(X-Y))^{p'-l'/2}\prod_{s=1}^{2p'-l'}g_X(S_{j_s})^{1/2}\left(\prod_{q'\in K'_s}g_X(T_{q'})^{1/2}\right).
\end{multline}
To estimate the last term in \eqref{bound-for-the-der-of-the-metr-sum-ter}, first we prove the following bounds: for every $m\in\NN$ there is $C\geq 1$ such that
\begin{multline}\label{est-der-metr-term-with-fun1}
\left|\left(\prod_{j=1}^m\partial_{\tilde{T}_j;X}\right) \chi^{(p')}_0(C'^2r^{-2}\tilde{g}_X(X-Y))\right|\leq C (1+g_X(X-Y))^m\\
\cdot\left(\prod_{j=1}^mg_X(\tilde{T}_j)^{1/2}\right)\sup_{m'\leq m}|\chi^{(p'+m')}_0(C'^2r^{-2}\tilde{g}_X(X-Y))|,
\end{multline}
for all $X,Y,\tilde{T}_1,\ldots,\tilde{T}_m\in W$. When $m=0$, \eqref{est-der-metr-term-with-fun1} trivially holds. To verify it for $m\in\ZZ_+$, observe that, by induction on $m$, one can prove that $\left(\prod_{j=1}^m\partial_{\tilde{T}_j;X}\right) \chi^{(p')}_0(C'^2r^{-2}\tilde{g}_X(X-Y))$ is a finite sum of terms of the form
$$
C'^{2m'}r^{-2m'}\chi^{(p'+m')}_0(C'^2r^{-2}\tilde{g}_X(X-Y)) \prod_{j=1}^{m'}\left(\left(\prod_{s\in J_j}\partial_{\tilde{T}_s;X}\right)\tilde{g}_X(X-Y)\right),
$$
where $m'\in[1,m]\cap\ZZ_+$ and the sets $J_1,\ldots,J_{m'}$ are pairwise disjoint and their union is $\{1,\ldots,m\}$. Now, the validity of \eqref{est-der-metr-term-with-fun1} follows from \eqref{bound-der-of-the-metr-for-add2}. We apply \eqref{est-der-metr-term-with-fun1} to deduce the following bounds for the last term in \eqref{bound-for-the-der-of-the-metr-sum-ter}:
\begin{multline}\label{est-third-part-der-of-the-metr}
\left|\left(\prod_{q\in K}\partial_{T_q;X}\right) \chi^{(p')}_0(C'^2r^{-2}\tilde{g}_X(X-Y))\right|\leq C'_3 (1+g_X(X-Y))^{|K|}\\
\cdot\left(\prod_{q\in K}g_X(T_q)^{1/2}\right)\sup_{m'\leq |K|} |\chi^{(p'+m')}_0(C'^2r^{-2}\tilde{g}_X(X-Y))|.
\end{multline}
Combining \eqref{est-first-part-der-of-the-metr}, \eqref{est-second-part-der-of-the-metr} and \eqref{est-third-part-der-of-the-metr} together with the fact that $\chi_0(C'^2r^{-2}\tilde{g}_X(X-Y))$ is identically equal to $0$ when $g_X(X-Y)>r^2$, we infer that \eqref{bound-for-the-der-of-the-metr-sum-ter} is bounded by
$$
C''\left(\prod_{j=1}^{k'}g_X(T_j)^{1/2}\right) \left(\prod_{j=1}^{l'}g_X(S_j)^{1/2}\right) \sup_{m'\leq k'+l'} |\chi^{(m')}_0(C'^2r^{-2}\tilde{g}_X(X-Y))|,
$$
which implies the validity of the $(d)$ with $\tilde{\varphi}_X$ in place of $\varphi_X$. As we commented above, this implies the validity of $(d)$ for $\varphi_X$.


\end{document}